\documentclass[11pt,letterpaper]{amsart}

\usepackage{amsmath,amsthm,amssymb,graphics,color}
\usepackage{hyperref} 
\usepackage{xcolor}
\usepackage[english]{babel}
\usepackage{soul} 
\usepackage{tikz}

\usetikzlibrary{shapes,calc,arrows,through,intersections}
\usetikzlibrary{quotes,angles}
\usetikzlibrary{calc}
\usetikzlibrary{shapes}
\usetikzlibrary{shapes.geometric}
\tikzset{>=latex} 
\usepackage{pgfplots} 

\usepackage{subcaption}
\usepackage{paralist}
\usepackage{graphicx}
\usepackage{esvect}
\usepackage{float}
\usepackage{array}
\usepackage{esvect}
\usepackage{float}
\usepackage{array}
\usepackage{tensor}
\usepackage{amsthm}
\usepackage{amsfonts}
\usepackage{bbm}
\usepackage{dsfont}

\newcommand{\C}{\mathbb{C}}
\newcommand{\Q}{\mathbb{Q}}
\newcommand{\Z}{\mathbb{Z}}
\newcommand{\N}{\mathbb{N}}

\newcommand{\R}{\mathbb{R}}

\newcommand{\T}{\mathbb{T}}
\newcommand{\LL}{\mathcal{L}}

\newcommand{\am}{\text{am}}

\newcommand{\Id}{\text{Id}}

\newcommand{\supp}{\text{supp}}
\newcommand{\dist}{\text{dist}}

\renewcommand{\phi}{\varphi}
\renewcommand{\theta}{\vartheta}
\renewcommand{\epsilon}{\varepsilon}

\newtheorem{theo}{Theorem}[section]
\newtheorem{prop}[theo]{Proposition}

\newtheorem{lemm}[theo]{Lemma}

\theoremstyle{definition}
\newtheorem{def1}[theo]{Definition}
\theoremstyle{remark}
\newtheorem{rema}[theo]{Remark}

\newcommand{\nwc}{\newcommand}
\nwc{\Oph}{\operatorname{Op}_\hbar}
\nwc{\la}{\langle}
\nwc{\ra}{\rangle}

\nwc{\mf}{\mathbf} 
\nwc{\blds}{\boldsymbol} 
\nwc{\ml}{\mathcal} 

\renewcommand{\Im}{\operatorname{Im}}
\renewcommand{\Re}{\operatorname{Re}}

\newcommand{\e}{\varepsilon}
\newcommand{\eps}{\varepsilon}
\newcommand{\tr}{\text{Tr}\,}

\newcommand{\wt}{\widetilde}
\newcommand{\wh}{\widehat}
\newcommand{\sgn}{\text{sgn  \,}}

\renewcommand{\d}{\partial}

\renewcommand{\phi}{\varphi}

\newcommand{\lsp}{\text{LSP}}

\newtheorem{cor}[theo]{{\sc Corollary}}

\newenvironment{rem}{\medskip\noindent{\it Remark:\/} }{\medskip}

\newcommand{\pink}[1]{{\color{black}{#1}}}

\newcommand{\red}[1]{{\color{black}{#1}}}
\newcommand{\brown}[1]{{\color{black}{#1}}}
\newcommand{\black}[1]{\color{black}}

\newcommand{\dt}{\delta}

\newcommand{\ka}{\kappa}

\newcommand{\len}{\text{length}}
\newcommand{\E}{\mathcal{E}}

\newcommand{\Ee}{\mathcal{E}_e}
\newcommand{\Fm}{\mathcal{F}_m}

\allowdisplaybreaks

\title{Silent Orbits and Cancellations in the Wave Trace}

\author{Illya Koval}
\address{Institute of Science and Technology Austria, Klosterneuburg, Lower Austria}
\email{Illya.Koval@ist.ac.at}

\author{Amir Vig}
\address{Department of Mathematics, University of Toronto, Toronto, ON M5S 2E4, Canada}
\email{amir.vig@utoronto.ca}

\date{\today}


\begin{document}
	
	\begin{abstract}
        This paper shows that the wave trace of a bounded and strictly convex planar domain may be arbitrarily smooth in a neighborhood of some point in the length spectrum. In other words, the Poisson relation, which asserts that the singular support of the wave trace is contained in the closure of $\pm$ the length spectrum, can almost be made into a strict inclusion. To do so, we construct large families of domains for which there exist multiple periodic billiard orbits having the same length but different Maslov indices. Using the microlocal Balian-Bloch-Zelditch parametrix for wave invariants developed in our previous paper, we solve a large system of equations for the boundary curvature jets, which leads to the required cancellations. We call such periodic orbits \textit{silent}, since they are undetectable from the ostensibly audible wave trace. Such cancellations show that there are potential limitations in using the wave trace for inverse spectral problems and, more fundamentally, that the Laplace spectrum and the length spectrum are inherently different mathematical objects, at least insofar as the wave trace is concerned.
	\end{abstract}
	\maketitle
	\section{Main results}
	In basic Fourier analysis, the Poisson summation formula on $\R$ can be written as
	\begin{align}\label{PSF}
		1 + 2 \sum_{k \in \N^*} \cos kt  = 2 \pi \sum_{j \in \Z} \dt_{2\pi j}(t),
	\end{align}
	in the sense of distributions. If one observes that $- \d_x^2 \cos (k x) = k^2 \cos (k x)$, the left hand side of \red{\eqref{PSF}} can be interpreted as the ``wave trace,'' $\tr \cos ( t \sqrt{- \d_x^2})$, which is a spectral quantity. The right hand side of \red{\eqref{PSF}} is a sum of distributions with singular support at multiples of $2\pi$, which are \textit{exactly the periods of closed geodesics} on the circle $\R / 2\pi \Z$. These, by contrast, are geometric in nature. In particular, the Poisson summation formula on the circle implies the \textit{Poisson relation}: the singular support of the wave trace is contained in $\pm$ the lengths of closed geodesics. Now denote by $\Omega$ a bounded, smooth and strictly \brown{(strongly)} convex planar domain and consider the eigenvalue problem for the Laplacian with Dirichlet boundary conditions:
	\begin{align*}
		\begin{cases}
			-\Delta u_j = \lambda_j^2 u_j,\\
			u_j \big|_{\d \Omega} = 0.
		\end{cases}
	\end{align*}
	The even wave trace is defined by
	\begin{equation*}
		\red{w_{\Omega}}(t) := \tr \cos t \sqrt{-\Delta} = \sum_{j = 1}^\infty \cos(t \lambda_j),
	\end{equation*}
	again to be interpreted in the sense of distributions. The Poisson relation, in this setting due to \red{Andersson} and Melrose (\cite{AndersonMelrose}), connects the singular support of $\red{w_{\Omega}(t)}$ to the length spectrum of $\Omega$:
	\begin{equation}
		\text{SingSupp }\red{w_{\Omega}(t)} \subset \left\{0\right\}\cup \pm \overline{\text{LSP}(\Omega)},
		\label{eq13}
	\end{equation}
	where $\lsp(\Omega)$ consists of the lengths of all periodic billiard trajectories in $\Omega$. As before, the beauty of \red{\eqref{eq13}} is that the left hand side is entirely spectral whereas the right hand side is geometric. However, in principle, the wave trace need not be singular at every point of \red{the} length spectrum, which leads us to our main result:
	
	\begin{theo}
		Denote by $\mathcal{E}_e$ an ellipse of eccentricity $e$. For any $0 \le e < 1$ and $N \in \mathbb{N}$, there exist arbitrarily $C^\infty$-small\footnote{See Remark \ref{rmk:Cinfinitymetric} for the definition of the $C^\infty$-metric.} deformations $\Omega$ of $\mathcal{E}_e$ such that for some $L$ in the length spectrum of $\Omega$, $w_\Omega(t)$ is locally $C^N$ near $L$.
		\label{mainth}
	\end{theo}

	
	\subsection{Balian-Bloch-Zelditch wave invariants}
    We now establish the trace formula which leads to Theorem \ref{mainth}. Let $L$ be isolated in $\lsp(\Omega)$ with finite multiplicity and assume that all corresponding periodic orbits are nondegenerate \red{(see Definition \ref{orbittypes})}. To study the smoothness of $\red{w_{\Omega}(t)}$ locally near a length $L$, choose $\wh{\rho} \in C_c^\infty(\R)$ to be identically equal to $1$ in a neighborhood of $L$ and satisfy $\supp(\wh{\rho}) \cap \lsp(\Omega) = \{L\}$. We then have an asymptotic expansion of the form
	\begin{align}
		\int_0^\infty \wh{\rho}(t) \red{w_{\Omega}}(t) e^{i\lambda t} dt \sim \sum_{\gamma: \text{length}(\gamma) = L} \mathcal{D}_\gamma \sum_{j = 0}^\infty B_{\gamma, j} \lambda^{-j},
		\label{eq14}
	\end{align}
	as $\Re \lambda \to \infty$, where the right hand side is a sum over all periodic orbits $\gamma$ of length $L$. We call $\mathcal{D}_{\gamma}$ the \textbf{symplectic prefactor} and $B_{\gamma, j}$ the \textbf{Balian-Bloch-Zelditch (BBZ) wave invariants} \brown{(see Theorem \ref{partone}).}
	\\
    \\
	As the BBZ expansion involves a sum over all periodic orbits having length $L$, their contributions to \eqref{eq14} may interfere destructively and cause cancellations. This is the approach we take in proving Theorem \ref{mainth}. Since the decay of a function is equivalent to smoothness of its Fourier transform, if the invariants $B_{\gamma, j}$ sum (over $\gamma$) to $0$ for each $j = 1, 2, \cdots, m$, then the series \eqref{eq14} will be $O(\red{\lambda}^{-m-1})$ and $\red{w_{\Omega}(t)}$ will be locally $C^{m-1}$. If they sum to zero for all $j \in \N$, the wave trace will be locally $C^{\infty}$. In our previous paper \cite{KKV}, we derived formulae for $\mathcal{D}_\gamma$ and $B_{\gamma, j}$ in terms of the length functional (see \eqref{lengthfunctional}) when $\gamma$ is a ``nearly degenerate'' periodic orbit. Note that in Theorem \ref{mainth}, $\Omega$ is a deformation of an ellipse, the integrability of which guarantees that periodic orbits in $\Omega$ will be nearly degenerate in the sense of \cite{KKV}. Deriving formulae for the Balian-Bloch-Zelditch invariants there used primarily microlocal methods, whereas here, we employ dynamical techniques to orchestrate cancellations amongst them.

	\subsection{Outline of the proof}
	\brown{Fix $m-1 \in \N$ to be the degree of desired smoothness of the wave trace.} We aim to show that for some $L$ which is isolated in the length spectrum of a domain $\Omega$, we have
    \begin{equation}
		\sum_{\gamma: \text{length}(\gamma) = L} \mathcal{D}_\gamma B_{\gamma, j} = 0, \; \; \; \forall j = 0, 1,\ldots, m.
		\label{eq15}
	\end{equation}
    The symplectic prefactor $\mathcal{D}_\gamma$ in \eqref{eq14} has an oscillatory part $\exp i\lambda L$ together with an additional complex phase, $\exp i \pi m_\gamma /4$, where $m_\gamma$ is the \textbf{Maslov index} of $\gamma$ (see Theorem \ref{partone}). To create cancellations, we begin with an ellipse and deform it in such a way that the Maslov indices of two families of periodic orbits in our new domain differ by $4 \mod 8$. This way, the complex exponential for one family is the negative of \red{the other}. At the level of geometry, this can be arranged by ensuring that each family of periodic orbits has distinct rotation numbers in the new domain and has varying ellipticity/hyperbolicity. Since the invariants $B_{\gamma, j}$ are in fact complex-valued, we actually need $4$ families, corresponding to the phases $\pm 1$ and $\pm i$. The additional phases are created by separating out elliptic and hyperbolic orbits. The technicalities of how we can select sufficiently many orbits of a fixed length $L$, prescribe their ellipticity/hyperbolicity \red{(see Definition \ref{orbittypes})} together with opposite Maslov factors, all while ensuring that no additional periodic orbits of length $L$ emerge in the deformation, is outlined below.
    \\
    \\
	\noindent \textbf{Step 1.} Start with an ellipse $\Ee$ of eccentricity $e$ and choose two rational caustics, $\Gamma$ and $\Gamma'$, which have distinct rotation numbers $p/q$ and $p'/q'$, but are such that all corresponding tangent billiard orbits have the same length $L$. Not every ellipse has a pair of such caustics. In fact, the set of eccentricities for which this happens has Lebesgue measure zero (see Remark \ref{resonancezeromeasure}). However, using elliptic integrals and tools from Aubry-Mather theory, in particular Mather's $\beta$-function, we show that there is in fact a dense set of eccentricities $e \in [0,1)$ for which there exist at least two such caustics; see Theorem \ref{Eccentricities}. We call this situation a \textit{length spectral resonance}.
	\\
    \\
	\textbf{Step 2.} Select a family $\Fm$ consisting of $4m$ distinct periodic orbits in $\Ee$, all of which have length $L$, such that $2m$ of them are tangent to the caustic $\Gamma$ and the other $2m$ are tangent to $\Gamma'$. It is clear that each of these orbits is degenerate. We will smoothly deform $\Ee$ while preserving both convexity and \red{the $1$-jet of the boundary} at the $2m(q +q')$ reflection points of the $4m$ orbits in $\Fm$.
    \\
    \\
    \textbf{Step 3.}
    Let $\bar{\eps}, \bar{\dt} \in \R^{4m}$ be small perturbation parameters and denote by $n_x \in \R^2$ the outward pointing unit normal vector to $\Ee$ at the point $x \in \Ee$. We introduce a multiscale family of deformations of $\Ee$, which we denote by
    \begin{align*}
        \Omega_{\bar{\eps}, \bar{\dt}} := \left\{ x + \mu_{\bar{\eps}, \bar{\dt}}(x) n_x: x \in \Ee \right\},
    \end{align*}
    shortened to $\Omega$ moving forward, where $\mu_{\bar{\eps}, \bar{\dt}}$ is a smooth $8m$-parameter family of functions. We insist that the deformation preserve all orbits from $\Fm$ in the sense described above. $\Fm$ remains a family of periodic orbits in $\Omega$ which have length $L$ and they will be nearly degenerate in the sense of \cite{KKV}. We can deform in such a way that of the $2m$ type $(p,q)$ orbits which are tangent to $\Gamma$, $m$ of them become elliptic and $m$ of them become hyperbolic. Similarly, we can deform so that $m$ of the orbits tangent to $\Gamma'$ become elliptic while $m$ become hyperbolic. Suppose we can arrange for such perturbations to destroy all other periodic orbits of length $L$; this is in fact the main difficulty in Section \ref{Destroying other orbits}. This combination of ellipticity and hyperbolicity will then generate all $4$ different Maslov factors mentioned above.
    \\
    \\
	\textbf{Step 4.} Let $\ell = |\d \Omega|$ and denote by $ds$ the arclength measure along $\d \Omega$. If $x: \R/\ell \Z:= \T_\ell \to \R^2$ is an arclength parametrization of $\d \Omega$, we denote the parametrization of $\d \Omega^q$ by $x(S) \in \R^{2q}$ with $S = (s_1, \cdots, s_q) \in \T_\ell^q$ and $x_i(S)$ being the $(x_i^1, x_i^2)$ coordinates, $i \in \Z/q\Z$. The \textit{length functional} is defined by
	\begin{align*}
		\brown{\LL(S) = \sum_{i \in \Z/q \Z} |x_{i+1}(S) - x_i(S)|.}
	\end{align*}
	The BBZ wave invariants $B_{\gamma, j}$ are calculated in \red{the proof of Theorem 1 of} \cite{KKV} by applying the method of stationary phase to an oscillatory integral representation of the left hand side of equation \red{\eqref{eq14}}\red{; see Section 6 in \cite{KKV}}. This oscillatory integral has the form
	\begin{align*}
		\int_0^\infty e^{ikt} \wh{\rho}(t) \red{w_{\Omega}(t)} dt = \int_{\T_\ell^q} e^{i k \LL(S)} a_0(S; k) dS,
	\end{align*}
	where $a_0$ is an amplitude depending on the lengths of links $\overline{x_i(S)x_{i+1}(S)}$ for any $q$-link configuration in $\d \Omega^q$. In order to match the \red{$B_{\gamma,j}$} coefficients in Theorem 1 of \cite{KKV} (Theorem \ref{partone} below), we must prescribe the derivatives of the length functional at reflection points on the boundary. \brown{It suffices to prescribe the curvature of the boundary together with its derivatives, as they determine and are determined by the jet of the length functional. See Section \ref{Destroying other orbits} for the precise relationship between them; it turns out that only the curvature and its first derivative will play a leading role.} The necessary curvature jets at the reflection points will be obtained as the solution to a large system of equations associated to all $2m(q+q')$ points of reflection of the $4m$ orbits.
    \\
    \\
	\textbf{Step 5.} Upon deforming $\Ee$, we may create additional orbits which now share the same length $L$, where we are trying to make the wave trace smooth. We refer to these accidental length spectral coincidences as \textit{stray} orbits. As they interfere with our cancellation procedure, we need to choose a deformation which both satisfies the equations in Step 4 above and destroys all stray orbits. To do so, we introduce the notion of \textit{controllable} families of domains, which requires that the family $\Fm$ of periodic orbits is fixed, they obtain the correctly matched Maslov factors upon deformation, no stray orbits are created, they preserve convexity, etc. We explicitly construct controllable families by finding \textit{correctly selected} orbits (Definition \ref{def62}) and decomposing $\mu$ into ``local'' and ``nonlocal'' zones. The nonlocal zone is away from the reflection points of fixed orbits in $\Fm$ and the deformation there ensures that there can be no stray orbits in the sense described above. The local zone is near the reflection points, where the deformation is more delicate. \brown{Special attention is paid to the deformation near the minor axis.}
    \\
    \\
	\textbf{Step 6.} Having established the existence of controllable families, the system of equations \eqref{eq15} (for cancelling BBZ wave invariants) takes the form
	\begin{equation}
		\begin{cases}
			\sum_{u = 1}^m C_u \dt_u^{2j} \eps_u^{-3j- 1/2} = \sum_{u = 2m+1}^{3m} C_u \dt_u^{2j} \eps_u^{-3j- 1/2}+\overline{\mathcal{R}}_j, & 0 \leq j \leq m,\\
			\sum_{u = m+1}^{2m} C_u \dt_u^{2j} \eps_u^{-3j- 1/2} = \sum_{u = 3m+1}^{4m} C_u \dt_u^{2j} \eps_u^{-3j- 1/2} + \overline{\overline{\mathcal{R}}}_j, & 0 \leq j \leq m,
		\end{cases}
    \label{introequation}
	\end{equation}
	which has a leading order part in $\bar{\eps}, \bar{\dt}$, together with a remainder. The leading order part can be transformed via homogeneity in the deformation parameters to a system of $2m$ nonlinear equations for $2m$ free deformation parameters, which can be solved explicitly. We do this in Section \ref{Linearized}. The Jacobian is shown to be a Vandermonde matrix, whose invertibility guarantees (by the inverse function theorem) the existence of a nearby solution which accounts for the remainders in \eqref{introequation}. To make this quantitative, we employ a fixed point argument. This requires us to match all parameters simultaneously rather than by an iterative procedure.

	\subsection{Structure of the paper}
	In Section \ref{Background}, we review the existing literature on trace formulae and their historical role in inverse problems. In Section \ref{Billiard dynamics}, we introduce the necessary dynamical preliminaries on billiards, integrability of the ellipse, and elements of Aubry-Mather theory. We also briefly review the Balian-Bloch-Zelditch parametrix in \red{Section \ref{BBZSubsection}} and explain how it was used in \cite{KKV} to compute the BBZ wave invariants $\red{B_{\gamma, j}}$. Section \ref{Maslov} describes the way in which ellipticity and hyperbolicity of deformed periodic orbits are translated into ``conjugate'' Maslov indices which generate the complex phases $\pm 1$ and $\pm i$, a prerequisite for cancellations. In Section \ref{LengthCoincidences}, we show how integrability of the ellipse and corresponding analyticity of Mather's $\beta$-function can be used to find length spectral resonances for a dense set of eccentricities. Section \ref{Destroying other orbits} deals with stray orbits which may interfere with the cancellation procedure. It also prescribes a specific structure for the required deformation and is the most difficult section of the paper. Finally, in Section \ref{Solving the main equation}, we reduce the leading order terms to an algebraic system of equations which can be solved explicitly. We then employ a fixed point argument to find nearby solutions which account for lower order terms in the perturbation parameters.

	\section{Background}\label{Background}
		
		\subsection{Trace formulae}

		The relationship between spectral and geometric data through the Poisson relation has historically made the wave trace a powerful tool in studying Kac's famous problem: ``Can you hear the shape of a drum?'' Mathematically, this amounts to determining the shape of a domain (manifold, billiard table, etc.) from knowledge of its Laplace eigenvalues. The Poisson relation \eqref{eq13} is accompanied by a trace formula reminiscent of that for the circle \eqref{PSF}. The \textbf{Poisson summation formula}, also called the \textbf{wave trace expansion}, by contrast, is an asymptotic singularity expansion near each length $L$ in the length spectrum, the coefficients of which are spectral invariants of $\Omega$ containing local geometric data associated to closed orbits of length $L$. The Poisson relation is more qualitative in that it only relates the Laplace spectrum with the lengths of closed geodesics without providing much insight into geometry near the underlying orbits themselves.
        \\
        \\
		The Poisson summation formula \eqref{PSF} for the circle was extended to hyperbolic surfaces by Selberg (\cite{Sel56}) and to Riemannian manifolds by \brown{Colin de Verdi\`ere (\cite{CdVsllgp1}, \cite{CdVsllgp2}), Chazarain (\cite{Chazarain1974}), and} Duistermaat-Guillemin (\cite{DuGu75}). In the case of billiard tables, the Poisson relation was derived using the propagation of singularities by \red{Andersson} and Melrose \cite{AndersonMelrose}, with Guillemin and Melrose later proving a trace formula for nondegenerate orbits in \cite{GuMe79a}. \brown{In the physics community, there were two groups that independently developed trace formulae for the Laplacian. Gutzwiller derived a semiclassical analogue for the Schr\"odinger equation in \cite{Gutzwiller4} while Balian and Bloch derived formulas for domains with boundary (\cite{BB1}, \cite{BB2}, \cite{BB3}; see also \cite{BB4}).} For billiard tables, we have the following trace formula when all orbits of a given length are isolated and nondegenerate:
		\begin{theo}[\cite{GuMe79b}, \cite{PeSt17}] \label{aj}
			Assume $\gamma$ is a nondegenerate periodic billiard orbit in a bounded, strictly convex domain with smooth boundary and $\gamma$ has length $L$ which is simple. Then near $L$, the even wave trace has an asymptotic expansion of the form
			\begin{align*}
				\text{Tr} \cos t \sqrt{-\Delta} \sim \Re \left\{ a_{\gamma,0} (t - L + i0)^{-1} + \sum_{k = 0}^\infty a_{\gamma, k} (t - L + i0)^{k} \log(t - L + i0) \right\},
			\end{align*}
			where the coefficients \red{$a_{\gamma, k}$} are wave invariants associated to $\gamma$. The leading order term is given by
			\begin{align*}
				a_{\gamma, 0} =  \frac{e^{i \pi m_\gamma / 4}  L_\gamma^\#}{|\det (I - P_{\gamma})|^{1/2}},
			\end{align*}
			with $\red{L_{\gamma}^\#}$ being the primitive period of $\gamma$, $m_\gamma$ the Maslov index, and $P_\gamma$ the linearized Poincar\'e map.
		\end{theo}
		  \red{The coefficients $a_{\gamma, k}$} are algebraically and canonically \brown{equivalent} to the Balian-Bloch-Zelditch wave invariants $B_{\gamma, j}$ in Theorem \ref{partone} via basic identities in Fourier analysis. Theorem \ref{mainth} amounts to constructing domains for which $\sum_\gamma a_{\gamma, k} = 0$ when $0 \leq k \leq m$. \brown{If the length spectrum is multiple and there are finitely many periodic orbits of length $L$, the coefficients $a_{\gamma, k}$ in Theorem \ref{aj} are replaced by a sum over all $\gamma$ such that $\text{length}(\gamma) = L$. This is done by constructing microlocal parametrices for $\cos t \sqrt{- \Delta}$ near each individual periodic orbit and then summing each of their contributions to the wave trace.}

		\subsection{Dynamical inverse problems}
		One also considers the dynamical inverse problem: ``Can you determine a drum from the lengths of all closed geodesics?'' This problem makes sense for both Riemannian manifolds and billiard tables. In either case, one can study the more refined \textit{marked} length spectrum, which is a function associating to each homotopy class the lengths of periodic orbits in that class. For billiards, the homotopy class of a periodic orbit is specified by its winding number $p$ and its bounce number $q$ (with the rotation number $\omega$ being $p/q$). The \textbf{marked length spectrum} then returns for each $(p, q)$ with $0 < p < q$, the set of lengths of all corresponding periodic orbits. One also considers the \textit{maximal} marked length spectrum, which returns, for each $p/q$, the maximal length of a periodic orbit having rotation number $p/q$. For negatively curved manifolds, each free homotopy class $c$ has a unique closed geodesic $\gamma(c)$, and it is conjectured by Burns and Katok that the map $c \mapsto \text{length}(\gamma(c))$ uniquely determines the metric (\cite{BuKa85}).
        \\
        \\
        Perhaps the most famous dynamical inverse problem in this category is the Birkhoff conjecture, which asserts that ellipses are the only domains with integrable billiard map. By integrable, we mean that there exists a local foliation of the phase space by invariant curves (interpreted as Lagrangian tori in the KAM setting). A related conjecture in the boundaryless setting asks whether all integrable metrics on the $2$-torus need be Liouville metrics of the form $(f(x) + g(y))(dx^2 + dy^2).$

		\subsection{Combined invariants}
		The Poisson relation is what connects the two dynamical problems above. In \cite{MaMe82}, Marvizi and Melrose defined a sequence of marked length spectral invariants which they showed are also generically Laplace spectral invariants. The idea is that the lengths of billiard orbits which form polygons \red{(periodic orbits with $p = 1$ and $q \ge 3$)} have an asymptotic expansion as the number of edges \red{$q$} tends to infinity. The coefficients of this expansion are clearly marked length spectral invariants and under a ``noncoincidence condition,'' it is shown that the wave trace (a Laplace spectral quantity) is in fact singular at the minimal and maximal lengths of periodic orbits having a given rotation number. In other words, the Poisson relation is an \textit{equality} at those lengths. Hence, the singular support of the wave trace is asymptotically distributed the same as the marked length spectrum is. This shows that their dynamical invariants are also Laplace spectral invariants. The exact form of the wave invariants at these lengths was computed by the second author in \cite{Vig18} and \cite{Vig23}. They were also used in \cite{HeZe19} to hear the shape of ellipses of small eccentricity.
        \\
        \\
		Both wave invariants and dynamical invariants are important and useful in their own right. For example, in \cite{KaDSWe17} the authors use only the length spectrum to prove dynamical spectral rigidity of nearly circular $\Z_2$-symmetric billiard tables. Hezari and Zelditch then used this in \cite{HeZe19} to determine ellipses of small eccentricity by their Laplace spectra without needing to compute any wave invariants. By contrast, Zelditch used the combinatorial structure of wave invariants associated to nondegenerate bouncing ball orbits in generic $\Z_2$ symmetric analytic domains to determine the Taylor coefficients of a boundary defining function within that class (\cite{Zel09}, \cite{Zelditch1} and \cite{Zelditch3}). It is known that the wave invariants give rise to a so-called quantum Birkhoff normal form, which microlocally describes the behavior of solutions to the wave equation. It was shown separately by Guillemin (\cite{GuilleminWTI}, \cite{GuilleminWTITZ}) and Zelditch (\cite{ZelditchWTI}) that the quantum Birkhoff normal form (equivalently, the Balian-Bloch-Zelditch wave invariants) fully determine the classical Birkhoff normal form and usually contains even more geometric information. In general, using a combination of both Laplace and length spectral data seems to be most effective in tackling the inverse problem.
		\\
		\\
		This all starts to break down when inclusion in the Poisson relation \eqref{eq13} is strict. The ambiguity surrounding equality versus strict inclusion is already visible in Theorem \ref{aj}. If the length spectrum is not simple and several orbits all have the same length, complex phases arising from the Maslov factors $e^{i\pi m_\gamma/4}$ could cause a cancellation between all $a_{\gamma,k}$. If all wave invariants sum to zero, the Poisson relation would then be a strict inclusion. In fact, it was pointed out by Colin de Verdi\`ere that a priori, there is nothing obvious which prevents the wave trace from being smooth on all of $\R \backslash \{0\}$; see \cite{ZelditchSurvey2014}. It is known to be singular at $t = 0$, which was used by H\"ormander in \cite{HormanderSpectralFunctionofanEllipticOperator} to prove the sharp Weyl law.
		\\
		\\
		It is shown in \cite{PeSt17} that for a residual set of boundary curvatures, the length spectrum of billiard tables is simple, meaning that all primitive orbits have rationally independent lengths. In this case, there is no risk of cancellation amongst wave invariants. However, even with length spectral simplicity, the wave invariants contain only local data associated to closed orbits and it is not immediately clear how to both compute and combine sufficiently many of them to ascertain global information about the underlying geometry.
        \\
        \\
        The purpose of this paper is to demonstrate the existence of pathological domains for which the Poisson relation is almost strict. We refer to the corresponding periodic trajectories as \textit{silent orbits}, since their contribution to the \brown{Laplace spectrum is essentially negligible insofar as the ostensibly audible wave trace is concerned; that is, many of the wave invariants vanish, obscuring insight into local geometry near the associated orbits}. Theorem \ref{mainth} demonstrates that for a large family of smooth and strictly convex domains, the wave trace can be made locally smooth up to any finite order. This \brown{suggests} that
		\begin{enumerate}
			\item The Laplace spectrum and the length spectrum \brown{may be} fundamentally different objects, at least insofar as the wave trace is concerned.
			\item There are serious limitations to using the wave trace to study inverse spectral problems.
		\end{enumerate}
        \brown{We believe that it is possible to push smoothness of the wave trace all the way to $C^\infty$, but we leave this to a future paper.}
		
		\subsection{Recent progress}
		Despite these limitations, there has been much recent progress on both dynamical and Laplace inverse spectral problems. While the only explicit examples of Laplace spectrally determined planar domains are ellipses of small eccentricity (\cite{HeZe19}), there are nonexplicit examples in \cite{MaMe82}, \cite{Watanabe1} and \cite{Watanabe2}. Zelditch showed that generic $\Z_2$ symmetric analytic domains are spectrally determined within that class (\cite{Zel09}), which was later extended together with Hezari to generic centrally symmetric analytic domains (\cite{HezariZelditchCentrallySymmetric}). The interior problem was complemented by the inverse resonance problem for obstacle scattering in \cite{Zel0res}. In the boundaryless case, Zelditch also showed that certain analytic \brown{metrics} of revolution \brown{on the 2-sphere} are generically spectrally determined in \cite{Zelditch2}. $C^\infty$-compactness of Laplace isospectral sets was shown in \cite{MelroseIsospectralCompactness}, \cite{POS1}, and \cite{POS2}, while \brown{infinitesimal spectral rigidity under various symmetry assumptions} was shown in \cite{HeZe12}, \cite{Zelditch5}, \cite{Vig18}, \cite{Vig23}, \cite{HKLV25}, and \cite{HezariZelditchEigenfunctionAsymptotics}.
		\\
		\\
		In the dynamical setting, recent progress in this direction was achieved by Kaloshin and Sorrentino in \cite{KaSo16}, together with the works \cite{KaAvDS16}, \cite{kov21}, and \cite{KaDSWe17}. In the first three, it was shown that ellipses are isolated within the class of integrable domains. The latter two concern spectral rigidity of ellipses and nearly circular domains, respectively. In \cite{BialyMironov}, Birkhoff's conjecture is shown to hold amongst centrally symmetric domains which are foliated near the boundary by caustics having rotation numbers in $(0,1/4]$. $C^\infty$-compactness of marked length isospectral sets was demonstrated by the second author in \cite{Vig24}. In the boundaryless case, Guillarmou and Lefeuvre show in \cite{GuillarmouLefeuvre} that negatively curved metrics are locally determined by their marked length spectrum, \brown{which was later extended in joint work with Paternain (\cite{GLP25}).} See also \cite{KaHuSo18}, \cite{PopTop12}, \cite{Popov1994}, \cite{PopovTopalov}, \cite{GuKa}, and \cite{CdVLSP}.

	\section{Dynamical notations and preliminaries}\label{Billiard dynamics}
	
	We begin by introducing some preliminaries on billiards and then state the main result of \cite{KKV} adapted to ellipses. For a more extensive overview of the length spectral properties of convex billiards, we refer the reader to \cite{SiburgPrincipleofleastaction}, \cite{Sorr15} and \cite{Vig24}.
	
	\subsection{The length functional}
	 \brown{Denote by $\Omega$ a bounded, smooth, and strictly (strongly) convex planar domain and} let $\ell = | \d \Omega|$ \brown{be its perimeter. By strictly convex, we mean that the boundary curvature is strictly positive; this property is also called \textit{strongly} convex by some authors.} Denote by $x: \R /\ell \Z := \T_\ell \to \R^2$, $s \mapsto (x_1(s), x_2(s))$, an arclength parametrization of $\partial \Omega$. $ds$ is then the usual arclength measure along the boundary. Let $x \in \d \Omega$ and $\sigma \in T_x^*\d \Omega \simeq \R$ satisfy $|\sigma| \leq 1$. There is a unique unit vector \brown{$\wh \theta$} in the inward pointing circle bundle $S_x^+(\d \Omega)$ which projects orthogonally onto $\sigma$. \brown{Let $s'$ be such that $x(s')$} is the subsequent point of intersection with $\d \Omega$ of the line passing through $x$ in the direction of \brown{$\wh \theta$}. We set $\sigma'$ to be the parallel translate of \brown{$\wh \theta$}, orthogonally projected onto the cotangent space of $\d \Omega$ at $x(s')$. Denoting by $B^*(\d \Omega)$ the coball bundle of the boundary, that is the set of covectors having length $\leq 1$, the billiard map is the map $B: B^*(\d \Omega) \to B^*(\d \Omega)$ defined by $(s, \sigma) \mapsto (s', \sigma')$.
	\\
	\\
	We now tabulate several important properties of the billiard map. Define the function $h: \T_\ell^2 \to \R$ by
	\begin{align*}
		h(s,s') =|x(s) - x(s')|.
	\end{align*}
	It is easily seen that $h$ is a generating function for the billiard map, in the sense that
	\begin{align}\label{exact}
		\begin{cases}
			\d_s h = - \cos \phi,\\
			\d_{s'} h = \cos \phi',
		\end{cases}
	\end{align}
	and hence, $B(s, -\d_s h) = (s', \d_{s'} h)$. $B$ is \textbf{exact symplectic}, in the sense that $B^* \lambda = \lambda$, where $\lambda = \sigma ds$ is the canonical one-form on $B^*\d \Omega$. This is precisely the condition \red{\eqref{exact}}. Taking the exterior derivative, it follows that $B$ preserves the \red{symplectic} two-form $d \sigma \wedge ds$ and is hence a canonical transformation. A triple $(s, s', s'')$ corresponds to a two-link billiard orbit if and only if $\d_{s'} (h(s, s') + h(s', s'')) = 0$, which is equivalent to the equality of angles of incidence and reflection at $x(s')$.
	\\
	\\
	We now discuss sums of generating functions, which correspond to compositions of symplectic maps. As in the introduction, we will write $x(S)$ to denote a configuration of $q$ points in $\R^2$ with $S = (s_1, \cdots, s_q) \in \T_\ell^q$ and $x_i(S)$ being the $(x_1^i, x_2^i)$ coordinates, $i \in \Z/q\Z$. The \textbf{length functional} is defined in arclength coordinates by
	\begin{align}
		\LL(S) = \sum_{i = 1}^q |x_{i+1}(S) - x_i(S)|.
        \label{lengthfunctional}
	\end{align}
	As above, a $q$-tuple $S \in \T_\ell^q$ corresponds to a $q-1$-link billiard orbit if and only if $\d_{s_2, \cdots, s_{q-1}} (h(s_1, s_2) + \cdots + h(s_{q-1}, s_q)) = \vec{0}$. Since periodic boundary conditions are built into $\LL$ by definition, critical points of $\LL$ correspond precisely to periodic orbits of $B$. It is also clear that the corresponding critical values constitute the length spectrum. We will denote by $\gamma$ the concatenation of links in $\R^2$ which constitute a period $q$ billiard orbit. Similarly, we denote by $\d \gamma = (x_1, \cdots, x_q)$ the set of impact points on the boundary and \brown{write} $S_\gamma \in \T_\ell^q$ \brown{for} the corresponding arclength coordinates.

	\subsection{The Poincar\'e map}
	\label{sec.poincaremap}
	Let $(s_1 ,s_2) \in \T_\ell^2$ and consider the corresponding infinite billiard orbit $(x(s_1), x(s_2), x(s_3), \cdots)$, which uniquely defines a sequence $s_3, s_4, \cdots$ of arclength coordinates. For each $q$, define a function $P^q$ which associates to an initial condition $(s_1, s_2) \in \T_\ell^2$ the corresponding pair $(s_{q+1}, s_{q+2})$. If the impact points $x(s_i)$ form a $q$-periodic orbit, then $(s_1, s_2)$ is a fixed point of $P^q$. $P^q$ is called the \textit{Poincar\'e map} \brown{of the $q$-periodic orbit}. When we consider periodic orbits of a fixed period $q$, we will denote by $P_\gamma$ the $2\times 2$ symmetric matrix obtained by linearizing $P^q$ at any pair of reflection coordinates in $S_\gamma$. Although this matrix is only determined up to conjugacy (by cyclic permutation of initial conditions), its trace and eigenvalues are defined invariantly.
	
	\begin{def1}
		We say that an orbit $\gamma$ is
		\begin{enumerate}\renewcommand{\labelenumi}{{\arabic{enumi}.}}
			\item \textbf{elliptic} if the eigenvalues of $P_\gamma$ are of the form $e^{\pm i \alpha}$ for some $\alpha \in (0, \pi]$,
			\item \textbf{hyperbolic} if the eigenvalues of $P_\gamma$ are of the form $e^{\pm \alpha}$ for some $\alpha > 0$, and
			\item \textbf{nondegenerate} if $\det (I - P_\gamma) \neq 0$.
		\end{enumerate}
		The numbers $\alpha$ are called \textbf{Floquet angles}.
        \label{orbittypes}
	\end{def1}
	
	For our purposes, it will be more convenient to work with the length functional, which is related to the linearized Poincar\'e map by the following
	\begin{prop}[\cite{KozTresh89}, \brown{Theorem 3 -- pg. 67}] \label{DetPoincare}
		\begin{align}
			\det (\Id - P_\gamma) = (-1)^{q+1} \det \partial^2 \mathcal L (S_\gamma) \prod_{i = 1}^{q} b_i^{-1},
			\label{eq35}
		\end{align}
		where
		\begin{align*}
			b_i = \frac{\partial^2 |x(s_{i}) - x(s_{i+1})|}{\partial s_i \partial s_{i+1}} = \frac{\cos \vartheta_i \cos \vartheta_{i+1}}{|x(s_i) - x(s_{i+1})|},
			\label{eq36}
		\end{align*}
		$\partial^2 \mathcal{L}$ is the Hessian of the length functional, \red{and $\vartheta_i \in (-\pi/2, \pi/2)$ is the reflection angle at $s_i$ (with \brown{respect to} the \brown{interior} normal)}. 
	\end{prop}
	The Hessian $\d^2 \LL$ is a $q\times q$ \brown{cyclic-tridiagonal symmetric matrix; see equation \eqref{eq:tridiagonal} below. This algebraic structure} comes from the observation that $\mathcal{L}$ is a sum of generating functions, each of which depends only on pairs of successive reflection points. The off diagonal entries are given by $b_i$ in Proposition \ref{DetPoincare} and diagonal entries by
	\begin{equation}
		a_i = \partial_{s_i}^2 \mathcal{L} = 2 \cos \vartheta_i \left(\frac{\cos \vartheta_i}{|x(s_i) - x(s_{i-1})|}+ \frac{\cos \vartheta_i}{ |x(s_i) - x(s_{i+1})|} - 2\kappa_i\right),
		\label{eq37}
	\end{equation}
	$\ka_i$ being the curvature at $x(s_i) \in \d \gamma$ (see Proposition 3 on page 64 of \cite{KozTresh89}). \brown{The Hessian of $\LL$ then has the form
    \begin{align}\label{eq:tridiagonal}
        \d^2 \LL (S) =
        \begin{pmatrix}
            a_1 & b_1 & 0   & \cdots & 0   & b_{q} \\
            b_1 & a_2 & b_2 & \cdots & 0   & 0   \\
            0   & b_2 & a_3 & \ddots & \vdots & \vdots \\
            \vdots & \vdots & \ddots & \ddots & b_{q-2} & 0 \\
            0   & 0   & \cdots & b_{q-2} & a_{q-1} & b_{q-1} \\
            b_{q} & 0   & \cdots & 0   & b_{q-1} & a_{q}
        \end{pmatrix}.
    \end{align}
    }
	\brown{In particular, if an orbit is elliptic, then for some $\alpha \in (0, \pi]$,
    \begin{align*}
        \det(\Id - P_\gamma) = (1-e^{i\alpha}) (1 - e^{-i \alpha}) > 0.
    \end{align*}
    Therefore, the right hand side of \eqref{eq35} is also positive and hence, $\det \d^2 \mathcal L$ shares its sign with $(-1)^{q+1}$. Since $\d^2 \LL$ is symmetric and real-valued, it is diagonalizable over $\R$. If there are no zero eigenvalues, then there is an odd number of positive eigenvalues if and only if $\gamma$ is elliptic; the sign on the right hand side of equation \eqref{eq35} is determined by
    $$
    (-1)^{q - \#\{\text{positive eigenvalues}\} + q + 1},
    $$
    which is positive if and only if $\d^2 \LL$ has an odd number of positive eigenvalues. Conversely, if there are no zero eigenvalues, then having an even number of positive eigenvalues is equivalent to the condition that the orbit is hyperbolic.} If zero is an eigenvalue, then we call both the Hessian and the orbit degenerate. There are two levels of degeneracy; when zero is a simple eigenvalue of the Hessian, we say that the orbit has a \textbf{type 1 degeneracy}. This is equivalent to saying that $P_\gamma$ has \red{$1$ as an eigenvalue} but is not the identity. We call an orbit \textbf{type 2 degenerate} if $\d^2 \LL$ has rank $q-2$ and zero is an eigenvalue of multiplicity $2$. In this case, the linearized Poincar\'e map $P_\gamma$ is \textit{equal} to the identity. This type appears, for example, in \cite{Callis22}. The rank cannot drop below $q-2$, since the Hessian has positive off diagonal entries. Type $1$ degeneracy is much more common and was the subject of focus in \cite{KKV}. As we shall see later, most periodic orbits in ellipses have this type, allowing us to take advantage of the formulae developed \brown{in our previous work}.
	
	\subsection{Length spectra}
	We now detail some additional symplectic aspects of the billiard map. The \textbf{rotation number} of a $q$-periodic orbit is defined to be $\omega = \frac{p}{q}$, where $p$ is the winding number of $\gamma$; there exists a unique lift $\wt{B}$ of the map $B$ to the closure of the universal cover $\R \times [-1,1]$ of $B^*(\d \Omega)$ which satisfies
	\begin{itemize}
		\item $\wt{B}$ is smooth on \brown{$B^*(\R)$} and extends continuously up to the boundary $\R \times \{-1, 1\}$.
		\item $\wt{B}(\wt s + \ell, \sigma) = \wt{B}(\wt s, \sigma) + (\ell,0)$.
		\item $\wt{B}(\wt s,1) = (\wt s,1)$.
	\end{itemize}
	For $(s,\sigma) \in \R/\ell \Z \times [-1,1]$ belonging to a $q$-periodic orbit, we see that $\wt{B}^q(\wt s,\sigma) = (\wt s + p \ell, \sigma)$ for some $p \in \Z$, which we call the winding number of the orbit generated by $(s, \sigma)$. \brown{For orbits $(s_q)_{q \in \Z}$ which are not periodic, we can still define the rotation number by
    \begin{align*}
        \omega = \lim_{q \to \infty} \frac{\wt s_q}{\ell q}
    \end{align*}
    for any lift $(\wt s_q)_{q \in \Z}$ of the orbit to $\R$, assuming that this limit exists.
    }
\\
\\
	Recall that $h(s,s') = |x(s) - x(s')|$ is a generating function for $B$. The billiard map satisfies the so-called \textbf{monotone twist condition}:
	$$
	\frac{\d^2 h}{\d s \d s'} > 0,
	$$
	and the \textbf{twist interval} has endpoints
	\begin{align*}
		0 =& \frac{\pi_1 B(s, 1) - s}{\ell},\\
		1 =& \frac{\pi_1 B(s, -1) -s}{\ell}.
	\end{align*}
	The twist condition implies that vertical fibers are twisted clockwise upon iteration of the map $B$:
	\begin{align*}
		\frac{\partial}{\partial \sigma} s'(s, \sigma) < 0.
	\end{align*}
	A classical theorem of Birkhoff guarantees that for each rational $0< p/q \leq 1/2$, there exist at least two geometrically distinct periodic orbits of rotation number $p/q$. Their time reversals have rotation number $1/2 \leq (q-p)/q  <1$. \brown{For irrational rotation numbers, we have the following result from Aubry-Mather theory:
    \begin{theo}[\cite{SiburgPrincipleofleastaction}, Theorem 1.3.4]
    \label{thm: existence of minimal orbits}
        A monotone twist map possesses maximal orbits for every rotation number in its twist interval; for rational rotation numbers, there are always at least two periodic maximal orbits. Every maximal orbit lies on a Lipschitz graph over the $x$-axis. Moreover, if there exists an invariant circle, then every orbit on that circle is maximal.
    \end{theo}}
	
	A curve $\Gamma$ in $\Omega$ is called a \textbf{caustic} if the tangency of any link in a billiard orbit implies tangency of all past and future links. \brown{Closed convex caustics} can be projected onto $B^*(\d \Omega)$, where they form invariant curves $\Lambda_\Gamma$ for the billiard map. If all trajectories tangent to \red{a closed convex caustic} $\Gamma$ are periodic, then $\Gamma$ is called a rational caustic. It turns out that all orbits tangent to a rational caustic share a common rotation number $\omega$ and length $L$. An adaptation of the KAM theorem to billiards, \brown{which is} \red{due to} Lazutkin (\cite{Lazutkin}, \cite{Lazutkin5}), shows that there exists a large family of caustics whose rotation numbers form a Cantor set of positive measure.
	\begin{def1}\label{mbf}
		\textbf{Mather's $\beta$-function} is the function
		\begin{align*}
			\beta(\omega) = \lim_{N \to \infty} \frac{1}{2N} \sum_{i = -N}^{N-1} h(s_i, s_{i+1}),
		\end{align*}
		for any maximal orbit $(s_i)_{i \in \Z}$ \brown{of rotation number $\omega$, the existence of which is guaranteed by Theorem \ref{thm: existence of minimal orbits}. It is independent of the choice of maximal orbit and is uniquely determined by the domain $\Omega$.} The function $q \beta(p/q)$ returns the maximal length of all periodic orbits with given $p$ and $q$.
	\end{def1}
    \brown{Whenever it is important to specify the underlying domain $\Omega$, we will write $\beta_\Omega(\omega)$ to emphasize its dependence.} Note that we follow the unusual convention of taking our generating function to be positive, \brown{so the periodic orbits in Theorem \ref{thm: existence of minimal orbits} are \textit{maximal} rather than minimal}. It is known \brown{(\cite{SiburgPrincipleofleastaction}, Theorem 1.3.7)} that Mather's $\beta$-function is
	\begin{itemize}
		\item strictly concave (and hence continuous).
		\item symmetric about $\omega = 1/2$.
		\item $3$ times differentiable at the boundary, with $\beta'(0) = \ell$.
		\item differentiable on $[0,1]\backslash \Q$.
		\item differentiable at a rational $p/q$ if and only if there is an invariant curve in $B^*(\d \Omega)$ such that all billiard orbits which are tangent to it \brown{are periodic, having} rotation number $\omega = p/q$.
		\item differentiable at rotation numbers $\omega = p/q < 1/2$ \red{if (but not only if)} there exists a corresponding rational caustic $\Gamma_\omega$ and
		\begin{align*}
				\beta'(p/q) = |\Gamma_\omega|.
			\end{align*}
	\end{itemize}
	
	The Cantor set mentioned above turns out to be robust enough that differentiability of $\beta$ at these rotation numbers gives rise to a full Taylor expansion at zero (\cite{Carminati}):
    \begin{align*}
        \beta(\omega) \sim \sum_{k = 1}^\infty \beta_{k} \omega^k.
    \end{align*}
    The coefficients $\beta_k = \beta^{(k)}(0)/k!$ are called Mather's $\beta$-invariants. See \cite{Sorr15} for more on the relationship between differentiability of $\beta$ and integrability of the dynamics, together with a calculation of the first $5$ nontrivial invariants. 
	\\
    \\
    \begin{figure}
	\centering
		\includegraphics[width=12cm]{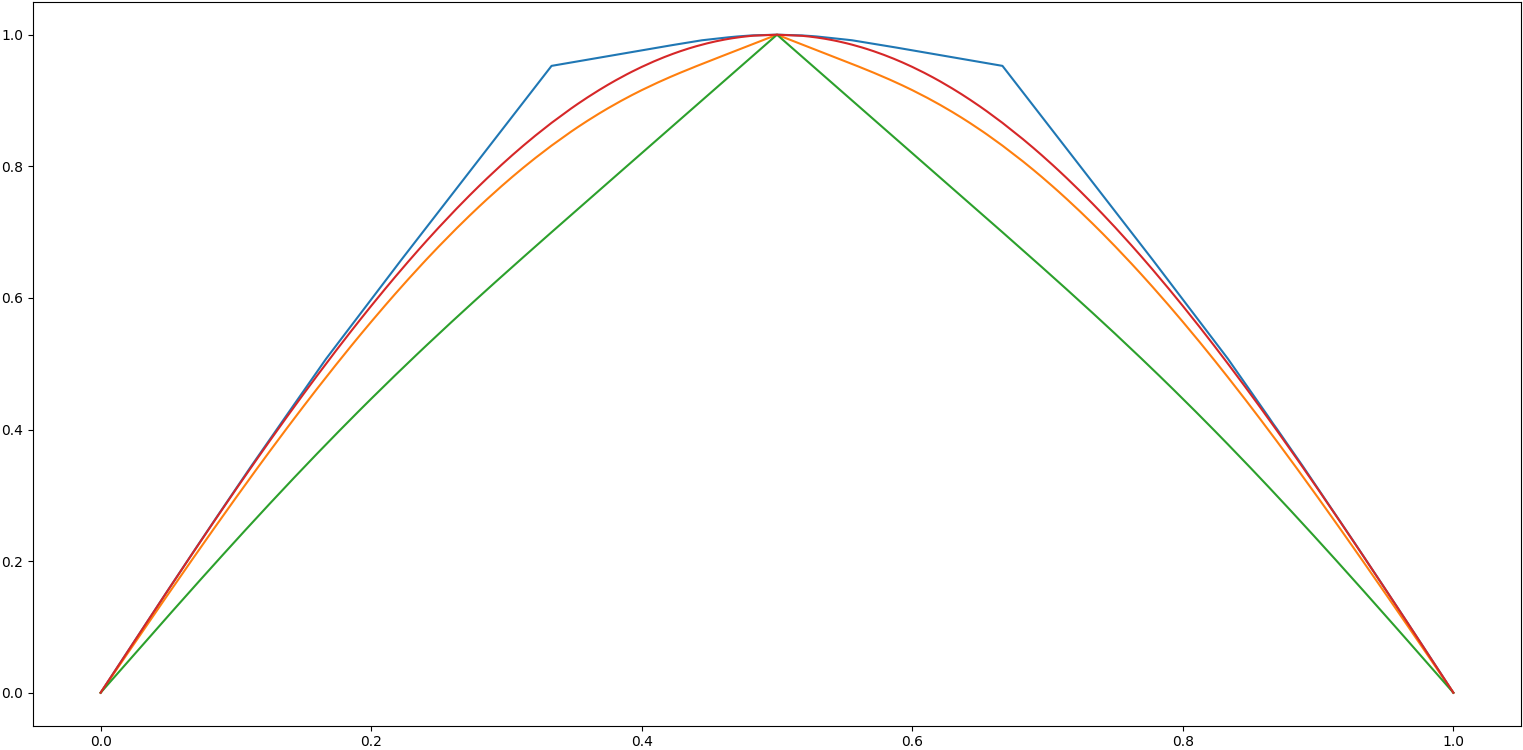}
			\caption{Mather's $\beta$-function for several convex domains. The red curve corresponds to a disc, the orange to an ellipse of eccentricity $e = 0.4$, the green to an ellipse with eccentricity $e = 0.9$, and the blue to a curve of constant width.}
		\label{fig:enter-label15}
	\end{figure}
	The \textbf{length spectrum} of $\Omega$ is the collection of all lengths of periodic billiard orbits together with integer multiples of the boundary length (corresponding to gliding orbits). For each $p$, the length spectrum has a finer filtration
	\begin{align*}
		\lsp(\Omega) = \bigcup_{\substack{p,q \in \N,\\ p/q \leq 1/2 } } \lsp_{p, q}(\Omega),
	\end{align*}
	where, denoting by $rt(\gamma) = (p, q)$ to be the winding number and the period of a periodic orbit,
	\begin{align*}
		\lsp_{p, q} (\Omega) = \bigcup_{\gamma\,\text{periodic} } \{ \text{length} (\gamma): rt(\gamma) = (p, q)\}.
	\end{align*}
	Marvizi and Melrose prove in \cite{MaMe82} that for each $p \geq 1$, the length spectrum breaks up into little clusters $[t_{p,q}, T_{p,q}]$, where
	\begin{align*}
		t_{p,q} = &\red{\min}_{rt(\gamma) = (p, q)} \len(\gamma),\\
		T_{p,q} =& \red{\max}_{rt(\gamma) = (p, q)} \len(\gamma),
	\end{align*}
	and \brown{$T_{p,q} - t_{p,q} = O_p(q^{-k})$ for all $k \in \N$}. For a residual set of boundaries, the length spectrum is simple \brown{(\cite{PeSt17}, Theorem 6.2.3)}. For an ellipse, $t_{p,q} = T_{p,q}$ for all $0 < p/q < 1/2$, corresponding to the one-parameter families of orbits tangent to confocal conic sections. In fact, this is true whenever there is a rational caustic of rotation number $p/q$. Marvizi and Melrose go further to show that
	\begin{align}
		T_{p,q} \sim p \ell + \sum_{k = 1}^\infty c_{p,k} q^{-2k},
        \label{eq38}
	\end{align}
	for a sequence of marked length spectral invariants $c_{p,k}$. \red{It} follows immediately that \red{for all $p \ge 1$}, the Taylor coefficients of Mather's $\beta$-function are given by
    \begin{align}\label{bettaylor}
    \begin{split}
        \beta_{\red{2k +1}} = & \frac{\red{1}}{p^{2k+1}} c_{p,k},\\
        \beta_{2k} = & 0.
    \end{split}
    \end{align}
    Generically, the wave trace is singular at $\{t_{p,q}, T_{p,q}\}$ (the noncoincidence condition in \cite{MaMe82}) and hence the $c_{p,k}$'s also constitute Laplace spectral invariants (\cite{MaMe82}). To ensure we have precise control over which orbits end up contributing to the wave trace in our main theorem, we need the following proposition, which states that orbits of a fixed perimeter cannot have arbitrarily large winding number.
	
		\begin{prop}[\red{Lemma 8.5 in }\cite{kov21}]\label{prop:largep}
		Given any ellipse $\Omega_0 = \Ee$ and $L > 0$, there exists a $p_0>0$, such that for all $C^\infty$-small deformations $\Omega$ and for every $p \ge p_0$ and $p/q \le 1/2$, the minimal length $t_{p, q}$ \red{in $\Omega$} is strictly greater than $L$.
	\end{prop}

    \brown{
    \begin{rema}
        Proposition \ref{prop:largep} is also true for general strictly convex billiard tables, although this is not needed for our purposes, so we omit the proof.
    \end{rema}
    
    }
    
    \brown{
    \begin{rema}
        \label{rmk:Cinfinitymetric}
        By $C^\infty$-small, we mean that the $d_\infty$-distance between boundary curvatures of $\Omega_0$ and the deformation $\Omega$ is small, where $d_\infty$ is the metric on $C^\infty(\R/\Z)$ defined by
    \begin{align}
        d_\infty(f,g) = \sum_{k = 0}^\infty 2^{-k} \frac{\| f- g \|_{C^k(\R/\Z)}}{1 + \| f- g\|_{C^k(\R/\Z)}}.
    \end{align}
    This metric topology coincides with the topology generated by the Fr\'echet seminorms $\|\cdot\|_{C^k(\R/\Z)}$. Note that strictly convex planar domains are uniquely characterized (up to rigid motions) by their curvature functions in $C^\infty(\R/\Z)$, which can be seen most easily from the curvature coordinate parametrization described in Proposition 2.7 of \cite{MaMe82}.
    \end{rema}
     }

	\subsection{Elliptical billiards}
	\label{subsecell}
	
	Ellipses have special dynamical properties; \brown{the Birkhoff-Portitsky Conjecture asserts that they are the only (even locally) integrable billiard tables; see \cite{birkhoff1927periodic, Poritsky};} integrability here means there is some open neighborhood in phase space which is foliated by caustics, \red{which are the level sets of some nonconstant function $\lambda$ (\cite{Poritsky})}. \brown{As our domain featuring cancellations in Theorem \ref{mainth} will arise from deformations of an ellipse, we need to understand how the billiard map changes in such a setting. The material in this section will be used in
    \begin{enumerate}
        \item the proof of Theorem \ref{Eccentricities}, where we find ellipses which have length spectral resonances (multiple rational caustics whose periodic tangent orbits have the same length) and
        \item even more extensively, Section \ref{subsec:proof of Theorem 6.6}, where we show that under suitable conditions, ``stray orbits'' (see Definition \ref{def:stray}) which could potentially interfere with our cancellation procedure can be eliminated; in particular, see Proposition \ref{Condition1}, equation \eqref{eq617}, and Lemma \ref{eq621}.
    \end{enumerate}
    The material below is rather technical and can be referred to on an as needed basis. The main concepts we introduce are:
    \begin{itemize}
        \item In Section \ref{subsubsec:rotation and libration numbers}, the distinction between rotation and libration numbers, both of which describe the winding number and bounce number of an orbit but are distinguished by the topology of their corresponding caustics.

        \item Action angle coordinates and elliptic integrals (Section \ref{subsubsec:aa coordinates and elliptic integrals}). This section recalls basic definitions from the theory of Jacobi elliptic functions and uses them to construct \textit{action angle coordinates} for the billiard map in an ellipse. These are the most important development of this section, since with respect to these coordinates, the billiard map is simplified to have the form
        \begin{align*}
            (I, \eta) \mapsto (I, \eta + \omega(I)),
        \end{align*}
        which is a simple translation.

        \item We then provide formulas for the lengths of periodic orbits as functions of their rotation or libration numbers in Section \ref{subsubsec: Lengths of periodic orbits in an ellipse}. They are also given in terms of elliptic integrals, which were introduced in Section \ref{subsubsec:aa coordinates and elliptic integrals}. The explicit formulas are helpful for observing some important properties such as analyticity in the eccentricity of the ellipse, which will be used in the proof of Theorem \ref{Eccentricities}.
    \end{itemize}
    }
    
    \brown{\subsubsection{Rotation and libration numbers.} \label{subsubsec:rotation and libration numbers}
    Let}
	\begin{align*}
		\mathcal{E}_e = \{(x,y) \in \R^2: x^2 + \frac{y^2}{1 - e^2} = 1\}
	\end{align*}
	denote an ellipse which is normalized to have semi-major axis length $1$ and eccentricity $e$. The phase space for the billiard map is the cylinder $\mathcal{E}_e \times [-1, 1].$ It turns out that all confocal conic sections are caustics for the billiard map on the ellipse. Projecting them onto invariant curves in $B^*(\mathcal{E}_e)$, we can divide the phase space into two regions (see Figure \ref{phase space}). Near the boundary, we have graphs of functions in the $x$ variable which correspond to confocal ellipses; we call these elliptic caustics, even though the corresponding orbits are degenerate. Further from the boundary are billiard orbits which pass between the focal points and are tangent to confocal hyperbolae; we call these hyperbolic caustics. The corresponding invariant circles in phase space (the ``eyes'') are homotopic to a point (the ``pupil'') on the cylinder. The separatrix between elliptic and hyperbolic invariant curves corresponds to homoclinic orbits which pass through the focal points. For the orbits in the ``eyes'', the rotation number, as we defined it before, is always equal to $1/2$. To obtain a more meaningful quantity, we introduce the following:
    \begin{def1}
        Let $\gamma$ be an orbit which is tangent to confocal hyperbolae in the ellipse $\mathcal{E}_e$ with rotation number $1/2$. We define the \textbf{libration number} of $\gamma$ to be its rotation number when considered as a diffeomorphism of its corresponding invariant circle in phase space.
    \end{def1}

    \brown{The rationale for treating orbits which are tangent to confocal hyperbolae separately is that they have nontrivial Morse index, since the invariant curves in $B^*\E_e$ on which they lie have points at which the projections to $\E_e \times \{\pm 1\}$ drop in rank, or equivalently, the tangent becomes vertical.} As mentioned before, all orbits in $\mathcal{E}_e$ are degenerate with the exception of the two bouncing ball orbits along the axes of symmetry. The period two orbit along the major axis consists of hyperbolic points on the separatrix, while the bouncing ball along the minor axis is elliptic (the ``pupils'' of the ``eyes'').
	\begin{figure}
		\includegraphics[scale =0.4]{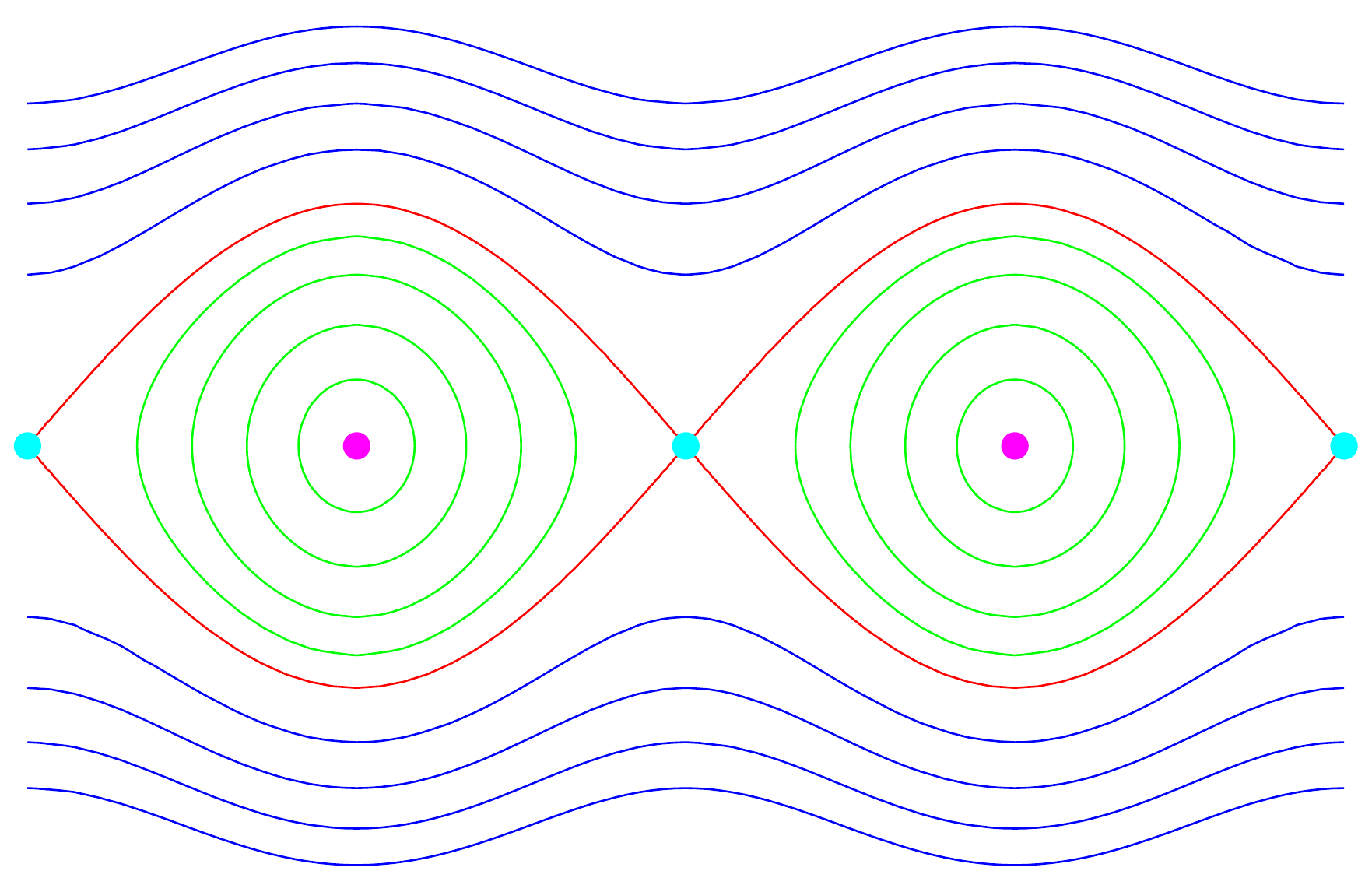}
		\caption{Phase space for billiards on the ellipse. The horizontal axis corresponds to the arclength coordinate along $\d \Omega$ while the vertical axis corresponds to the dual fiber variable $\sigma = \cos \phi$. The phase space is symmetric about the midline, corresponding to time reversal of orbits. The blue curves come from ``rotations'' tangent to confocal ellipses. The green curves are ``librations,'' corresponding to orbits which pass between the focal points and are tangent to confocal hyperbolae. These two families are separated by the red separatrix, with the three cyan colored dots corresponding to orbits which pass through the focal points, including the bouncing ball along the major axis. The magenta pupils at the center of the ``eyes'' correspond to the period $2$ bouncing ball orbit along the minor axis.}
		\label{phase space}
	\end{figure}

    \brown{\subsubsection{Action angle coordinates and elliptic integrals} \label{subsubsec:aa coordinates and elliptic integrals} We will now introduce action angle coordinates \brown{in the phase space of an ellipse} which conjugate the billiard map to a rigid rotation. To begin, consider elliptical polar coordinates $(\phi, \mu)$ on $\R^2$, which are given by the coordinate transformation}
	\begin{align*}
		x = -e \cosh \mu \sin \phi, \qquad y = e\sinh \mu \cos \phi,
	\end{align*}
	where $\{(\pm \pi/2, 0)\}$ are focal points for all corresponding coordinate level sets $\mu = \text{const.}$ (confocal ellipses) and $\phi = \text{const.}$ (confocal hyperbolae). In this case, $\mathcal{E}_e$ is parametrized by
	$$
	(-\sin \phi, \sqrt{1 - e^2} \cos \phi).
	$$
	Confocal conic sections are given by
	\begin{align}\label{confocal}
		\mathcal{E}_{e, \lambda} = \{(x,y) \in \R^2: \frac{x^2}{1 - \lambda^2} + \frac{y^2}{1 - e^2 - \lambda^2} = 1\}.
	\end{align}
	For $0 \leq \lambda^2< 1 - e^2$, the corresponding curves are confocal ellipses while for $1 - e^2 < \lambda^2  < 1$, they are confocal hyperbolae.
    \\
    \\
    We now define the \textbf{incomplete elliptic integrals}, of the first and second type respectively, by
	\begin{align*}
		F(\varphi, k) = \int_{0}^{\varphi} \frac{d \psi}{\sqrt{1 - k^2\sin^2 \psi}}, \; \; \; E(\varphi, k) = \int_{0}^{\varphi}\sqrt{1 - k^2\sin^2 \psi} d\psi,
	\end{align*}
	for $-1 < k < 1$. Their complete versions are defined by
	\begin{align*}
		K(k) = F(\pi/2, k), \; \; \; E(k) = E(\pi/2, k).
	\end{align*}
	The inverse of an incomplete elliptic integral of the first kind is called the \textbf{Jacobi amplitude} $\am(\eta, k)$:
	\begin{align*}
		F(\am(\eta, k), k) = \eta.
	\end{align*}
	For a confocal conic section $\mathcal{E}_{e, \lambda}$, we set the parameter $k$ in the elliptic integrals above to be the eccentricity of the caustic:
	\begin{align}
		k =& \frac{e}{\sqrt{1 - \lambda^2}}.
        \label{kelldefinition}
	\end{align}
    
    In particular, $0 < k < 1$ for elliptic conics and $k > 1$ for hyperbolic ones. \brown{The parameter $k$ is clearly in one-to-one correspondence with $\lambda$, which describes the eccentricity of a confocal conic section and is important for the following definition of action angle coordinates.}
	\begin{def1}\label{AAC}
		\brown{Local \textbf{(rotational) action angle coordinates}} for the ellipse $\mathcal E_e$, on an invariant curve of an elliptic caustic $\mathcal E_{e, \lambda}$, are given by $(\eta, I)$, where the angle variable is defined implicitly in terms of the elliptic coordinate $\phi$ by
		\begin{align}\label{eq: phi of eta}
	       \phi = \am \left(\frac{2K(k)}{\pi} \eta, k \right).
		\end{align}
        The action variable $I$ is defined to be the symplectic dual\footnote{\brown{
        By symplectic dual, we mean the following. Let $(\phi, p)$ be symplectic coordinates on $B^*\E_e$ and set $\phi = \psi(\eta)$ as in \eqref{eq: phi of eta}. Then $I := d \psi^* p$ is a coordinate on $B^*\E_e$ such that $I d\eta = p d \phi$.
    }} of \brown{$\eta$} and depends on the parameter $\lambda$ corresponding to a confocal conic section in \eqref{confocal}, to which a given orbit will be tangent. In these coordinates, the billiard map is given by
		\begin{align*}
			B(\eta, I ) = (\eta + \omega(I), I).
		\end{align*}
        We denote by $\Psi_{\text{ell}}$ the corresponding coordinate transformation, so that $\phi = \pi_1 \Psi_{\text{ell}}(\eta, I)$, with $\pi_1$ being the projection onto the first coordinate.
	\end{def1}

 \begin{def1}\label{AAChyp}
		\brown{Local \textbf{(librational) action angle coordinates}} for the ellipse $\mathcal E_e$, on an invariant curve of a hyperbolic caustic $\mathcal E_{e, \lambda}$, are given by $(\wt \eta, (\wt I))$, where the angle variable is again defined implicitly in terms of the elliptic coordinate $\phi$ by
		\begin{align*}
			\phi = \am \left(\frac{2\tilde{K}(k)}{\pi} \wt \eta , k\right) \mod \pi \Z,
		\end{align*}
        where $\wt{K}\left(k\right) = F\left(\arcsin k^{-1}, k\right)$. We again define $\wt I$ to be the symplectic dual coordinate to $\wt \eta$. In these coordinates, the billiard map is given by
		\begin{align*}
			B(\wt \eta, \wt I ) = (\wt \eta + \wt{\omega}(\wt I), \wt I),
		\end{align*}
		where $\wt{\omega}$ is the libration number. We continue to denote by $\Psi_{\text{hyp.}}$ the corresponding coordinate transformation, so that $\phi = \pi_1 \Psi_{\text{hyp.}}(\wt \eta, \wt I)$, with $\pi_1$ being the projection onto the first coordinate.
	\end{def1}
    Whenever it is clear from context, we will abbreviate the coordinate transformations $\Psi_{\text{ell.}/\text{hyp.}}$ to $\Psi$. 
    
    \begin{rema}
        The reason for computing residues modulo $\pi \Z$ as opposed to $2\pi \Z$ is the presence of two ``eyes'' in \red{Figure} \ref{phase space}, with {$\wt \eta$} defined separately in each. The billiard map $B$ maps one eye into the other. Note also that by an abuse of notation, we continue to write $F$ and $\am$ even when $k > 1$.
    \end{rema}

    \brown{\subsubsection{Lengths of periodic orbits in an ellipse} \label{subsubsec: Lengths of periodic orbits in an ellipse}
	We now parametrize the reflection points of a periodic orbit associated to an elliptic (resp. hyperbolic) caustic of eccentricity $k$, which will be used in the proof of Proposition \ref{Condition1}. For an elliptic caustic, the corresponding elliptic coordinates of the reflection points are given by
	\begin{equation}
		\varphi(s_i) = \am \left(4K(k)\left(\frac{\eta_0}{2\pi} + (i-1)\omega  \right), k \right),
		\label{eq315}
	\end{equation}
	where $\eta_0$ parametrizes the starting point. For a hyperbolic caustic, the situation is more complicated. We have:
	\begin{align}
		\varphi(s_i) = \pi i + \am \left(4\tilde{K}\left(k\right)\left(\frac{\wt \eta_0}{2\pi} + (i-1)\tilde{\omega}  \right), k\right),
		\label{eq316}
	\end{align} 
	where $\tilde{\omega}$ is the libration number defined above. Note that $\arcsin k^{-1}$ is the supremum of arguments of $F$ which give a convergent integral. Usage of the Jacobi amplitude $\am$ also needs to be justified. When $k > 1$, one can see that $\am\left(\wt \eta, k\right) = \varphi$ is well defined when $\wt \eta \in [-\tilde{K}(k), \tilde{K}(k)]$ and $\varphi \in [-\arcsin k^{-1}, \arcsin k^{-1} ]$. This allows us to analytically continue the Jacobi amplitude by
	\begin{align*}
		\am\left(\wt \eta + 2\tilde{K}\left(k\right), k\right) = -\am\left(\wt \eta, k\right), \qquad \wt \eta \in \mathbb{R}.
	\end{align*} 
	We see that in this case, $\varphi$ is restricted to $\left[-\arcsin k^{-1}, \arcsin k^{-1}\right]$, together with the shift of this interval by $\pi$. This is to be expected, as all the reflections of the billiard trajectory occur on one side of the hyperbolic caustic.
	\\
	\\
	There are other caustic parameters which are related to $k$ and are connected to billiard dynamics inside the ellipse. \brown{In particular, there is an amplitude $\phi_0$ which is} defined by
	\begin{align}
    \begin{split}
		\phi_0 =& \arcsin\left(\frac{\lambda}{\sqrt{1 - e^2}}\right), \qquad \text{(elliptic)} \label{phielldefinition}\\
		\varphi_0 =& \arcsin \left(\frac{\sqrt{1-e^2}}{\lambda}\right), \qquad \text{(hyperbolic)}
    \end{split}
	\end{align}
    for elliptic and hyperbolic caustics respectively. Remarkably, the rotation number of an elliptic caustic or the libration number of a hyperbolic caustic can be expressed in terms of these parameters:
	\begin{align*}
		\omega = \frac{F(\varphi_0, k)}{2K(k)}, \; \; \; \tilde{\omega} =  \frac{F(\varphi_0, k^{-1})}{2K(k^{-1})}.
	\end{align*}
	While elliptic caustics with rotation numbers from $0$ to $1/2$ exist for every ellipse, we note that hyperbolic caustics of a given libration number $\tilde{\omega}$ only exist when $e > \cos \pi \tilde{\omega}$. The libration number is strictly monotone in $\lambda$; on the phase portrait shown in Figure \ref{phase space}, the inner curves have larger $\lambda$ and smaller $\tilde{\omega}$.
	\\
	\\
	One is also able to derive formulae for the lengths of periodic orbits in an ellipse. The vertical and horizontal bouncing ball orbits have lengths which are simply four times the lengths of the semi-major and semi-minor axes. According to \cite{sieber} \brown{(equation (19) on page 4569)}, the lengths of orbits having period greater than two can be expressed in terms of elliptic integrals. For the ellipse $\mathcal{E}_e$, orbits tangent to an elliptic caustic $\mathcal{E}_{e, \lambda}$ have lengths given by
	\begin{equation}\label{lengths of hyperbolic orbits}
		L_{p, q} = 2q \sin \varphi_0 - \frac{2eq}{k}E(\varphi_0, k) + \frac{4ep}{k} \mathcal{E}(k).
	\end{equation}
	\brown{In particular, dividing by $q$ gives an expansion of Mather's $\beta$-function near $\omega = 0$, where it is smooth, corroborating the remarks made earlier about integrability. We also remark that the dependence of $\beta_{\E_e}$ (Mather's $\beta$-function for the ellipse $\E_e$) on the eccentricity $e$ is analytic, which will be used in the proof of Theorem \ref{Eccentricities}}.}
    \\
    \\
    Between the focal points, orbits which are tangent to a hyperbolic caustic have lengths given by
	\begin{equation}\label{lengths of elliptic orbits}
		L_{\tilde{p}, q} = 2q\sin \varphi_0 - 2eqE(\varphi_0, k^{-1}) + 4e\tilde{p} \mathcal{E}(k^{-1}).
	\end{equation}
    \brown{(See equation (26) on page 4570 of \cite{sieber}.)}
    
	\subsection{Balian-Bloch-Zelditch wave invariants for nearly elliptical domains} \label{BBZSubsection}
	In \cite{KKV}, the authors derived a formula for the regularized resolvent trace with coefficients called Balian-Bloch-Zelditch (BBZ) wave invariants, which are algebraically equivalent to the full quantum Birkhoff normal form (wave trace invariants). The main theorem in \cite{KKV} applies to nearly degenerate periodic orbits and gives leading order asymptotics (in the deformation parameter) of the BBZ invariants. We now specialize to the case of $\Omega$ being a deformation of an ellipse. Degenerate orbits which are tangent to a confocal conic section become nearly degenerate after the deformation. As a preliminary step, we recall some notation from \cite{KKV} and discuss complex phases associated to the Hessian of the length functional.

	\begin{prop}
		Let $\Omega_0 = \mathcal{E}_e$ be an ellipse and $\gamma$ a $q$-periodic orbit which is tangent to an elliptic caustic. Then the Hessian $\partial^2 \mathcal{L}(S_\gamma)$ has zero as a simple eigenvalue together with $q-1$ strictly negative ones. Moreover, all entries of the zero eigenvector share a common sign.
		\label{prop32}
	\end{prop} 
	
	\begin{proof}
		Denote the elliptic caustic by $\mathcal{E}_{e, \lambda}$. Since it is a \red{closed convex} caustic, all periodic orbits of type $(p, q)$ \red{which are tangent to the caustic share the same length. Note that there are no other periodic orbits of type $(p, q)$; otherwise at some coordinate $s$, there would be two possible $\sigma$-s which give a $(p, q)$ periodic orbit, contradicting the twist condition. Then, the orbits tangent to the caustic} realize the maximum of the length functional on $(p, q)$ configurations; \red{as the maximum occurs at a critical point and all critical points arise from periodic trajectories\footnote{There also exist critical points which two consecutive ``reflection'' points coincide, but they are not maximal.} which must be tangent to the caustic,} $\d \gamma$ is a local maximum for $\mathcal{L}$. This implies that $\partial^2 \mathcal{L}(S_\gamma)$ cannot have any positive eigenvalues. We also know that $\gamma$ is degenerate, so the Hessian has nontrivial kernel, corresponding to rotation along $\red{\mathcal{E}_{e, \lambda}}$. It remains to show that zero is a simple eigenvalue  or alternatively, that the Poincar\'e map $P_\gamma$ is not the identity.
\\
\\
\begin{figure}
\begin{center}
\includegraphics[scale = 1.5]{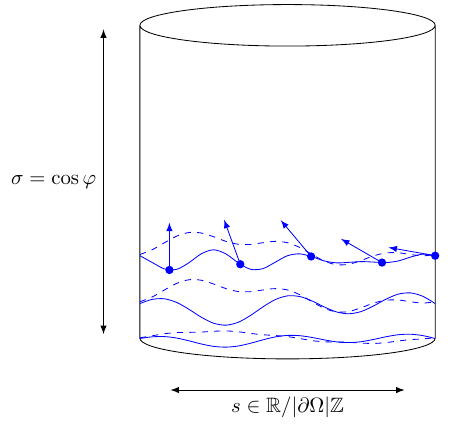}
\end{center}

\caption{The phase space $B^* \partial \Omega$ of the billiard map with invariant curves (in blue). The twist property ensures that a vertical vector in the tangent space is twisted monotonically along iterates of the billiard map.}
\label{Cylinder}
\end{figure}
We claim that the iterates of $\frac{\partial}{\partial \sigma}$ under $dB^{i}$ have a negative $\frac{\partial}{\partial s}$ component ($\d s' /\d \sigma < 0$) and are thus twisted to the left \red{(see Figure \ref{Cylinder})}. On the one hand, these iterates lie on one side of an invariant curve which contains $\gamma$, so the extent to which they can rotate counterclockwise \red{under iteration of the billiard map} is limited \red{by the tangent direction to the invariant curve}. On the other hand, the twist property prohibits them from rotating clockwise back towards the vertical direction. \red{Hence, after $q$ iterations, the previously vertical vector cannot return to its original position and the} Poincar\'e map \brown{will differ from the} identity. \brown{To see that all entries of the zero eigenvector share a common sign, note that rotation of the periodic orbit along \red{$\mathcal{E}_{e, \lambda}$} moves each vertex simultaneously in the same direction, either clockwise or counterclockwise. Hence, in arclength coordinates $S \in \T^q$, the eigenvector entries are either all positive or negative.}
	\end{proof}
	
	Lastly, we introduce the adjugate matrix $M := \text{adj}(\d^2 \LL)$ of the Hessian in $\Omega_0$, which was featured in \cite{KKV}. Its entries are complementary minors of $\partial^2 \mathcal{L}(\gamma)$, \brown{which we will evaluate at critical points of $\LL$, corresponding to} \red{periodic orbits $\gamma$} of \red{the} ellipse. As shown in \cite{KKV}, $M$ is a real symmetric matrix of rank $1$. Thus, the entries of $M$ are given by
	\begin{align}
		M_{\alpha, \beta} = \pm h_\alpha h_\beta,
        \label{happear}
	\end{align}
	for some $h \in \mathbb{R}^q$. As noted in \cite{KKV}, if $\partial^2 \mathcal{L}$ has a one dimensional kernel, then $h$ is an eigenvector along this \red{null} direction \brown{(i.e., in the kernel)}. In particular, when $\gamma$ is an orbit tangent to an elliptic caustic, all of the $h_i$ can be chosen to be positive by Proposition \ref{prop32}.
    \\
    \\
    Before stating the main theorem from \cite{KKV}, which provides an explicit formula for the leading order terms in the BBZ invariants $B_{\gamma, j}$, we briefly comment on the contribution of various reorderings of reflection points. In \cite{KKV}, the contribution of orbits to the trace formula was studied for a fixed ordering of reflection points. However, each orbit is invariant under both cyclic permutation of the reflection points as well as time reversal, so for the purpose of computing the full wave trace, we will sum over all such symmetries. By Theorem \ref{theo: symmetries}, all such reorderings produce identical contributions to the wave trace. We can now state the main result of \cite{KKV} adapted to ellipses:

	\begin{theo}[\red{adapted from Theorem 1 of \cite{KKV}}]\label{partone}
		Let $\Omega_0 = \mathcal{E}_e$ be an ellipse and $\gamma$ be a billiard orbit of period $q$, tangent to an elliptic caustic. Assume that $\left\{\Omega_\eps\right\}$, $\eps \in [0, \eps_0]$ is a smooth family of deformations which fixes the reflection points of $\gamma$ \red{together with the tangent direction at these points}. Assume further that the deformation makes $\gamma$ nondegenerate for $\eps>0$, with $|\det \partial^2 \mathcal{L}| \sim c_\gamma \eps$ for some $c_\gamma > 0$. Then, for $\eps > 0$, the $j$-th Balian-Bloch-\red{Zelditch} invariant associated to $\gamma$ from \eqref{eq14} has the form
		\begin{align}
			\begin{split}
			    B_{\gamma, j} = & \frac{|G| e^{-\frac{i\pi q}{4}}}{q} L_\gamma \left(\sum_{i_1, i_2, i_3 = 1}^{q}h_{i_1}h_{i_2} h_{i_3} \partial_{i_1,i_2,i_3}^3\mathcal{L}(\gamma) \right)^{2j} (c_{\gamma}\eps)^{-3j}\\
                &\times \prod_{l = 1}^{q} \frac{\cos \vartheta_l}{|x(s_l) - x(s_{l+1})|} \frac{(\pm i)^{j}}{w(j)} + \mathcal{R}_{\gamma,j},
			\end{split}
		\label{eq325}
		\end{align}
		where 
		\begin{itemize}
            \item $|G|$ is the number of reorderings of reflection points of $\gamma$. For example, if $\gamma$ is a polygon with no orthogonal segments, $G$ is isomorphic to the dihedral group on $q$ vertices and $|G| = 2q$.
        
			\item $w(j) > 0$ \red{is a positive combinatorial constant.}
			\item The signs $\pm$ are given by $+$ if in the deformed domain $\gamma$ becomes elliptic and $-$ if it becomes hyperbolic.
			\item $\mathcal R_{\gamma,j} = \mathcal{R}(\gamma, \eps, j, \kappa, \kappa^{(1)}, \ldots, \kappa^{(2j)})$, with $\kappa^{(n)}$ being the $n$-th derivative of \red{the} curvature of $\Omega$ at the reflection points, is a remainder which satisfies:
			\begin{itemize}
				\item $\eps^{3j-1}\mathcal{R}_{\gamma, j}$ is smooth in all parameters down to $\eps = 0$. Particularly, it is bounded from above whenever the deformation parameters remain bounded and is locally uniformly continuous in them.
				\item $\mathcal R_{0, \gamma} = 0$.
			\end{itemize}
		\end{itemize}
		The prefactor $\mathcal D_\gamma$ is given by
		\begin{equation}
        \label{eq: symplectic prefactor}
			\mathcal{D}_\gamma(k) =  \frac{e^{i k L_\gamma} e^{\frac{i\pi \sgn \d^2 \LL(S_\gamma)}{4}}}{\sqrt{|\det \d^2 \LL(S_\gamma)|}}.
		\end{equation}
	\end{theo}

    Theorem \ref{partone} gives a formula for the BBZ wave trace invariants associated to a single periodic orbit $\gamma$ of length $L$. It is essentially an expansion of the regularized resolvent trace (\cite{Zel09}). For a general convex domain, if $L$ is isolated in the length spectrum and all orbits of length $L$ are nondegenerate, the resolvent trace is a sum of contributions of all such orbits. As it does not appear explicitly elsewhere in the literature for billiards, we now show that all symmetries of a periodic orbit generate the same wave trace invariants.
    
	\begin{theo}\label{theo: symmetries}
        Let $\Omega$ be any smooth, strictly convex, planar domain and denote by $\gamma$ a nondegenerate $q$-periodic orbit. \pink{Let $G \subset D_{2q}$ denote the symmetry group of $\gamma$, where $D_{2q}$ is the dihedral group on $q$ vertices.} For any $g \in G$, define $g \gamma$ to be the $q$-periodic orbit obtained by applying $g$ to the points of reflection $\d \gamma$. Then $B_{g \gamma, j} = B_{\gamma, j}$ for all $j \in \N$.
	\end{theo}
	\begin{proof}
		\red{We remind the reader that $j$ here denotes the order of the Balian-Bloch-\red{Zelditch} invariant in \eqref{eq14}}. Corollary 4.12 in \cite{KKV} shows that
		\begin{align}\label{expansion}
			\mathcal{D}_{\gamma}(\lambda)   \sum_{j = 0}^\infty B_{\gamma, j} \lambda^{-j} \sim^* \sum_{\substack{1 \leq M < \infty\\ \sigma: \Z_M \to \{0,1\}}} \frac{e^{ -M \pi i / 4}}{M} \int_{{\d \Omega}_S^{M}}  e^{i \lambda \LL(S)} a_0^\sigma(\lambda, S) dS,
		\end{align}
		where $a_0^\sigma \in S_\varrho^{\frac{M}{2}}(\d \Omega^M)$ is a semiclassical symbol which is compactly supported, locally near $\LL^{-1}(L) \subset \d \Omega^M$. The notation $\sim^*$ indicates, as in \cite{KKV}, that the above is not a true asymptotic expansion but rather, provides an effective algorithm for computing the invariants $\red{B_{\gamma, j}}$. The critical points of the phase function occur precisely at the arclength coordinates of $M$-fold billiard orbits of length $L$, including their time reversals and cyclic permutations. The time reversal of an orbit $(x_1, x_2, \cdots, x_q)$ is simply $(x_q, x_{q-1}, \cdots, x_1)$. The corresponding arclength coordinates are similarly reversed, which corresponds to a change of variables in which one simply permutes the coordinates. The symbol $a_0^\sigma$ is given in \cite{KKV} by
		\begin{align*}
			a_0^\sigma (k +i \tau, S) = \hat{\rho}(\LL(S)) \left(\LL(S)  A_\sigma (k + i \tau, S) - \frac{1}{i}  \frac{\d}{\d k} A_\sigma (k + i \tau, S)  \right),
		\end{align*}
		which belongs to the class $ S_\varrho^{\frac{M}{2}}(\d \Omega^M)$, where $A_\sigma ( k + i \tau, S)$ is
		\begin{align*}
			( k + i \tau)^{M/2}  \prod_{i = 1}^{M}  \chi_i(k + i \tau,S) \frac{\cos \theta_i}{ |x(s_i) - x(s_{i+1})| ^{1/2}}  a_1(( k + i \tau) |x(s_i) - x(s_{i+1})|).
		\end{align*}
		It is clear that both $\LL$ and $a_0^\sigma$ are invariant under the permutations $s_i \mapsto s_{i + j}$ and $s_i \mapsto s_{M-i}$. When expanding \red{\eqref{expansion}} via the method of stationary phase, the Jacobian has determinant of magnitude one and the measure $dS$ is unoriented, so permutation of the coordinates leaves the integral unchanged.
	\end{proof}
    
	In light of this, we have the following corollary:

\begin{cor}
    \pink{Assume that $\gamma$ is primitive and has no orthogonal angles of reflection. Then its symmetry group $G$ coincides with $D_{2q}$ and we multiply the leading order term of Theorem \ref{partone} by $2q$ when summing over all geometrically distinct orbits.}
\end{cor}

    \begin{rema}
        If an orbit has orthogonal reflections at some point, the incident and reflected rays coincide and the symmetry group is reduced. \pink{More precisely, if the first point of reflection of a periodic orbit has an orthogonal reflection, then the first and last links must coincide. Similarly, the second and second to last links must coincide and so on. As a consequence, we see that $\gamma = \gamma^{-1}$, the period of $\gamma$ must be even, and if $\gamma$ is primitive, then there exist exactly two distinct points with orthogonal angles of reflection. As such, if the orbit is primitive and has orthogonal angles of reflection, then its symmetry group is reduced from $D_{2q}$ to $\Z_q$, in which case the leading order term in Theorem \ref{partone} should instead be multiplied by $q$.}
    \end{rema}

    We also remark that the third derivative component in \eqref{eq325} is zero for ellipses, since there is a \red{closed convex} caustic.
	\begin{lemm}
		If $\gamma$ is an orbit around a \red{rational} caustic in $\Omega_0 = \Ee$ \red{(or, for that matter, in any other domain)}, we have
		\begin{equation}
			\d_{\deg}^3 \LL := \sum_{i_1, i_2, i_3 = 1}^{q}h_{i_1}h_{i_2} h_{i_3} \partial_{i_1,i_2,i_3}^3\mathcal{L} = 0.
			\label{eq327}
		\end{equation}
		\label{lema31}
	\end{lemm}
	\begin{proof}
		The expression \eqref{eq327} is the third derivative of $\mathcal L$ along the degenerate direction $(h_1, h_2, \ldots, h_{q})$. Let $S(\tau) = (s_1 + \tau h_1, \ldots, s_q + \tau h_q)$ be a configuration along this subspace. Since $h$ is tangent to the curve of caustic orbit configurations, there exists an orbit $\gamma_\tau$ along the caustic with $S_{\gamma_\tau}$ being $O(\tau^2)$ close to $S(\tau)$. We know that $\mathcal L_{\Omega_0}(S_{\gamma_\tau}) = L$ and that it is a critical point. Hence,
		\begin{align*}
			\mathcal L_{\Omega_0}(S(\tau)) = L + \partial \mathcal L_{\Omega_0}(S_{\gamma_\tau}) \cdot O(\tau^2) + O(\tau^4) = L + O(\tau^4).
		\end{align*}
		In particular, there is no $\tau^3$ term, so the third derivative is $0$.
	\end{proof}
	
	\section{Conjugate Maslov indices}\label{Maslov}
	In this section, we find conditions on the Maslov indices of periodic orbits which facilitate cancellations in the wave trace. It is clear that the contribution of just one orbit to the BBZ expansion would be insufficient to make the wave trace smooth at any given length. Indeed, the Poisson relation is an \textit{equality} whenever the length spectrum is simple, which is a residual property (in the sense of the Baire Category Theorem). Instead, multiple orbits of the same length are required to create a cancellation; complex phases arising from Maslov factors together with the corresponding magnitudes $|\red{B_{\gamma, j}}|$ of the BBZ wave invariants must align perfectly. Consider the possible phases of $\mathcal{D}_\gamma B_{\gamma, j}$ appearing in \eqref{eq15}. We will focus on the leading order term in the expansion of $\red{B_{\gamma, j}}$, since the error terms $\mathcal R_j$ are both extremely complicated and relatively small in the deformation parameter. Since the $\red{B_{\gamma, j}}$-s are complex-valued themselves, arranging that they sum to zero does not reduce to a one dimensional problem. The oscillatory part $e^{ikL_\gamma}$ in the symplectic prefactor depends on $k$, but appears in the contribution of every orbit, so it can be factored out and ignored while keeping track of the phases of individual $\red{B_{\gamma, j}}$-s.
	\\
	\\
	Without the oscillatory factor $e^{i k L_\gamma}$ and the error term $\mathcal{R}_{\gamma, j}$, $\mathcal{D}_\gamma B_{\gamma, j}$ shares its phase with
	\begin{equation}
		e^{\frac{i\pi \sgn \d^2 \LL(S_\gamma)}{4}} e^{-\frac{i\pi q}{4}}(\pm i)^j.
		\label{eq41}
	\end{equation}
	In the unperturbed ellipse $\Omega_0 = \Ee$, if $\gamma$ is tangent to an elliptic caustic and becomes hyperbolic after deformation $\Omega_\eps$, then the signature of the Hessian changes to $-q$ and \eqref{eq41} reduces to
	\begin{align*}
		i^{-q-j}.
	\end{align*}
	Alternatively, if the perturbed orbit becomes elliptic, \eqref{eq41} becomes
	\begin{align*}
		i^{1-q+j}.
	\end{align*}
	In total, $4$ phases are possible: $1$, $i$, $-1$ and $-i$. If $\mathcal{D}_\gamma (B_{\gamma, j} - \mathcal{R}_{\gamma,j})$ share the same phase for all $j$, so would their sum and there could be no chance for cancellations. One can try to resolve this issue by having $2$ different phases (say $1$ and $-1$), that will counteract one another. However, the error terms may not be real-valued. Hence, we need to choose orbits which realize all $4$ possible phases. 
	\\
	\\
	The leading order term of phase $+1$ will be approximately cancelled by that of phase $-1$ and the leading order term of phase $+i$ will similarly be approximately cancelled by that of phase $-i$. Clearly, each of these orbits will introduce additional error terms, but they are comparatively small and we already have orbits with both real and imaginary phases. The upshot is then that we obtain an open map from the parameter space of deformations to $\C$. We will show that the origin is close to its image by perturbing the domain and cancelling leading order terms. To account for the errors, we use a fixed point argument to show that the map from parameter space to $\C$ in fact covers a neighborhood of the origin.
		\begin{figure}
				\centering
				\includegraphics[width=9cm]{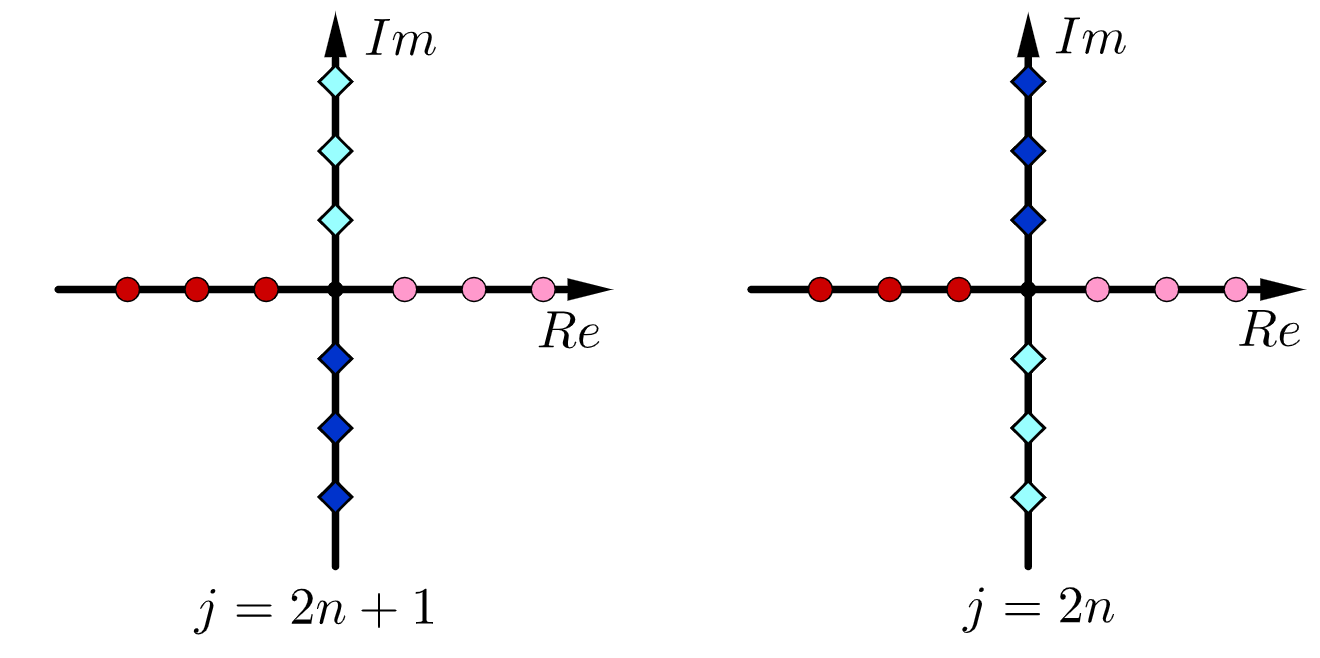}
				\caption{Relative positions of the main contributions on the complex plane for both parities of $j$. There are $4$ types of orbits represented. The lighter colored orbits have rotation number $p/q$ and the darker ones have rotation number $p'/q'$. Red and pink orbits (circles) become elliptic while blue and light blue ones (diamonds) become hyperbolic after the deformation.}
				\label{fig:fourtypes}
			\end{figure}
\\
\\
We now explain how each phase may be realized. Note that if all orbits share the same bounce number $q$, then it would be impossible to create coefficients with all $4$ phases; only $2$ would be possible. Thus, we need at least $2$ periods; call them $q$ and $q'$. We consider $4$ families of orbits of varying period and ellipticity/hyperbolicity, as described in the introduction. If we make $q \equiv 0 \mod 4$ and $q' \equiv 2 \mod 4$, we get the following phases for four families:
	\begin{align}
    \label{eq:phases mod 4}
		\begin{cases}
			i^{1 + j} & \text{period }q, \,\gamma \text{ elliptic},\\
			i^{- j} & \text{period } q, \, \gamma \text{ hyperbolic},\\
			i^{-1 + j} & \text{period } q', \, \gamma \text{ elliptic},\\
			i^{2 - j} & \text{period }q', \, \gamma \text{ hyperbolic}.
		\end{cases}
	\end{align} 
	For a fixed period, we refer to the signatures $\sgn \d^2 \LL(S_\gamma)$ of elliptic/hyperbolic orbits as \textbf{conjugate Maslov indices}. For any $j \ge 0$, these give all $4$ phases. However, not every ellipse has orbits of the same length but different rotation numbers; we will show that it is indeed possible to do so in the next section (see Theorem \ref{LengthCoincidences}).
    \\
    \\
    We briefly comment on the number of orbits in each family. Note that in the main term of \eqref{eq325}, the only orbit parameters which can be explicitly prescribed under a deformation are $c_\gamma \eps$ and
    
	\begin{equation}
		\partial_{\deg}^3 \mathcal L = \sum_{i_1, i_2, i_3 = 1}^{q}h_{i_1}h_{i_2} h_{i_3} \partial_{i_1,i_2,i_3}^3\mathcal{L}.
		\label{eq45}
	\end{equation} 
	Recall that the expression \eqref{eq45} is zero for $\Omega_0 = \Ee$ (cf. Lemma \ref{lema31}). As we will see later, \eqref{eq45} can be easily prescribed and deformed away from zero. All other parameters in \eqref{eq325} are independent of the deformation; for example $w(j)$ is a combinatorial constant and $\cos \vartheta_j$ depends only on the initial orbit. Hence, in contrast to the expansions in \cite{Zel09}, there are essentially only $2$ free parameters. To solve the system \eqref{eq15} of $m+1$ equations, we will choose $4$ corresponding families of $m$ orbits each.

	\section{Length spectral resonances in ellipses}\label{LengthCoincidences}

	To create arbitrarily high cancellations in the wave trace, we need to find domains with suitably many orbits of the same length. This section is devoted to finding ellipses for which one has two families of orbits tangent to different \red{elliptic} caustics, yet having the same length. \red{In theory, there may be an ellipse with three or more such families, but then we will use only two of them;} \pink{the existence of additional length spectral resonant families will not cause any interference in our cancellation procedure in Section \ref{Destroying other orbits}.}
		\begin{def1}
			We say that $L \in \lsp(\Omega)$ has multiplicity $k$ if there are $k$ distinct closed billiard orbits of length $L$.
		\end{def1}
		Oftentimes in the integrable setting, one has rational caustics which correspond to infinite multiplicity in the length spectrum. In that case, it is useful to generalize the notion of multiplicity:
		
		\begin{def1}
			We say that a domain $\Omega$ has a \textbf{length spectral resonance} if there exist \red{at least} two distinct rational caustics such that all orbits tangent to either caustic have the same length.
		\end{def1}

		We now find suitable ellipses with length spectral resonances.
		
		\begin{lemm}
			Denote by $\beta_{\E_e}(\omega) = \beta(\omega; e)$ Mather's $\beta$-function associated to an ellipse $\mathcal{E}_e$ of eccentricity $e$ and semi-major axis length $1$ (cf. Section \ref{Billiard dynamics}). Then, the ratio
			\begin{align*}
				\mathcal{R}(\omega_1, \omega_2; e) = \frac{\beta(\omega_1; e)}{\beta(\omega_2; e)}
			\end{align*}
			is analytic in $(\omega_1, \omega_2, e) \in (0,1/2)^2 \times (0,1)$ and nonconstant in $e$.
		\end{lemm}
        
		\begin{proof}
			Formulae \red{\eqref{lengths of hyperbolic orbits} and \eqref{lengths of elliptic orbits}} give the lengths of periodic orbits in terms of $p, q, k$ and the eccentricity $e$. \brown{Dividing by $q$}, we see that
			\begin{align}\label{eq:Mather's beta ellipse}
				\begin{split}
					\beta\left(\omega; e\right) =&  2 \sin \phi_0 - 2 e E\left(\phi_0, k\right) + \frac{4 e \omega}{k} \mathcal{E}(k) \qquad \text{elliptic caustics},\\
					\beta\left(\wt{\omega};e \right) =&  2 \sin \phi_0 - 2 e E\left(\phi_0, k^{-1}\right) + 4 e \wt{\omega} \mathcal{E}(k^{-1}) \qquad \text{hyperbolic caustics}.
				\end{split}
			\end{align}
			The definitions of $k = k(\lambda;e)$ and $\phi_0 = \phi_0(\lambda;e)$ clearly imply analyticity in both $\lambda$ and $e$ (see \eqref{kelldefinition} and \eqref{phielldefinition}). The dependence of $\lambda$ on $e$ and $\omega$ is also clearly analytic, from which the first claim follows.
			\\
			\\
			For a circle (eccentricity $e = 0$), we have
			\begin{align*}
				\beta\left(\frac{p}{q}; 0\right) = 2 \sin \left( \frac{p \pi}{q} \right),
			\end{align*}
			while for a line segment (eccentricity $e = 1$), we have
			\begin{align*}
				\lim_{e \to 1} \beta\left(\frac{p}{q} ;e  \right) = 2 \frac{p}{q}.
			\end{align*}
			Consequently, the ratio function for the circle is given by
			\begin{align*}
				\mathcal{R}(\omega_1, \omega_2; 0) =  \frac{\sin \left( \pi \omega_1 \right)}{\sin \left( \pi \omega_2 \right)},
			\end{align*}
			and for the segment by
			\begin{align*}
				\brown{\lim_{e \to 1} } \mathcal{R}(\omega_1, \omega_2; e) = \frac{\omega_1}{\omega_2},
			\end{align*}
			from which we conclude that the ratio function $\mathcal{R}(\omega_1, \omega_2; e)$ is real analytic and nonconstant in $e$.
		\end{proof}
  
		\begin{lemm}\label{LowerBound}
			There is a locally finite set $E \subset [0,1)$ such that for each $\eps > 0$ and $\dist(e, E) > \eps$, there exist constants $c(\eps), \dt(\eps) > 0$ with the property that whenever $0 < 2 \omega_1 \leq \omega_2 \leq \dt(\eps)$,
			\begin{align*}
				\left|\frac{\d}{\d e} \mathcal{R}(\omega_1, \omega_2;e)\right| \geq c(\eps) \left|\omega_2 \omega_1^3 - \omega_1 \omega_2^3 \right|.
			\end{align*}
		\end{lemm}
		
		\begin{proof}
			\brown{Fix $0 < \eps < 1$. To compute Mather's $\beta$-function for the ellipse $\E_e$, it is simpler to use the formulae \red{\eqref{betai}}, rather than differentiating \red{\eqref{lengths of hyperbolic orbits} and \eqref{lengths of elliptic orbits}}. Denote by $\beta_k(e)$ the $k^{\text{th}}$ $\beta$-invariant associated to the ellipse $\mathcal{E}_e$.} Note that the expansion of Mather's $\beta$-function
			\begin{align*}
				\beta(\omega_i;e) \sim \beta_1(e) \omega_i + \beta_3(e) \omega_i^3 + O_e(\omega_i^5), \quad i = 1, 2
			\end{align*}
            converges \brown{absolutely and} uniformly in $e \in [0,1 - \eps]$. From \cite{Sorr15}, we know that
			\begin{align}\label{betai}
				\begin{split}
					\beta_1(e) =& \ell = | \d \E_e |,\\
					\beta_3(e) =&- \frac{1}{4} \left( \int_0^\ell \ka^{2/3}(s; e) ds \right)^3,
				\end{split}
			\end{align}
        $\ka(s; e)$ being the curvature of $\mathcal{E}_e$ in arclength coordinates. The remainder is again uniform in $e \in [0,1-\eps]$. Noting that $|\beta(\omega_2; e)| \gtrsim_e \omega_2$ is nonzero for $\omega_2 > 0$, uniformly \brown{in $e \in [0, 1-\eps]$,} we have
			\begin{align}\label{ratio derivative}
				\frac{\d}{\d e} \mathcal{R}(\omega_1, \omega_2;e) = \frac{\left(\beta_1 \frac{\d \beta_3}{\d e} - \beta_3 \frac{\d \beta_1}{\d e}\right) \left(\omega_2 \omega_1^3 - \omega_1 \omega_2^3\right) + r(\omega_1, \omega_2;e)}{\beta(\omega_2;e)^2},
			\end{align}
			where $r = O_e(\omega_1^5 + \omega_2^5)$ is another remainder \brown{which is again uniform in $e \in [0, 1-\eps]$. We parametrize the ellipse $\mathcal{E}_e$ by $(-\sin \phi, (1- e^2)^{1/2}\cos \phi)$ and perform a change of variables (from arclength coordinate $s$ to the elliptical polar coordinate $\phi$) in \eqref{betai} so that}
            
			\begin{align*}
				\begin{split}
					\beta_1(e) = & \int_0^{2\pi} (1 - e^2 \sin^2 \phi)^{1/2} d\phi,\\
					\beta_3(e) = &- \frac{1-e^2}{4} \left(\int_0^{2\pi} \left( 1 - e^2 \sin^2 \phi \right)^{-1/2} d\phi \right)^3.
				\end{split}
			\end{align*}
			Differentiating, we find that
			\begin{align*}
				\frac{\d \beta_1}{\d e} = &- \int_0^{2\pi} \frac{e \sin^2 \phi}{(1 - e^2 \sin^2 \phi)^{1/2}} d \phi,\\
				\frac{\d \beta_3}{\d e} = &  - \frac{3}{4} (1- e^2) \left(\int_0^{2\pi} \left( 1 - e^2 \sin^2 \phi \right)^{-1/2} d\phi \right)^2 \int_0^{2\pi} \frac{e \sin^2 \phi}{(1 - e^2 \sin^2 \phi)^{3/2}} d\phi\\
				& + \frac{e}{2}\left(\int_0^{2\pi} \left( 1 - e^2 \sin^2 \phi \right)^{-1/2} d \phi \right)^3.
			\end{align*}
        
    For $e$ small, we have
			\begin{align*}
				\beta_1 \frac{\d \beta_3}{\d e} - \frac{\d \beta_1}{\d e} \beta_3 \sim - 8 \pi^4 e + O(e^2),
			\end{align*}
			which bounds the truncated Taylor series of $\frac{\d}{\d e} \mathcal{R}$ away from zero and shows that it is nonzero and nonconstant as a function of $e$. Hence, the zeros of its leading order coefficient, given by
			\begin{align*}
				E = \left\{ e : \beta_1 \frac{\d \beta_3}{\d e} - \frac{\d \beta_1}{\d e} \beta_3  = 0\right\},
			\end{align*}
			are locally finite and of finite multiplicity. Away from an $\eps$ neighborhood of $E \cup \{1\}$, there exists $c_1(\eps) > 0$ such that
            \begin{align*}
                \left|\beta_1 \frac{\d \beta_3}{\d e} - \frac{\d \beta_1}{\d e} \beta_3 \right| \geq c_1(\eps) > 0.
            \end{align*}
            Note that on the region $0 < 2 \omega_1 \leq \omega_2 \leq \dt(\eps)$, we have $|\omega_2 \omega_1^3 - \omega_1 \omega_2^3| \geq \omega_1^4$. Consequently, we may choose $\delta(\eps)$ so that the remainder $r(\omega_1, \omega_2; e)$ in \eqref{ratio derivative} is bounded by $\frac{1}{2} c_1(\eps) |\omega_2 \omega_1^3 - \omega_1 \omega_2^3|$, in which case the estimate $\beta(\omega_2;e) \leq C(e) \omega_2$, for some $C(e)$ which is uniform in $e \in [0,1-\eps]$, together with \red{\eqref{ratio derivative}} implies that
			\begin{align*}
				\left|\frac{\d}{\d e} \mathcal{R}(\omega_1, \omega_2;e) \right| \gtrsim c(\eps) \left| \omega_2 \omega_1^3 - \omega_1 \omega_2^3 \right|,
			\end{align*}
            where the dependence on $\eps$ is uniform in $e$ on the region $\dist (e, E) \geq \eps$.
		\end{proof}

        \brown{In the proof of Theorem \ref{Eccentricities} below, we will need a quantitative version of the implicit function theorem, which is quite simple in codimension $1$, so we include a proof of a $1$-dimensional inverse function theorem for the sake of completeness.
\begin{lemm}\label{lemma: calculus}
    Let $f \in C^\infty(\R)$ be any function of the form $f(t) = a_0 + a_1 (t- t_0) + r(t)$ where $a_1 \neq 0$, $|r(t)| \leq B_2 (t - t_0)^2$ and $|r'(t)| \leq B_1 |t - t_0|$ for all $t \in \R$, with $B_1 > B_2 > 0$. Then, $f$ is a local diffeomorphism on the interval $I = \{ t : |t - t_0| \leq |a_1|/B_1\}$ and the range of $f$, restricted to any subinterval of the form $\{t : |t - t_0| \leq \rho\}$, contains the closed interval with endpoints $a_0 \pm a_1 \rho \mp B_2 \rho^2$.
\end{lemm}
        \begin{proof}
            Clearly, $f$ is a diffeomorphism on any interval on which it is strictly monotone, or equivalently $|f'(t)| > 0$. Without loss of generality, assume $a_1 > 0$. Since $f'(t) = a_1 + r'(t)$, we have $f'(t) \neq 0$ whenever $|t - t_0| < |a_1|/B_1$. By the intermediate value theorem, the image of this interval is another interval with endpoints satisfying
            \begin{align*}
                f(t_0 +\rho ) =& a_0 + a_1 \rho + r(\rho) \geq a_0 + a_1 \rho - B_2 \rho^2,\\
                f(t_0 - \rho) \leq & a_0 - a_1 \rho + B_2 \rho^2,
            \end{align*}
            from which the lemma follows.
        \end{proof}
        }
        
		\begin{theo}\label{Eccentricities}
			There exists a dense set of eccentricities $e$ such that any ellipse of eccentricity $e$ has at least two rational caustics of distinct types $(p,q)$ and $(p',q')$ with $p,q$ and $p',q'$ coprime and all corresponding orbits having the same length. Moreover, for any fixed $m, n \in \Z_{>0}$, one can choose the residues of $q$ and $q'$ so that $q = q' + m \mod n$.
		\end{theo}
        \begin{proof}
			\brown{ Fix $e_0 \in (0,1) \backslash E$ and set $\eps = \dist(e_0, E)$. We first establish local monotonicity of the ratio function
            \begin{align*}
                \mathcal{R}(\omega_1, \omega_2, e)
            \end{align*}
            for a certain interval of eccentricities $e$, uniformly with respect to the rotation numbers $\omega_1, \omega_2$ in the region
            \begin{align}\label{eq: D epsilon region}
                D(\eps) = \left\{ (\omega_1, \omega_2) \in \left(0, 1/2\right)^2: \frac{1}{32}\dt(\eps) \leq \omega_1 \leq \frac{1}{2} \omega_2 \leq \frac{1}{2} \delta(\eps) \right\},
            \end{align}
            where $\dt(\eps)$ is as in Lemma \ref{LowerBound}. Note that since $\beta(\omega_2;e)$ is bounded away from zero on $D(\eps)$, the ratio function is $C^\infty$ with uniformly bounded derivatives on the compact region $D(\eps) \times [0, 1- \eps]_e$. Expanding to first order in powers of $e$, we have
            \begin{align*}
                \mathcal{R}(\omega_1, \omega_2, e) = \mathcal{R}(\omega_1, \omega_2, e_0) + (e - e_0) \frac{\d \mathcal{R}}{\d e} (\omega_1, \omega_2, e_0) + r(\omega_1, \omega_2, e), 
            \end{align*}
            where $|r(\omega_1, \omega_2, e)| \leq B_2({\omega_1, \omega_2}, \eps) |e - e_0|^2$, for some constant $B_2({\omega_1, \omega_2}, \eps) > 0$ which is uniform on $D(\eps) \times [0, 1-\eps]_e$. There exists another constant $B_1(\omega_1, \omega_2, \eps) > 0$ such that
            \begin{align*}
                \left| \frac{\d r}{\d e} (\omega_1, \omega_2, e) \right| \leq B_1(\omega_1, \omega_2, \eps) |e - e_0|,
            \end{align*}
            again uniformly on $D(\eps) \times [0, 1 - \eps]_e$.
            \\
            \\
            Now set
            \begin{align*}
                \bar B_1(\eps) &= \max_{\substack{0 \leq 2 \omega_1 \leq \omega_2 \leq \delta(\eps)}} B_1(\omega_1, \omega_2, \eps),\\
                \bar B_2(\eps) &= \max_{\substack{0 \leq 2 \omega_1 \leq \omega_2 \leq \delta(\eps)}} \left\{B_2(\omega_1, \omega_2, \eps), 2 \bar B_1(\eps) \right\}.
            \end{align*}
            Combining Lemmata \ref{LowerBound} and \ref{lemma: calculus}, we see that $\mathcal{R}(\omega_1, \omega_2, \cdot)$ is monotone on the interval
            \begin{align}
            \label{eq: monotonicity}
                e_0 -  \frac{c(\eps)}{\bar B_1(\eps)} \left|\omega_2 \omega_1^3 - \omega_1 \omega_2^3 \right| \leq e \leq e_0 +  \frac{c(\eps)}{\bar B_1(\eps)} \left|\omega_2 \omega_1^3 - \omega_1 \omega_2^3 \right|
            \end{align}
            and the image of its restriction to any subinterval of the form $e_0 - \rho \leq e \leq e_0 + \rho$ contains a symmetric interval with endpoints
            \begin{align}\label{eq:symmetric interval}
                \mathcal{R}(\omega_1, \omega_2, e_0) \pm \frac{c(\eps)}{\bar B_1(\eps)} \left|\omega_2 \omega_1^3 - \omega_1 \omega_2^3 \right| \rho \mp \bar B_2(\eps) \rho^2.
            \end{align}
            Note that $|\omega_1 \omega_2^3 - \omega_2 \omega_1^3|$ factors into $\omega_1 \omega_2 (\omega_1 + \omega_2) |\omega_1 - \omega_2|$, which is bounded below on $D(\eps)$:
            \begin{align}
            \label{omega epsilon lower bound}
                \left| \omega_1 \omega_2^3 - \omega_2 \omega_1^3\right| \geq 3 \times 2^{-19} \delta(\eps)^4 > 0,
            \end{align}
            so the interval with endpoints given by \eqref{eq:symmetric interval} is nonempty. 
            \\
            \\
            Now  fix any $\wh \omega_1 \in (0, \frac{1}{2}) \backslash \Q$ and choose $\wh \omega_2 \in (0, \frac{1}{2}) \backslash \Q$ different from $\wh \omega_1$ such that the ratio
            \begin{align*}
                \mathcal{R}(\wh \omega_1, \wh \omega_2;e_0) = \frac{\beta (\wh \omega_1; e_0)}{\beta(\wh \omega_2;e_0)} = \tau \in \R \backslash \Q
            \end{align*}
            is irrational, which is possible since the ratio is analytic (and nonconstant) as a function of $\omega_1, \omega_2$. In more detail, we can choose any open interval $I$ in the image of $\omega_2 \mapsto \mathcal{R}(\omega_1, \omega_2;e_0)$ and discard the measure zero sets $\Q\cap I$ together with the image of $(0, \frac{1}{2}) \cap \Q$ under $\mathcal{R}(\omega_1, \cdot; e_0)$. Note that by real analyticity in dimension $1$, the inverse image of a measure zero set is again of measure zero. This leaves a full measure set of irrational $\omega_2$ in $I$ so that the preceding conditions are satisfied. In particular, these conditions are satisfied by a full measure set of pairs $(\omega_1, \omega_2) \in (0,1/2)^2$. Hence, we can in fact choose $(\wh  \omega_1, \wh  \omega_2) \in D(\eps)$ to also satisfy the additional constraint
            \begin{align}
            \label{eq: wh omega bound}
                \frac{1}{16} \delta(\eps) \leq \wh \omega_1 \leq \frac{1}{4} \wh \omega_2 \leq \frac{1}{8} \delta(\eps).
            \end{align}
            Note that the collection of all such $(\wh \omega_1, \wh \omega_2)$ is compactly contained within the interior of $D(\eps)$ and is hence a positive distance (depending only on $\eps$) away from its complement $D(\eps)^c$. The constraint on $\wh \omega_1, \wh \omega_2$ in \eqref{eq: wh omega bound} is chosen so that when we approximate $\wh \omega_1, \wh \omega_2$ by rationals $\frac{p_1}{q_1}$ and $\frac{p_2}{q_2}$ below, the limiting sequences will eventually belong to $D(\eps)$, where the constants $\bar B_1(\eps)$ and $\bar B_2(\eps)$ are uniform.
            \\
            \\
            We now approximate 
			\begin{align*}
				\tau = \lim_{k \to \infty} \frac{q_2(k)}{q_1(k)}
			\end{align*}
			by a sequence of rationals. As $\tau$ is irrational, each of the individual sequences $q_1(k), q_2(k)$ tends to infinity as $k\to \infty$. Since $\omega_1$ and $\omega_2$ are also irrational, they can be individually approximated by two other sequences 
			\begin{align*}
				\lim_{k\to \infty} \frac{p_j(k)}{q_j(k)} = \omega_j, \text{ for } j=1,2,
			\end{align*}
            with $p_j(k) := \lfloor q_j(k) \omega_j \rfloor \to \infty$ as $k \to \infty$. It is also clear that for any fixed $n$, we can choose the residues of $q_1(k)$ and $q_2(k)$ modulo $n$, since
            \begin{align*}
                \frac{q_2(k) - q_1(k) \mod n}{q_2(k)} \to 0,
            \end{align*}
            as $k \to \infty$.
			\\
            \\
            We now choose $k_1 \in \N$ so that for all $k\geq k_1$, we have 
			\begin{align*}
				\left(\frac{p_1(k)}{q_1(k)}, \frac{p_2(k)}{q_2(k)} \right) \in D(\eps),
			\end{align*}
			which is possible by our constraint on $\wh \omega_1, \wh \omega_2$ in \eqref{eq: wh omega bound}. In particular, for $k \geq k_1$, we have monotonicity of $\mathcal{R}(p_1(k)/q_1(k), p_2(k)/q_2(k), e)$ as a function of $e$ with the uniform constants $\bar B_1(\eps)$ and $\bar B_2(\eps)$ chosen above. Note that since the ratio
			\begin{align*}
				\mathcal{R}\left(\frac{p_1}{q_1}, \frac{p_2}{q_2}, e_0 \right) = \frac{\beta \left(\frac{p_1(k)}{q_1(k)}; e_0\right)}{\beta \left(\frac{p_2(k)}{q_2(k)}; e_0\right)}
			\end{align*}
			approximates $\tau$ arbitrarily well, we have
            \begin{align}\label{eq: beta tau approx}
                \left|\frac{\beta \left(\frac{p_1(k)}{q_1(k)}; e_0\right)}{\beta \left(\frac{p_2(k)}{q_2(k)}; e_0\right)} - \tau \right| \to 0, \qquad \text{as } k \to \infty.
            \end{align}
            We now set
            \begin{align*}
                \sigma_k := \left|\frac{\beta \left(\frac{p_1(k)}{q_1(k)}; e_0\right)}{\beta \left(\frac{p_2(k)}{q_2(k)};e_0\right)} - \frac{q_2(k)}{q_1(k)} \right|.
            \end{align*}
            Since $\tau$ is approximated by $q_2(k)/q_1(k)$, \eqref{eq: beta tau approx} implies that $\sigma_k \to 0$ as $k \to \infty$.
            \\
            \\
            If the image of $\mathcal{R}(p_1(k)/q_1(k), p_2(k)/q_2(k), e)$ contains a symmetric interval of width $4 \sigma_k$ as $e$ varies through an interval of the form $[e_0 - \rho_k, e_0 + \rho_k]$, then there exists a nearby eccentricity $e_k$ such that $\sigma_k$ vanishes and the ratio function equals $q_2(k)/q_1(k)$, i.e., $e_k$ solves the length spectral resonance equation
			\begin{align*}
				q_1(k) \cdot \beta \left(\frac{p_1(k)}{q_1(k)}; e_k\right) = q_2(k) \cdot \beta \left(\frac{p_2(k)}{q_2(k)}; e_k\right).
			\end{align*}
            (Recall from Definition \ref{mbf} that $\beta(p/q)$ is $1/q$ times the maximal length of periodic orbits with rotation number $p/q$, so $q \beta(p/q)$ gives the length of such a maximizing orbit). Using \eqref{eq:symmetric interval}, we see that as $e$ varies over an interval of the form $[e_0 - \rho_k, e_0 + \rho_k]$ with $\rho_k$ sufficiently small so as to ensure monotonicity of the ratio function on the region \eqref{eq: monotonicity}, the ratio function varies through a symmetric interval of width $\geq 4 \sigma_k$ if
            \begin{align*}
                2 \sigma_k \leq \frac{c(\eps)}{\bar B_1(\eps)} \left|\frac{p_2(k)}{q_2(k)} \left(\frac{p_1(k)}{q_1(k)}\right)^3 -  \frac{p_1(k)}{q_1(k)} \left(\frac{p_2(k)}{q_2(k)}\right)^3 \right| \rho_k - \bar B_2(\eps) \rho_k^2.
            \end{align*}
            The right hand side is quadratic in $\rho_k$, so we can solve for the roots to obtain
            \begin{align}
            \begin{split}
                \rho_{k, \pm} =&
                \frac{\frac{c(\eps)}{\bar B_1(\eps)} \left|\frac{p_2(k)}{q_2(k)} \left(\frac{p_1(k)}{q_1(k)}\right)^3 -  \frac{p_1(k)}{q_1(k)} \left(\frac{p_2(k)}{q_2(k)}\right)^3 \right|}{2 \bar B_2(\eps)}\\
                & \pm \frac{\sqrt{\frac{c(\eps)^2}{\bar B_1(\eps)^2}\left|\frac{p_2(k)}{q_2(k)} \left(\frac{p_1(k)}{q_1(k)}\right)^3 -  \frac{p_1(k)}{q_1(k)} \left(\frac{p_2(k)}{q_2(k)}\right)^3 \right|^2 + 8 \bar B_2(\eps) \sigma_k}}{2 \bar B_2(\eps)}.
            \end{split}
            \end{align}
            Note that the maximal root satisfies $\rho_{k,+} \to 0^+$ as $k \to \infty$, so we can choose $k_1 \in \N$ such that for all $k \geq k_1$,
             \begin{align*}
                 0 < \rho_{k,+} &\leq \frac{c(\eps)}{\bar B_1(\eps)} 3 \times 2^{-19} \delta(\eps)^4\\
                 &\leq \frac{c(\eps)}{\bar B_1(\eps)} \left|\frac{p_2(k)}{q_2(k)} \left(\frac{p_1(k)}{q_1(k)}\right)^3 -  \frac{p_1(k)}{q_1(k)} \left(\frac{p_2(k)}{q_2(k)}\right)^3 \right|, 
             \end{align*}
             in which case $\mathcal{R}(p_1(k)/q_1(k), p_2(k)/q_2(k), e)$ is monotone on the interval $[e_0 - \rho_{k, +}, e_0 + \rho_{k, +}]$ and its image contains a symmetric interval of width $4 \sigma_k$. Further, since $\rho_{k,+} \to 0^+$, the sequence $e_k$ will converge to $e_0$ as $k$ tends to $\infty$. As $e_0$ was arbitrary outside of a locally finite set $E$ (cf. Lemma \ref{LowerBound}), we conclude that the set of eccentricities for which there exists a length spectral resonance is dense in $[0,1]$, as stated in the theorem.}
		\end{proof}

	\begin{rema}
	    The family of ellipses with length spectral resonances is dense, but nevertheless has measure $0$ and is in fact countable. This follows from the fact that the lengths of all orbits are given by the functions \eqref{lengths of hyperbolic orbits} and \eqref{lengths of elliptic orbits}, which are analytic in $e$. Since there are only countably many pairs of integers $(p, q)$ and $(\tilde{p}, q)$, it is sufficient to prove the countability of resonances between any two given pairs. However, those resonances are the zeros of analytic functions and thus are either locally finite or coincide with the domain of definition. The latter case (identity resonance) can be excluded by considering expansions as $e\rightarrow 1$. 
     \label{resonancezeromeasure}
	\end{rema}
	
	\section{Destroying stray orbits}\label{Destroying other orbits}
	
	\subsection{Controllable family of deformations}
    
	We begin with an ellipse $\Omega_0 = \mathcal{E}_e$ satisfying the hypotheses of Theorem \ref{Eccentricities}; there are one-parameter families of type $(p, q)$ and type $(p', q')$ orbits which are tangent to confocal ellipses and share the same length $L$. \brown{From here on, the choice of $\Omega_0$ will be fixed once and for all, as well as the length $L$.} \red{We also assume that $e>0$.} Let $\Omega$ be any $C^\infty$-small deformation of $\Omega_0$ which preserves a family $\Fm$ of $4m$ distinct periodic orbits -- $2m$ of type $(p,q)$ and $2m$ of type $(p', q')$. \red{We recall that $m$ denotes the desired smoothness of the wave trace at $L$ and that it is related to the number of $B_{\gamma, j}$-s that we need to cancel (see \eqref{eq15}).}
	
	\begin{def1}\label{def:stray}
		Let $\Omega$ be a deformation of the ellipse \red{$\Omega_0$} as described above. A \textbf{stray orbit} of length $L$ is an orbit in $\Omega$ which has length $L$ and does not belong to the family $\Fm$ of $4m$ preserved orbits.
	\end{def1}
	Our aim is to prescribe a \red{family of deformations} which destroys any additional stray orbits of the same length $L$, which will lead to cancellations in Theorem \ref{mainth}. \red{The deformations in this family} will preserve $4m$ orbits inside of \red{the} ellipse, provide needed values of $\det \partial^2 \mathcal{L}$ and $\partial_{\deg}^3 \mathcal L$, and destroy other orbits of length $L$. In Section \ref{Solving the main equation}, we will use this family to match the parameters and create cancellations.
	\begin{figure}
				\centering
				\includegraphics[width=8cm]{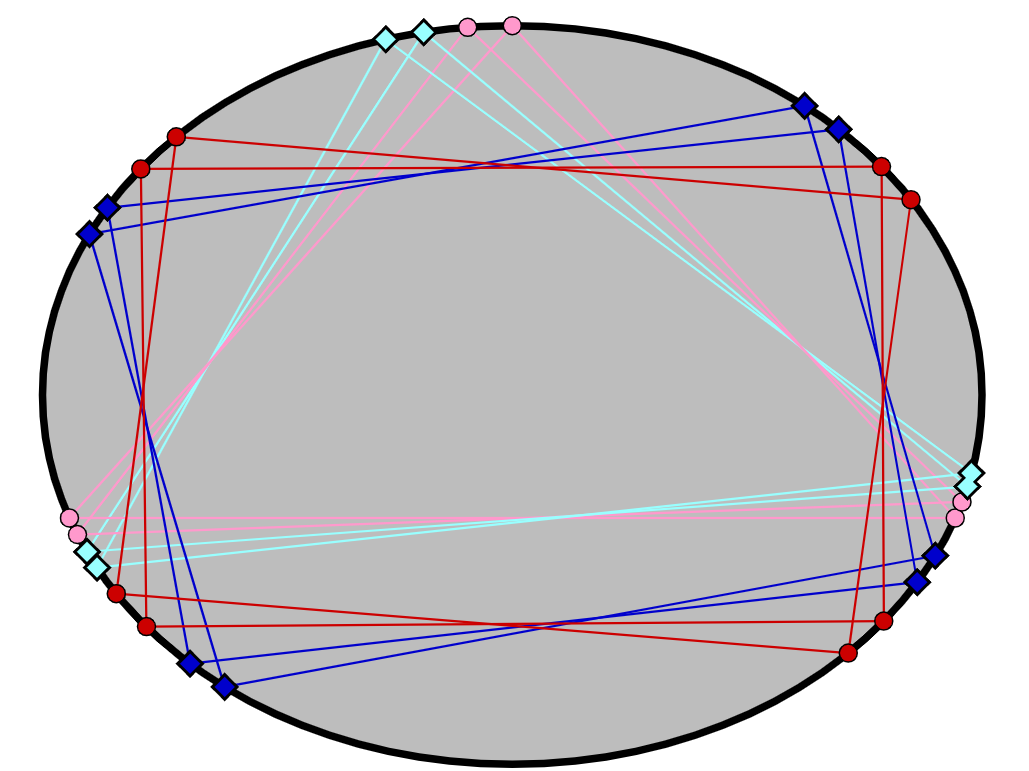}
				\caption{An ellipse with $4m$ orbits selected ($m = 2$). There are $4$ families of orbits. Lighter colored orbits have rotation number $p/q$ and the darker ones have rotation number $p'/q'$. Red and pink orbits (circles) will become elliptic while blue and light blue ones (diamonds) will become hyperbolic after the deformation.}
				\label{fig:scheme}
			\end{figure}
\\
\\
    \red{When prescribing a deformation $\Omega$ of \red{the} ellipse $\Omega_0$ (or any other domain for that matter), it may be convenient to do so using normal deformations. \brown{Recall that for any domain $\Omega_0$, its boundary $\d \Omega_0$ can be parametrized with respect to arclength by a function $x_{\Omega_0}(s) = (x_{\Omega_0,1}(s), x_{\Omega_0, 2}(s))$. A \textbf{normal deformation} is one such that} at every point of \red{the} ellipse $x_{\Omega_0}(s)$ (parametrized with respect to arclength), there is a value $\mu(s)$ which measures how far one has to travel from $x_{\Omega_0}(s)$ along the unit normal $\vec{n}(s)$ of $\Omega_0$ to hit a point of $\Omega$. The smooth function $\mu(s)$ uniquely determines the deformation $\Omega$.}
    
	\begin{def1}
		Let $m > 1$ and $\Omega_0 = \Ee$ be any ellipse satisfying the length spectral resonance condition in Theorem \ref{Eccentricities}. We introduce the parameters $\bar{\eps} = (\eps_1, \ldots, \eps_{4m})$ and $\bar{\dt}= (\dt_1, \ldots, \dt_{4m})$, with $0 < \eps_u < \eps_0$ and $\red{\sqrt[3]{\eps_u}} < \dt_u < \dt_0$ for every $u = 1, \ldots, 4m$ and some fixed $\eps_0, \dt_0 > 0$. We call a family of deformations $\left\{\Omega_{\bar{\eps}, \bar{\dt}}\right\}$ \textbf{controllable} if the following conditions hold:
		\begin{enumerate}
			\item $\Omega_{\bar{\eps}, \bar{\dt}}$ is a multiscale normal deformation of $\Omega_0$, by which we mean that \brown{if $x_0(s)$ is an arclength parametrization of $\Omega_0$, then}
			\begin{align}
                x_{\bar \eps, \bar \dt}(s) = x_0(s) + \mu_{\bar{\eps}, \bar{\dt}}(s)\Vec{n}(s)
				\label{eq62}
			\end{align}
            is a parametrization of $\Omega_{\bar \eps, \bar \dt}$.
            
			\item There is a family $\mathcal{F}_m = \left\{\gamma_u\right\}_{u = 1}^{4m}$ consisting of $4m$ periodic orbits in $\Omega_0$, all of which have the same length and remain fixed in $\Omega_{\bar{\eps}, \bar{\dt}}$, i.e. the deformation makes first order contact with $\Omega_0$ at the reflection points of each orbit. The first $2m$ are of type $(p, q)$ and are tangent to a \red{rational} caustic $\Gamma$, while the last $2m$ are of type $(p', q')$ and are tangent to a second \red{rational} caustic $\Gamma'$. We denote the collection of their reflection points by $\d \Fm$ and individually by $x_i$, $1 \leq i \leq 2m(q+q')$.
			\item In $\Omega_{\bar{\eps}, \bar{\dt}}$, the orbits $\gamma_u$, for $u = 1, \ldots, m$ and $u = 2m+1, \ldots, 3m$, are hyperbolic while the others are elliptic.
			\item For $u = 1, \ldots, 4m$, we have
			\begin{align*}
				|\det \partial^2 \mathcal L(\gamma_u)| = \eps_u, \;\;\; \partial_{\deg}^3 \mathcal L = \dt_u.
			\end{align*}
            \brown{This is possible by rescaling $\eps_u$ so that $c_\gamma = 1$ in the notation of Theorem \ref{partone}.}
			\item There are no stray orbits of length $L$ in $\Omega_{\bar{\eps}, \bar{\dt}}$ and $L$ is isolated in the length spectrum.
			\item The perimeter of $\Omega_{\bar{\eps}, \bar{\dt}}$ is strictly less than that of $\Omega_0$.
			\item For all $\bar{\eps}, \bar{\dt}$, the deformation $\mu_{\bar{\eps}, \bar{\dt}}(s)$ remains in a fixed $C^{\infty}$ neighborhood of $0$.
			\item There exists a $\xi > 0$ such that for each $\gamma_u \in \mathcal{F}_m$ \red{and each $i$}, the perturbation $\mu$ satisfies
			\begin{align*}
				\left\|\mu_{\bar{\eps}, \bar{\dt}}(s)|_{\left(s_i(\gamma_u) - \xi, s_i(\gamma_u) + \xi\right)}\right\|_{C^{2m+2}} \red{\le \max(\eps_0, \dt_0)^{1/2}},
			\end{align*}
			where $s_i(\gamma_u)$ are the arclength coordinates of reflection points of $\gamma_u$. 
		\end{enumerate}
		\label{def61}
	\end{def1}
	
	While Condition $8$ above guarantees that the deformation $\Omega$ will be $\red{\eps_0^{1/2}}$-$\red{\dt_0^{1/2}}$-close to $\Omega_0$ in a neighborhood of the reflection points of selected orbits, no constraints are imposed away from them. In what follows, our deformation $\mu$ will be split into \textbf{local} (near the reflection points) and \textbf{nonlocal} (away from them) \textbf{zones}. If we denote these by $\mu_1$ and $\mu_2$ respectively, then Condition $8$ asserts that $\red{\|\mu_1\|_{C^{2m + 2}} \lesssim \max(\eps_0, \dt_0)^{1/2}}$. We will measure the norm of $\mu = \mu_1 + \mu_2$ by an independent parameter $\iota$, which is allowed to be much greater than $\red{\eps_0}$ and $\red{\dt_0}$; it effectively measures the size of $\mu_2$.
	
	\begin{rema}
		Note that the limiting domain of $\Omega_{\bar{\eps}, \bar{\dt}}$ as $\bar{\eps}, \bar{\dt} \to 0$ will \textit{not} be an ellipse. \red{This is because $\iota$ will not depend on $\bar{\e}$, $\bar{\dt}$ and will instead measure the norm of $\mu_2$.} If the perturbation is $C^2$-small (\red{meaning that $\iota$ is small}), then convexity will be preserved.
	\end{rema}

	The remainder of this section will be dedicated to a proof of the following theorem.
    
	\begin{theo}
		For every $m>1$, there exist controllable families of deformations in any $C^{\infty}$ neighborhood of $\Omega_0$.
		\label{th3}
	\end{theo}

    \begin{rema}
        The proof of Theorem \ref{th3} is quite technical. With no loss of continuity, the reader may skip ahead directly to Section \ref{Solving the main equation} in order to see how the existence of controllable families implies Theorem \ref{mainth}.
    \end{rema}
	
	The deformation $\mu$ will be divided into two zones. In a neighborhood of the reflection points of preserved orbits $\gamma_u \in \mathcal{F}_m$, finer control is required: $\mu$ is smaller with respect to the \red{$C^{2m+2}$ norm} and the Hessian of the length functional $\LL$ will be prescribed there. Outside of this neighborhood, there is more freedom to deform the domain but we must eliminate all stray orbits of length $L$. The rationale for such a decomposition into local and nonlocal zones is the dependence of Balian-Bloch-Zelditch wave invariants only on the local geometry of the boundary near reflection points in $\d \Fm$. As above, we will correspondingly decompose the normal deformation into $\mu = \mu_1 + \mu_2$, with $\mu_1$ and $\mu_2$ supported near the local and nonlocal zones respectively. The proof of Theorem \ref{th3} has \red{three} main steps.
	\begin{itemize}
		\item \textbf{Orbit selection.} In Definition \ref{def62}, we introduce the notion of a \textit{correctly selected family} of orbits in the starting ellipse $\Ee$. The conditions are combinatorial in nature and ensure that the reflection points of orbits (not necessarily those in $\Fm$) are sufficiently disjoint. This makes it possible to deform $\Ee$ within a controllable family in Steps $2$ and $3$ below. To show that such families exist, we show in Proposition \ref{Condition1} that orbits which are not correctly selected form a union of submanifolds in the configuration space $\Ee^{4m}$ which have positive codimension. In particular, ``correctly selected'' is an open and full measure condition.
   
        \item \textbf{Nonlocal deformation.} Having fixed a family $\Fm$ of $4m$ correctly selected orbits, the main idea is to make the nonlocal part $\mu_2$ of the deformation $\mu$ negative while preserving convexity. This will decrease the perimeter of $\Omega_0$ together with the lengths of all possible stray orbits of length $L$ which avoid a neighborhood of $\d \Fm$. \red{Note that we focus on orbits of length $L$ in \red{the} ellipse since Lemma \ref{lema61} and Proposition \ref{prop61} prohibit orbits of other lengths in \red{the} ellipse from changing their length to $L$, provided that the deformation is small enough.} The norm of $\mu_2$ is measured by a parameter \red{$\iota$, which is independent of $\eps_0$ and $\dt_0$}, and is allowed to be much greater than \red{$\|\mu_1\|$}, unrestricted by the deformation parameters $\bar{\eps}$ and $\bar{\dt}$. Hence, if even one reflection point of a potentially stray orbit falls in the nonlocal zone, its length will shrink to less than $L$, the length at which we localized the wave trace; the negativity of $\mu_2$ will dominate whatever influence $\mu_1$ has on the length functional $\LL$, allowing more freedom with $\mu_1$ to prescribe the curvature jet at reflection points of $\Fm$. Both $\mu_1$ and $\mu_2$ will be smooth.

		\item \textbf{Local deformation.} Without loss of generality, for each $\gamma_u \in \Fm$, we may select a first (marked) reflection point, which we denote by $x_u^1$. Define the set of all first reflection points of orbits in $\Fm$ to be $\d \Fm^1$. In order to simplify the construction of $\mu_1$, we localize its support to a neighborhood $\Fm^1$; the deformation $\mu$ will leave unchanged the boundary of $\Omega_0$ near the remaining $2m(q+q'-2)$ reflection points. This will also simplify the exclusion of stray orbits which reflect in a neighborhood of $\d \Fm$.
	\end{itemize}

	\begin{theo}
		In $\Omega_0 = \E_e$, there exists a $C_0(e)$ such that for any $C \geq C_0$ and $\xi > 1$, one can select the positions of $4m$ orbits and an arbitrarily $C^\infty$-small deformation $\mu_2$, which is identically $0$ on a $\xi$-neighborhood of $\Fm$, such that the following holds: for every $C^\infty$-small $\mu_1$ which is supported on a $C\xi$-neighborhood of $\d \Fm^1$, the $\mu = \mu_1 + \mu_2$ deformed domain $\Omega$ has no stray orbits of length $L$, provided that $\Omega$ and $\Omega_0$ make first order contact ($\mu(s) = \mu'(s) = 0$) only at $\d \Fm^1$ within a $C\xi$-neighborhood of $\d \Fm^1$.
		\label{th4}
	\end{theo}
	
	We will first construct $\mu_2(s)$ and then apply Theorem \ref{th4} to pick a local deformation $\mu_1$ in Section \ref{LocalDeformation}, which will give the needed $3$-jet at reflection points while simultaneously eliminating stray orbits. Such a pair will give rise to a deformation $\mu$ satisfying the conditions of Definition \ref{def61}.

    \subsection{Preliminaries}

	We first tabulate some important properties of stray orbits which will be used throughout this section. Assume that a stray orbit of length $L$ exists and is of type $(p'',q'')$. We first bound $p''$ and $q''$:
	\begin{lemm}
		There exist $p_0$ and $q_0$ such that for every $C^\infty$-small perimeter decreasing deformation $\Omega$ of $\Omega_0$, the length $L$ cannot be achieved by orbits of type $(p'',q'')$ whenever $p'' > p_0$ or $q''>q_0$.
		\label{lema61}
	\end{lemm}
	
	\begin{proof}
		From Proposition \ref{prop:largep}, we know that $p''$ is bounded by a fixed $p_0$ \red{for every small deformation of the ellipse}. To bound $q''$, we use \eqref{eq38} and Lemma 8.7 in \cite{kov21} \brown{(for some choice $1 > e_{\text{max}} > e(\Omega_0)$)}. Since $p''$ is already bounded, all \red{lengths of} orbits of type $(p,q)$ with sufficiently large $q$ belong to a small neighborhood of $p | \d \Omega|$. If $L \neq p |\d \Omega_0|$ for some $p$, then $L \neq p |\d \Omega|$, since $|\d \Omega|$ and $c_{p, 1}$ (the first Marvizi-Melrose invariant, cf. equation \eqref{eq38}) are continuous with respect to $C^\infty$ deformations and \brown{$L \neq p |\d \Omega|$ is an open condition. On the other hand, if $L = p|\d \Omega_0|$ for some $p$, then $p |\d \Omega| < L$ since we assumed that our deformation decreases the perimeter of $\Omega_0$. Since $c_{p, 1} < 0$ \red{(see \eqref{bettaylor} and \eqref{betai})}, the lengths of type $(p,q)$ orbits accumulate at $\red{p}|\d \Omega|$ from below. In particular, they do not coincide with $L$.}
	\end{proof}
	
	Throughout the rest of the paper, we will denote the diffeomorphism mapping $\d \Omega_0 \to \d \Omega$ by
	\begin{align*}
		\Phi_\mu(x) = x + \mu_{\bar{\eps}, \bar{\dt}}(x) n(x) \in \d \Omega,
	\end{align*}
	and its induced product map by
	\begin{align}\label{diffeo}
		\Phi_{\mu, q}:=\Phi_\mu \underbrace{\times \cdots \times}_{q \text{ times}} \Phi_\mu : \d \Omega_0^q \to \d \Omega^q.
	\end{align}
	By an abuse of notation, we will sometimes write $\Phi_{\mu,q}$ with an argument in either arclength or action angle coordinates, as opposed to Euclidean coordinates on $\R^2$.\\
    \\
	It follows from Lemma \ref{lema61} that lengths of orbits in $\Omega$ are close to lengths in $\Omega_0$:
	\begin{prop}[\red{Lemmata 8.3 and 8.4 in} \cite{kov21}]
		Let $L_{p, q}(\Omega) \in \lsp_{p, q}(\Omega)$, where $\Omega$ is a $C^{\infty}$-small deformation of the ellipse $\Omega_0 = \Ee$ (Condition 1 in Definition \ref{def61}). Then, there exists $L_{p, q}(\Omega_0) \in \lsp_{p, q}(\Omega_0)$, such that:
		\begin{equation}
			|L_{p, q}(\Omega) - L_{p, q}(\Omega_0)| = o_{p, q, e}(1),
		\end{equation} 
		with respect to the $C^\infty$-distance on boundary curvatures. Moreover, for every type $(p, q)$ orbit $\gamma'$ in $\Omega$, there exists a type $(p,q)$ orbit $\gamma$ in $\Omega_0 = \Ee$ such that $\Phi_{\mu, q}^{-1} (\d \gamma')$ is $o_{p, q, \red{e}}(1)$-close to $\gamma$ with respect to the \pink{$C^\infty$-distance between $\mu$ and zero.}
		\label{prop61}
	\end{prop}
	
	Since $p$ and $q$ are already bounded by Lemma \ref{lema61}, the $o_{p,q,\Omega_0}(1)$ estimates above can be made uniform. Note also that the type $(p,q)$ part of the length spectrum of \red{the} ellipse is finite for $p, q$ bounded. Hence there can be no accumulation at the specified length $L$ and an orbit of length $L$ in $\Omega$ must be close to an orbit of length $L$ in the ellipse.
    \\
    \\
    Thus, any stray orbit in $\Omega$ will be close to some orbit in $\Omega_0$ of length $L$. As long as we ensure that the deformation either decreases or increases their lengths, there will be no stray orbits in $\Omega$. \brown{In Definition \ref{def61}, we also required that $L$ be isolated in the length spectrum. Theorem \ref{th4} does not address this directly. However, we have the following.
\begin{prop}
    In the setting of Theorem \ref{th4}, if the selected orbits become nondegenerate under the deformation, then $L$ is isolated in the length spectrum of the deformed domain. 
\end{prop}
    \begin{proof}
		If $L$ were an accumulation point in the length spectrum, Lemma \ref{lema61} \red{would} imply that the corresponding orbits have bounded $p$ and $q$. The arclength coordinates of these trajectories would belong to the compact set $\cup_{q \leq q_0} \T_\ell^q$ for some $q_0 < \infty$ and hence have a convergent subsequence. By continuity, such a limit point would correspond to a degenerate billiard trajectory of length $L$. As our assumption rules out all degenerate trajectories of length $L$, it follows that $L$ must be isolated in the length spectrum.
	\end{proof}
}
	
	\subsection{Proof of Theorem \ref{th4}}\label{subsec:proof of Theorem 6.6}

    \subsubsection{Part 1: orbit selection}
    
    We now select the exact positions of our $4m$ fixed periodic orbits in the ellipse $\Omega_0 = \Ee$. Any \red{rational} caustic provides a one-parameter family of possible orbits, but there are additional constraints. For example, suppose that some other orbit (not one of the $4m$ fixed ones) of length $L$ in $\Omega_0$ has reflection points contained in a subset of $\red{\partial \mathcal F_m}$. Since our deformation makes first order contact at $\red{\partial \mathcal F_m}$, this stray orbit will persist through any controllable deformation we create. To ensure such problems don't arise, we introduce the following four requirements for orbit selection.
		
	\begin{def1} \label{def62}
		We call a family $\mathcal{F}_m = \left\{\gamma_u\right\}$ of $4m$ orbits in $\Ee$ \textbf{correctly selected} if the following conditions hold:
		\begin{enumerate}
			\item If $\gamma_u$ is selected, then no other orbit \red{$\zeta$} of length $L$ in $\Ee$ can have $\d \red{\zeta} \subset \d \gamma_u$.
			\item If $\gamma_u$ is selected, its reflection points do not lie on the minor axis of \red{the} ellipse.
			\item If $\gamma_u$ and $\gamma_v$ are selected, their reflection points are disjoint \brown{($\d \gamma_u \cap \d \gamma_v = \emptyset$)}.
			\item If $\gamma_u$ and $\gamma_v$ are selected, no orbit \brown{$\zeta$} of length $L$ in $\Omega_0 = \Ee$ shares its reflection points with both $\gamma_u$ and $\gamma_v$ \brown{(i.e., either $\d \zeta \cap \d \gamma_u = \emptyset$ or $\d \zeta \cap \d \gamma_v = \emptyset$).}
		\end{enumerate}
	\end{def1}

	\begin{def1}
		Whenever $\d \red{\zeta} \subset \d \gamma$ for some orbits $\red{\zeta}, \gamma$ (the negation of Condition $1$ above), we say that $\red{\gamma}$ \textbf{covers} \red{$\zeta$}.
	\end{def1}
	
	\brown{Conditions $1$ and $4$ forbid stray orbits whose reflection points are completely contained in $\d \mathcal{F}_m$.} Condition $3$ allows us to change $\mu_1$ independently for each orbit, with no risk of interference from the others. As we shall see, it is harder to prove that bouncing ball orbits along the minor axis do not give rise to stray orbits, so Condition $2$ is useful in the event that the length of the minor axis divides $L$. Note also that Conditions $2$, $3$ and $4$ hold outside of the union of (stratified) submanifolds which have positive codimension in $\d \Omega^{2m(q + q')}$. In particular, they are open and full measure conditions, so they can be easily satisfied. For example, amongst the one-parameter family of orbits tangent to any given \red{rational} caustic, only finitely many others have reflection points on the minor axis. Similarly, after fixing any such chosen orbit, only finitely many others share a common point of reflection.
	\\
	\\
	Note that Conditions $1$ and $4$ in Definition \ref{def62} pertain to \textit{all} orbits in $\Ee$, not just those in $\Fm$. This makes them more difficult to verify than Conditions $2$ and $3$. However, for Condition $4$, we can argue as follows. Consider an orbit \red{$\gamma_u$} and the collection of all others of length $L$ in $\Ee$ (having any rotation number) which share a common point of reflection with \red{$\gamma_u$}. Lemma \ref{lema61} guarantees that there are at most finitely many possible rotation numbers, so there are only finitely many orbits.
    \\
    \\
    \red{Note that this uses the fact that $\Omega_0$ is an ellipse. More specifically, it requires an invariant curve to rotate the orbits around. If one were to consider another domain $\Omega_0$ which has no invariant curves of periodic orbits and if the number of orbits of length $L$ were finite, then it would be difficult to correctly select orbits. Moreover, having an invariant curve alone is not enough to satisfy Condition 1. To satisfy Condition 1, we will use special formulas for ellipses, introduced in Section \ref{subsecell}. To correctly select orbits in a more general domain, one would need to analyze the properties of action-angle coordinates.}
	\\
	\\
	Now we focus on satisfying Condition $1$. If one rotates an orbit of type $(p, q)$ and length $L$ around a rational caustic in $\Ee$, a priori, there may persist another orbit of length $L$ which remains covered by it. We claim that aside from a finite number of exceptions, this cannot happen.
	
	\begin{prop}\label{Condition1}
		There are at most finitely many orbits in the ellipse $\Omega_0 = \Ee$ which violate Condition $1$ in Definition \ref{def62}.
	\end{prop}

	\begin{proof}
		The negation of Condition 1 is that for a family of $4m$ orbits of length $L$, there exists an orbit external to $\Fm$ which shares a common reflection point with one in $\Fm$. Note that by Lemma \ref{lema61}, \red{there are only finitely many orbit types of $\zeta$ that may fail Condition $1$, so} it suffices to prove the lemma for orbits of a fixed type, say type \red{$(p'',q'')$}. Such an orbit will be tangent to either an elliptic or hyperbolic caustic, pass through the focal points or be a bouncing ball orbit along the minor axis. The only periodic orbits through focal points are bouncing ball orbits along the major axis. It is clear that at most finitely many bouncing ball orbits along either axis can have length $L$.
		\\
		\\
		Assume that some type $(p, q)$ orbit $\gamma$ fully covers another orbit $\zeta$ of type $(p'',q'')$. Denote their corresponding caustics \red{(elliptic or hyperbolic)} by $\Gamma_\gamma$ and $\Gamma_\zeta$. Without loss of generality, assume that the first \brown{reflection} points of $\gamma$ and $\zeta$ coincide and denote their corresponding elliptical polar \brown{angle} coordinate by $\phi_1$. \red{Since $\d \zeta \subset \d \gamma$, we know in particular that} the second reflection point of $\zeta$ coincides with the $i+1$-st of $ \gamma$ for some $i$, and we denote the corresponding coordinate by $\phi_{i+1}$. 
		\\
		\\
		We now rotate \red{both} $\gamma$ along $\Gamma_\gamma$ \red{and $\zeta$ along $\Gamma_{\zeta}$} by changing $\phi_1$ \red{(so that the first reflection points of both $\gamma$ and $\zeta$ are at $\varphi_1$)} and ask whether the covering will persist. By formulae \eqref{eq315} and \eqref{eq316}, $\phi_{i+1}$ depends analytically on $\varphi_1$. Similarly, the elliptical polar coordinate of the second reflection point of $\zeta$ (call it ${\phi}_2''$) depends analytically on $\varphi_1$ when rotated simultaneously along $\Gamma_\zeta$. Hence, \red{the second reflection point of $\zeta$ and the $i+1$-st reflection point of $\gamma$ coincide} on the zero set of a real analytic function \red{in $\varphi_1$} -- that is, either everywhere or only on a locally finite set.
				\begin{figure}
						\centering
						\includegraphics[width=8cm]{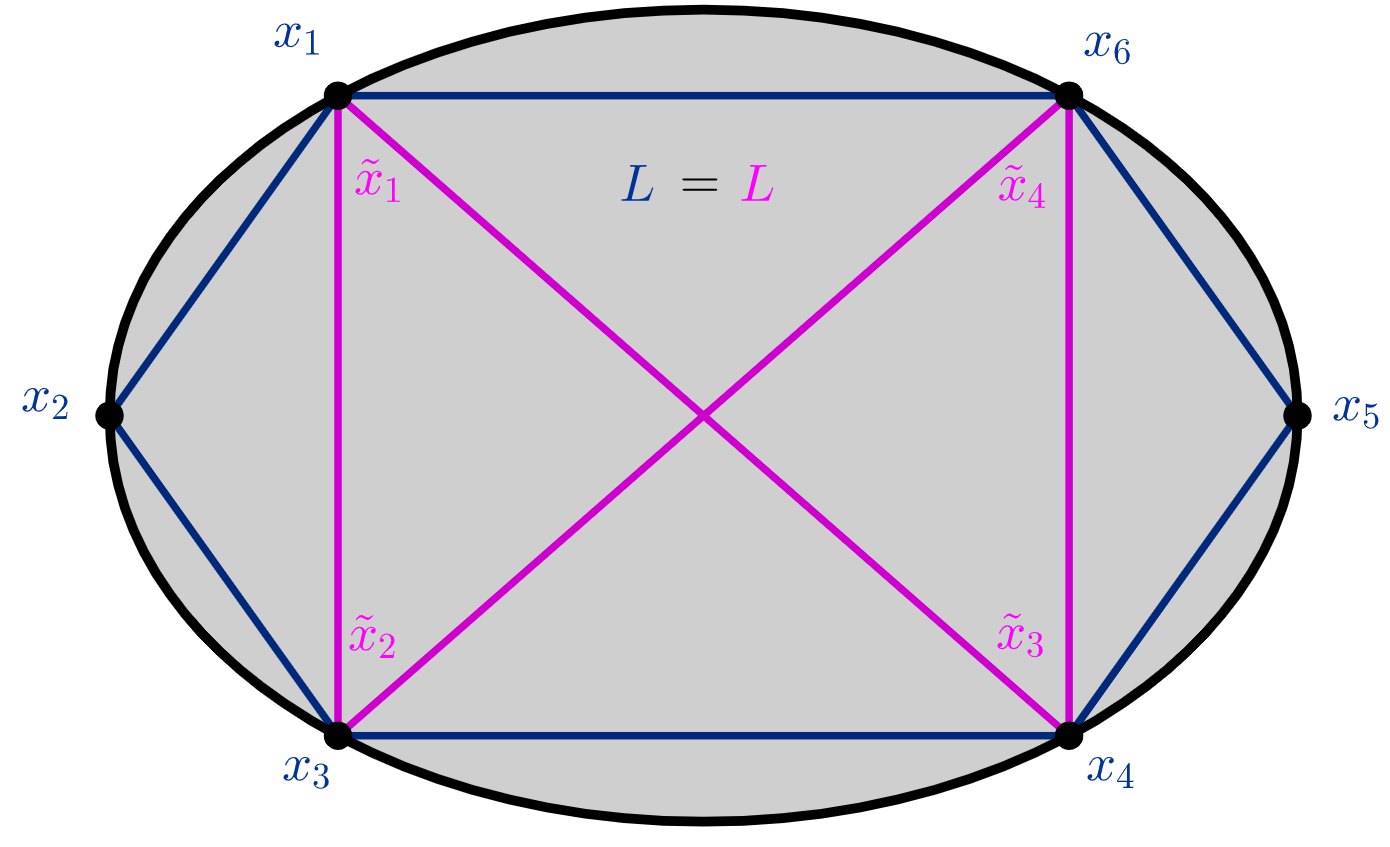}
						\caption{A potential example of an orbit, not satisfying the first selection condition. In this case, the selected orbit (in blue) has $p = 1$ and $q = 6$, whereas the stray orbit is tangent to a hyperbolic caustic, with $\tilde{p}'' = 1$ and $q'' = 4$. The figure is for illustration purposes only and is not drawn to scale -- the actual orbits in this ellipse may have different lengths and/or reflection points.}
						\label{fig:firstcondition}
					\end{figure}
		We focus on a single orbit $\gamma$ which is tangent to an elliptic caustic; in this case, $\phi_{i+1}$ extends to an analytic function of $\phi_1$ on a complex strip containing the real axis. The same holds for ${\phi}_2''$, if $\zeta$ is also tangent to a confocal ellipse. However, if $\zeta$ is tangent to a confocal hyperbola and $x_1$ is on the lower side of \red{the} ellipse, then ${\phi}_2''(\phi_1)$ is defined only for $\phi_1 \in [-\arcsin k''^{-1}, \arcsin k''^{-1}]$, where
		\begin{align*}
			k'' = \frac{e}{\sqrt{1- (\lambda'')^2}},
		\end{align*}
		with $\lambda''$ being the parameter of the corresponding conic section (cf. equation \eqref{confocal}).
		\\
		\\
		We address the latter case first. Assume $\phi_1 \to \arcsin k''$ from below so that ${\phi}_2''$ approaches some limit, corresponding the elliptical polar coordinate of the reflection point on the upper half of the $\Omega_0$. We claim that
		\begin{equation}
			\frac{d}{d\phi_1} {\phi_2}''(\phi_1) \rightarrow \infty, \; \; \phi_1 \rightarrow \arcsin k''^{-1}.
			\label{eq65}
		\end{equation}
		Since $\varphi_1$ is near the singularity, formula \eqref{eq316} implies that $\varphi_1'(\wt \eta_0) \to 0$ and hence, $\wt \eta_0'(\phi_1) \to \infty$.
		\\
		\\
		Since
        \begin{align*}
            \frac{d}{d \phi} \phi_2'' = \frac{d \phi_2'}{d \wt \eta}(\wt \eta_1) \frac{d \wt \eta}{d \phi}(\phi_1),
        \end{align*}
        it remains to show that the former is bounded away from $0$. If ${\phi}_2''$ tends to a limit in $(\pi - \arcsin k''^{-1}, \pi + \arcsin k''^{-1})$, then boundedness away from zero is clear, since those points are nonsingular. We claim that ${\phi_2}''$ cannot approach the singular points $\pi \pm \arcsin k''^{-1}$; otherwise, in the limit, the link between $x_1$ to $\tilde{x}_2$ will intersect the hyperbolic caustic at $2$ points ($x_1$ and $\tilde{x}_2$), which contradicts tangency. Hence, \eqref{eq65} follows.
		\\
		\\
		Since the corresponding function is analytic near $\arcsin k''^{-1}$, we know that $\frac{d \phi_{i+1}}{d \phi_1}$ is bounded there. Close to $\arcsin k''^{-1}$, the functions ${\phi}_2''(\phi_1)$ and $\phi_{i+1}(\phi_1)$ have different derivatives and thus coincide on at most a locally finite set. A priori, this set may have $\pm \arcsin k''^{-1}$ as a limit point. However, the derivative of $\phi_{i+1}$ is bounded away from zero, whereas that of ${\phi}_2''$ approaches zero, so points of coincidence do not accumulate at $\arcsin k''^{-1}$. Thus, there can be only finitely many stray orbits of length $L$ which are tangent to a hyperbolic caustic and remain covered by another orbit belonging to the family $\Fm$ of selected orbits.
        \begin{figure}
				\centering
				\includegraphics[width=7cm]{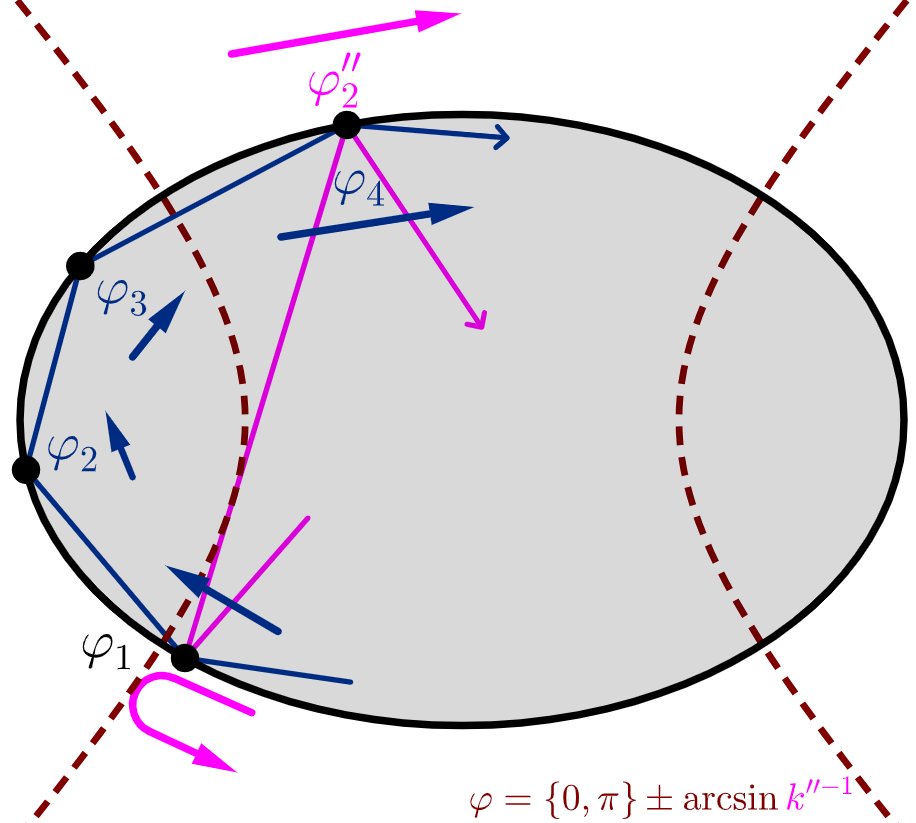}
				\caption{As we rotate the blue orbit $\gamma$ of an elliptic caustic and the stray pink orbit $\zeta$ of a hyperbolic caustic, we see that the covering cannot persist as $\varphi_1$ is close to the limiting hyperbola. The speed and trajectory of reflection points under rotations are indicated with the arrows. For $\zeta$, the first reflection point is limited by the hyperbola, so it will decelerate and eventually reverse direction. The second reflection point is away from the hyperbola, so it rotates rapidly, thus $\frac{d}{d\phi_1} {\phi_2}''(\phi_1) \rightarrow \infty$. Meanwhile $\gamma$ is unaffected by the hyperbola, leading to bounded derivatives.}
				\label{fig:enter-label13}
			\end{figure}
\\
\\
  We now consider the case when $\zeta$ is also tangent to an elliptic caustic. Assume that for all $\phi_1 \in \mathbb{R}$, we have
		\begin{align*}
			\phi_{i+1}(\phi_1) = {\phi}_{2}''(\phi_1) \mod 2\pi \;\;\Rightarrow\;\; \phi_{i+1}(\phi_1) = {\phi}_{2}''(\phi_1) + 2\pi n.
		\end{align*}
		Iterating this equality $t$ times, we get
		\begin{align*}
			\phi_{it+1}(\phi_1) = {\phi}_{t+1}''(\phi_1) + 2\pi t n.
		\end{align*}
		Substituting $t = q$ gives
		\begin{align*}
			2\pi i p = 2\pi \omega'' q + 2\pi q n\;\; \Rightarrow \;\;  i\omega = \omega'' \mod 1.
		\end{align*}
		We now exploit this relation using \eqref{eq315}. Choose $\phi_1 = \eta_0 = 0$, so that
		\begin{align*}
			\phi_{i+1} = \am \left(4K(k)i\omega, k\right), \; \; \; {\phi}_2'' = \am \left(4K(k'')\omega'', k''\right).
		\end{align*}
  Using periodicity of the Jacobi amplitude, we see that
	\begin{equation}
			\am \left(4K(k)\omega'', k\right) = \am \left(4K(k'')\omega'', k''\right) \; \; \Rightarrow \; \; \frac{F({\phi}_2'', k)}{K(k)} = \frac{F({\phi}_2'', k'')}{K(k'')}.
			\label{eq610}
		\end{equation}
		The latter equality will lead to contradiction.
        \\
        \\
        For simplicity, we denote ${\phi}_2''$ by $\phi$ and assume $\phi \in (0, \pi/2]$. If $\phi$ were in $(\pi/2, \pi)$, \eqref{eq610} would hold for $\pi - \phi$ instead. If $\pi < \phi< 2\pi$, then the equality would hold for $2\pi - \phi$. Note that $\phi \ne \pi$, since otherwise, the segment of $\zeta$ from $\phi_1$ to $\phi$ would pass through the center of \red{the} ellipse and hence intersect the caustic. In this case, we get a contradiction, since for $\phi \in (0, \pi/2)$,
		\begin{align*}
			\frac{F(\phi, k)}{K(k)}
		\end{align*}
		is a strictly decreasing function in $k$. This can be verified by direct computation:
		\begin{align*}
		\frac{d}{dk} \frac{F(\phi, k)}{K(k)} = & k \int_0^\phi \frac{\sin^2\psi d\psi}{\sqrt{1 - k^2\sin^2\psi}^3} \int_{\phi}^{\pi/2} \frac{d\psi}{\sqrt{1 - k^2\sin^2\psi}}\\
			&- k \int_0^\phi \frac{d\psi}{\sqrt{1 - k^2\sin^2\psi}} \int_{\phi}^{\pi/2} \frac{\sin^2\psi d\psi}{\sqrt{1 - k^2\sin^2\psi}^3}\\
			= & k \int_0^\phi \frac{g(\psi) d\psi}{\sqrt{1 - k^2\sin^2\psi}} \int_{\phi}^{\pi/2} \frac{d\psi}{\sqrt{1 - k^2\sin^2\psi}}\\
			&- k \int_0^\phi \frac{d\psi}{\sqrt{1 - k^2\sin^2\psi}} \int_{\phi}^{\pi/2} \frac{g(\psi) d\psi}{\sqrt{1 - k^2\sin^2\psi}}\\
			< & k \int_0^\phi \frac{g(\phi) d\psi}{\sqrt{1 - k^2\sin^2\psi}} \int_{\phi}^{\pi/2} \frac{d\psi}{\sqrt{1 - k^2\sin^2\psi}}\\
			&- k \int_0^\phi \frac{d\psi}{\sqrt{1 - k^2\sin^2\psi}} \int_{\phi}^{\pi/2} \frac{g(\phi) d\psi}{\sqrt{1 - k^2\sin^2\psi}} = 0,
		\end{align*}
        for strictly increasing $g(\psi) = \frac{\sin^2 \psi}{1 - k^2 
        \sin^2 \psi}$. Unless $\phi = \pi/2$, we get a contradiction. This last case also leads to a contradiction, since instead of considering $\phi_1 = 0$, we could take $\phi_1 + \phi = 0$.
	\end{proof}
	
	Hence, it is possible to choose $4m$ orbits of rotation numbers $p/q$ and $p'/q'$ which satisfy all four conditions in Definition \ref{def62}. 
	
	\subsubsection{Part 2: nonlocal deformation}
	
	We now complete the proof of Theorem \ref{th4} by analyzing $\mu_2$, the nonlocal deformation. It is supported away from the reflection points of all selected orbits in $\mathcal{F}_m$ and hence determines the change in perimeter of \textit{all} periodic orbits in $\Omega$ which do not intersect a small neighborhood of $\d \mathcal{F}_m$. From Proposition \ref{prop61}, we know that periodic points in $\Omega_{\bar{\eps}, \bar{\dt}}$ are close to those in \red{the} ellipse. The following lemma provides a more precise estimate of their locations and perimeters:
	
	\begin{lemm}
		\brown{Let $q>1$ and $\Omega$ be a $C^\infty$-small deformation of the ellipse $\Omega_0$ which is parametrized by $x_0(s) + \mu(s) \Vec{n}(s)$, where $x_0(s)$ is an arclength parametrization of $\Omega_0$.} Assume that $||\mu||_{C^{10}} = O(\iota)$ for some auxiliary parameter $\iota$, which is independent of $\bar{\eps}, \bar{\dt}$. Then, every $q$-periodic orbit $\gamma$ \red{in $\Omega$} which is away from a bouncing ball orbit along the minor axis \red{of $\Omega_0$} \red{(as a set in $\mathbb{R}^2$)} is $O(\iota)$-close to some $q$-periodic orbit $\gamma_0$ in $\Omega_0:$
		\begin{align*}
			\left|\Phi_{\mu, q}^{-1} (\d \gamma) - \d \gamma_0 \right| = O(\iota),
		\end{align*}
		(with the notation in \eqref{diffeo}). Moreover, the length of $\gamma$ is given by
		\begin{equation}
			L_\gamma = L_{\gamma_0} + 2 \sum_{i = 1}^{q} \mu \big( x_i(\gamma_0) \big) \cos \big( \theta_i(\gamma_0)\big) + o(\iota),
			\label{eq612}
		\end{equation}
		where $x_i$ (resp. $\theta_i$) are the points (resp. angles \red{with the interior normal}) of the reflection of $\gamma_0$ in $\mathcal{E}_e$.
		\label{lema63}
	\end{lemm}

	\begin{proof}
		\red{We restrict ourselves to orbits with $|\theta_i| < \pi/2 - d$. We can do this since $q$ is fixed and nearly glancing reflections cannot appear in $q$-periodic orbits (they would make the reflection segments too short to go around the boundary; see the proof of Proposition \ref{prop:largep}).} We first show that $\gamma$ is $O_d(\iota)$ close to $\gamma_0$. We already know that $\gamma$ is close to some orbit in \red{the} ellipse, but we don't have a quantitative estimate. In principle, it could be $O(\sqrt{\iota})$ close. We use, \red{for example, Lemma 8.1} in \cite{kov21} to get the following bound:
		\begin{equation}
			\sup_{\substack{0 \leq s \leq \ell\\ -\pi/2 + d \leq \theta \leq \pi/2 - d }}  \left\|B_\Omega (s, \theta) - B_{\Omega_0}(s, \theta) \right\|_{C^9}= O_d(\iota),
			\label{eq613}
		\end{equation}
		for any fixed $d>0$, \red{where $B$ is the billiard map}. As $q$ is bounded by Lemma \ref{lema61} and $\gamma$ is close to a $q$-periodic point in the ellipse, its reflection angles are uniformly bounded away from $\pm \pi/2$, independently of $\iota$. This allows us to use \eqref{eq613}. Hence, in a neighborhood of $(s_1(\gamma), \theta_1(\gamma))$, the map $B_\Omega^q$ is $O(\iota)$-close to $B_{\Omega_0}^q$.
		\\
		\\
		Now let $\gamma_0$ be an orbit of period $q$ in $\Omega_0$, whose initial coordinates $(s_1(\gamma_0), \theta_1(\gamma_0))$ are closest to $(s_1(\gamma), \theta_1(\gamma))$ with respect to Euclidean distance. Denote by $v$ the vector connecting $(s_1(\gamma_0), \theta_1(\gamma_0))$ to $(s_1(\gamma), \theta_1(\gamma))$ and consider the linearization of $B_{\Omega_0}^q$ near $\gamma_0$:
		\begin{align*}
			B_{\Omega_0}^{q}(s_1(\gamma), \theta_1(\gamma)) - (s_1(\gamma_0), \theta_1(\gamma_0))=\red{DB}_{\gamma_0, \Omega_0}^q v + O(||v||^2). 
		\end{align*}
		Approximating $B_{\Omega_0}^q$ by $B_\Omega^q$, we also have
		\begin{align}
			v= D B_{\Omega_0}^q v + O(||v||^2 + \iota).
			\label{eq615}
		\end{align}     
		Hence, the image of $v$ is close to $v$ itself. Assume for the sake of contradiction that $v \neq O(\iota)$. If $\gamma_0$ was tangent to an elliptic or hyperbolic caustic, then $v$ would be close to the sole eigenvector of the differential of $B_{\Omega_0}^q$. However, a slightly rotated $\gamma_0$ along the caustic will be closer to $\gamma$ than $\gamma_0$. This leads to a contradiction. \red{Thus, $v = O(\iota)$, giving us the first assertion of the theorem in this case.} Equation \eqref{eq615} \red{also leads to a contradiction when} the orbit $\gamma_0$ \red{is} a bouncing ball along the major axis (since it is hyperbolic).
		\\
		\\
		The assertion on the length of $\gamma$ follows from the closeness; since the tangent space to the caustic direction is critical for the length functional \eqref{lengthfunctional}, any $O(\iota)$-shift of reflection points along it will not influence the main term in \eqref{eq612}. The perimeter change associated to the normal deformation $\mu$ is described in \eqref{eq612}. 
	\end{proof}
	
	\begin{rema}
			We briefly comment on why Lemma \ref{lema63} fails for bouncing balls along the minor axis. It turns out that for a dense class of ellipses, the rotation number of that elliptic point is \textit{rational}. Hence, the differential of some iterate of the billiard map is the identity: $DB_{\Omega_0}^q = \Id$ (see \red{Section 8.7 in }\cite{kov21}). \brown{Now consider a periodic orbit $\gamma_0$ which traverses the minor axis multiple times.} This is the only \brown{type of} orbit in \brown{the} ellipse which exhibits type $2$ degeneracy \red{(see Section \ref{sec.poincaremap} for the definition} \brown{of type $2$ degeneracy)}. \brown{Now consider how this periodic orbit evolves under a deformation of size $\iota$. Since $DB_{\Omega_0}^q$ coincides with the identity, the remainder $O(\|v\|^2 + \iota)$ on the right hand side of \eqref{eq615} vanishes identically and one sees that $\left\|v\right\| \sim \sqrt{\iota}$ is indeed possible}. In that case, $\gamma$ may wander further from $\gamma_0$ \brown{upon deformation of the ellipse}, which is why we exclude this case from the lemma. This is also the reason why Condition $2$ for correctly selected orbits in Definition \ref{def62} excludes the minor axis.  
	\end{rema}

    From here on, we will assume that $\xi$ is so small that orbit selection principles $2$, $3$ and $4$ from Definition \ref{def62} extend to an $O(\xi)$-neighborhood of all reflection points in $\d \mathcal{F}_m$. In particular, no periodic orbit of length $L$ should reflect in an $O(\xi)$-neighborhood of the reflection points of two different selected orbits. However, Condition $1$ cannot be extended in the same way, as one can always rotate a selected orbit along its \red{rational} caustic to produce another, whose set of reflection points are $O(\xi)$-close to the original one.
	\\
	\\
	We claim that this is in fact the only obstruction to extending Condition $1$. If $\gamma$ is an orbit which is \textit{not} the rotation of \red{a selected orbit of} length $L$ \brown{that is} tangent to a \red{rational} caustic in $\Omega_0$, at least one of its reflection points must be away from $\d \Fm$. We select parameters in the following order: first $\xi$, second $\iota$, and third $(\eps_0, \dt_0)$. We will work in the regime
	\begin{equation*}
		(\eps_0, \dt_0) \ll \iota \ll \xi \ll 1.
	\end{equation*}
	Within this regime, it follows that away from a $C\xi$-neighborhood of $\d \Fm$, $\mu(s) \le -\iota$ and $\mu(s) < \iota^2$ on $[0, \ell]$. Let $\gamma_0$ be an orbit of length $L$ in $\Omega_0$ which is neither an exception to Condition $1$ (corresponding to rotation along a \red{rational} caustic) nor a bouncing ball orbit along the minor axis. Then, Lemma \ref{lema63} implies that the second order term in \eqref{eq612} for $\gamma_0$ is dominated by the  contribution of $\mu$ away from $\d \Fm^1$ and is thus negative. Hence, no stray orbit in $\Omega$ can appear close to $\gamma_0$; every stray orbit in $\Omega$ is localized near either the minor axis or a rotation of $\gamma_0$ along its corresponding \red{rational} caustic.
		\begin{figure}
				\centering
				\includegraphics[width=8cm]{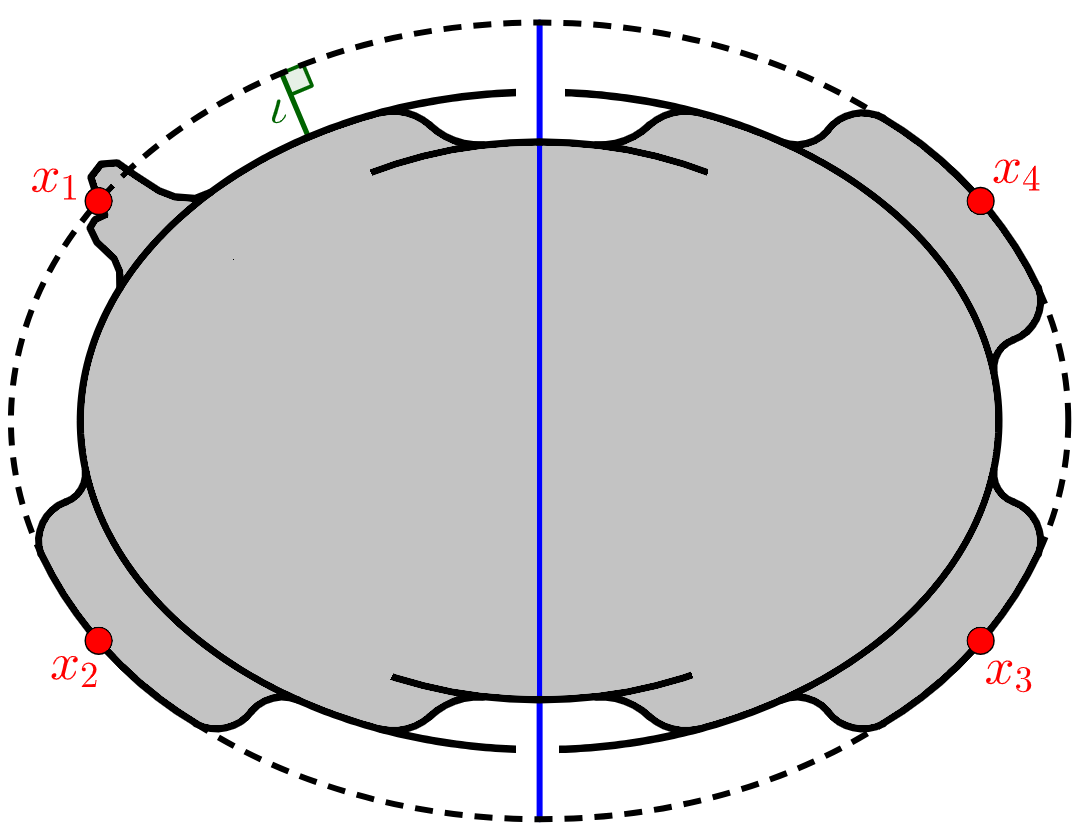}
				\caption{An exaggerated image of the proposed deformation (shaded in gray) with just one periodic orbit $\d \gamma = (x_1, x_2, x_3, x_4)$ (in red). Note that the deformation near the first point is different from the others and that the deformation is a homothetic ellipse near the minor axis. The original ellipse is dashed. For the purpose of illustration only, the shaded domain here is not convex. As convexity is a $C^0$ open property for the boundary curvature and the original ellipse is convex, the deformed domain would also be convex if the deformation is sufficiently small. The lines inside the shaded domain are also for visualization purposes only and do not constitute the boundary of the domain.}
				\label{fig:deformationpicture}
			\end{figure}
            \\
            \\
	At distances $C\xi$-away from $\d \Fm^1$ and $\xi$-away from the minor axis, we set $\mu_2(s) = -\iota$. We now analyze what happens along the minor axis. As we cannot apply Lemma \ref{lema63} for orbits whose reflection points are all near the minor axis, we propose the following: in a $\xi/2$-neighborhood of the minor axis, let $\Omega_{\bar{\eps}, \bar{\dt}}$ be a $1 - O(\iota)$-homothety of $\Omega_0$ with respect to the origin. This ensures that any orbit in $\Omega_{\bar{\eps}, \bar{\dt}}$ which is close to the former minor axis will be the negative dilation of an orbit in \red{the} ellipse. In particular, its length will have decreased to be strictly less than $L$. We now smoothly connect the deformation $\xi$-away from the minor axis to the deformation in its $\xi/2$-neighborhood, which resolves the aforementioned issue with the minor axis.
	\\
	\\
	Now consider a $C\xi$-neighborhood of $\d \Fm$. For every orbit $\gamma_u$ in our correctly selected family $\mathcal{F}_m$, we will prescribe a deformation near its reflection points as follows. Recall that the only possible stray orbits arise from rotations of $\gamma_u \in \mathcal{F}_m$ along a \red{rational} caustic. Since they all share the same rotation number (say $p/q$) with $\gamma_u$, it is easier to define the deformation $\mu$ in action angle coordinates $(\eta, I)$ (as opposed to arclength coordinates). This way, for each angle $\eta$ at which we sample $\mu$ in \eqref{eq612}, each action angle coordinate of the rotated orbit will be equidistant from the those of $\gamma_u$. Recall from Definition \ref{AAC} that the change of variables from rotational action angle to elliptic coordinates is given by $\Psi(\eta, I) = (\phi, \lambda)$. We need to ensure that for $\eta$ small, the first variation in Lemma \ref{lema63} satisfies
	\begin{equation}
		\sum_{i = 1}^{q} \Psi^* \mu \big(\eta_i(\gamma_u) + \eta \big) \Psi^* \cos \big( \theta(\eta_i(\gamma_u) + \eta \big) + o(\iota) \ne 0,
		\label{eq617}
	\end{equation}
    \brown{where $\Psi^*$ is the pullback by $\Psi$; since $\mu$ was originally defined in terms of the arclength parameter $s$, which is in turn a function of $\phi$, we write $\Psi^* \mu$ for $\Psi^* \mu(s(\cdot))$.}
    \\
    \\
	When working in action-angle coordinates, we introduce the parameter $\hat{\xi}$, corresponding to $\xi$, which again quantitatively separates the deformation $\mu$ into local and nonlocal zones. There exists some $\hat{\xi} = O(\xi)$ such that a $\hat{\xi}$-neighborhood in action-angle coordinates of the point $\d{\gamma_u}$ contains $\xi$-neighborhood in arclength coordinates and that all points $C\xi$-arclength away from $\d \gamma_u$ are at least $2\hat{\xi}$-action-angle away from $\d \gamma_u$, i.e.
	\begin{align*}
		\bigcup_{x_i \in \d \gamma} B_{\text{AL}}(x_i, \xi)  \subset& \bigcup_{x_i \in \d \gamma} B_{\text{AA}}(x_i, \hat{\xi}),\\
		\bigcup_{x_i \in \d \gamma} B_{\text{AA}}(x_i, 2 \hat{\xi}) \subset& \bigcup_{x_i \in \d \gamma} B_{\text{AL}}(x_i, C \xi),
	\end{align*}
	where $B_{\red{\text{AL}}}(x, r)$ is the ball of arclength radius $r$ centered at $x$ and $B_{\text{AA}}(x,r)$ is the same but with radius measured in the action angle coordinate $\eta$. We are then able to define $\mu_2(\eta)$ from $\hat{\xi}$ to $2\hat{\xi}$ away from reflection points. Inside of the ring, we have $\mu_2(\eta) = 0$ and outside, $\mu_2(\eta) = -\iota$.
	\\
	\\
    We now prescribe the nonlocal deformation $\mu_2$. It will be a sum of bump functions over all orbits $\gamma_u$ in $\Fm$. In what follows, we assume without loss of generality that $\iota$ is much smaller than the minimal distance between reflection points in $\d \Fm$.
	\\
	\\
	\textbf{Away from $\d \Fm^1$.}  Inside of a $\frac{5}{3} \hat{\xi}$ neighborhood of $\d \Fm \backslash \d \Fm^1$, we set $\mu_{2,u}$ to be identically $0$. Between $\frac{5}{3} \hat{\xi}$ and $2\hat{\xi}$, we smoothly interpolate $- \iota \leq \mu_2 \leq 0$. Since $\mu_1$ is supported only near $\d \Fm^1$, this choice of $\mu_2$ completely describes $\mu$ near $\d \Fm \backslash \d \Fm^1$ (see Figure \ref{gluing}).
	\\
	\\
	\textbf{Near $\d \Fm^1$.} Between $\frac{4}{3}\hat{\xi}$ and $2\hat{\xi}$ away from $\d \Fm^1$, we set $\Psi^*\mu_2(\eta)$ to be $-\iota$. This guarantees that as long as $|\mu_1|$ is sufficiently small, \eqref{eq617} will hold for $\frac{4}{3}\hat{\xi} \leq |\eta| \leq 2\hat{\xi}$; all terms in the sum will be negative and $\Psi^* \mu_2(\eta_1 + \eta)$ will differ from $-\iota$, which will dominate both the $\mu_1(\eta_1 + \eta)$ and $o(\iota)$ terms. Since Theorem \ref{th4} requires $\mu_2$ to be $0$ in a $\xi$-neighborhood of $\d \Fm$, we smoothly and nonpositively interpolate it from $0$ to $-\iota$ on the interval between $\hat{\xi}$ and $\frac{4}{3}\hat{\xi}$.
	\\
	\\
	We have now globally prescribed the nonlocal deformation $\mu_2$. It remains to show \eqref{eq617} when $|\eta| < \frac{4}{3} \hat{\xi}$. The values of $\mu_2$ at each reflection point in \eqref{eq617} will lie in a small neighborhood of $0$. Since $\mu_1$ is comparatively small, the remainder term $o(\iota)$ may dominate the expression. We need an alternative way to control stray orbits in a $\frac{4}{3}|\hat{\xi}|$ neighborhood of $\d \Fm$. It is \red{greatly} simplified by the fact that the local deformation $\mu_1$ is zero in a neighborhood of $\d \Fm \backslash \d \Fm^1$.	

	\begin{lemm}
		Assume that in a $\frac{4}{3}\hat{\xi}$ neighborhood of $\d \Fm^1$, \red{$\mu_1$ is chosen such that} the deformation $\mu = \red{\mu_1 +\mu_2}$ vanishes to first order \red{(i.e., $\mu=\mu' = 0$)} only when $\eta = \eta_1(\gamma_u)$. Then there are no stray orbits near $\gamma_u$.
	\end{lemm}
	
	\begin{proof}
		Assume to the contrary that there exist stray orbits of length $L$ in a $\frac{4}{3} \hat{\xi}$ neighborhood of the first reflection point in the perturbed domain. Let $\zeta$ be such an orbit which is close to $\gamma_u$, but not equal to it, and denote its first reflection point by $x_1(\zeta)$ (near $x_1(\gamma_u)$). Under the deformation, $x_1$ becomes
		\begin{equation}
			x_1(\zeta) = x_1^0 + \Vec{n}(x_1^0)\cdot \mu(s_1), 
		\end{equation}
		where $x_1^0 = x_{\Omega_0}(s_1)$ is a point on the starting ellipse $\Omega_0 = \Ee$. Fixing $x_1(\zeta)$ and $x_1^0$, we define two auxiliary length functionals:
		\begin{align*}
			\mathcal{L}_1(x_2, x_3, \ldots, x_q) =& ||x_2 - x_1(\zeta)|| + ||x_3 - x_2|| + \ldots + ||x_1(\zeta) - x_q||,\\
			\mathcal{L}_1^0(x_2, x_3, \ldots, x_q) =& ||x_2 - x_1^0|| + ||x_3 - x_2|| + \ldots + ||x_1^0 - x_q||,
		\end{align*}
		where $x_2, x_3, \ldots, x_q$ are points in the ellipse which are close to the respective reflection points of $\gamma_u$. They remain fixed under the deformation, which is identically zero there. 
		\\
		\\
		The functional $\mathcal{L}_1^0$ has a unique critical point which corresponds to \red{the} billiard orbit starting at $x_1^0$. The corresponding critical value $L$ is a global maximum. Nondegeneracy of this critical point is equivalent to transversality of $\nabla \mathcal{L}_1^0$ to the zero section, which is a $C^1$ open property. Since $\mathcal{L}_1$ is close to $\mathcal{L}_1^0$, it also has a unique nondegenerate critical point corresponding to a global maximum. If $x_1(\zeta)$ were outside of the ellipse, $\mathcal{L}_1$ would be strictly greater than $\mathcal{L}_1^0$ for any choice of $x_2, x_3, \ldots, x_q$; its unique critical value would be strictly greater than $L$. Similarly, if $x_1(\zeta)$ were inside the ellipse, the critical value would be strictly smaller.
		\\
		\\
		We conclude that $\mu(s_1) = 0$ \red{and hence, $\mu'(s_1) \ne 0$ by our assumption, so that} $x_1(\zeta) = x_1^0$. It follows that $\mathcal{L}_1$ coincides with $\mathcal{L}_1^0$ and the stray orbit in our deformation in fact coincides with an orbit in the starting ellipse. However, $\mu'(s_1) \ne 0$ and the tangent direction at $x_1^0$ is perturbed nontrivially. This means that the angles of incidence and reflection associated to $\zeta$ disagree at $x_1(\zeta)$, which is a contradiction.  
	\end{proof}
	Hence we have managed to control all stray orbits, completing the proof of Theorem \ref{th4}.
	
	\begin{rema}
		Assume $\mu_1 \equiv 0$. Since $\mu_2$ is always nonpositive, the resulting deformation will be contained inside of the starting ellipse. Moreover, the perturbation will remain convex, since the $C^2$ norm of $\mu$ is small. The perimeter of the deformed domain will be smaller than that of the ellipse $\Omega_0$. Assuming the local perturbation $\mu_1$ is sufficiently small, the decrease in perimeter under deformation $\mu = \mu_1 + \mu_2$ will continue to hold.
	\end{rema}

		\begin{figure}
				\centering
				\includegraphics[width=12cm]{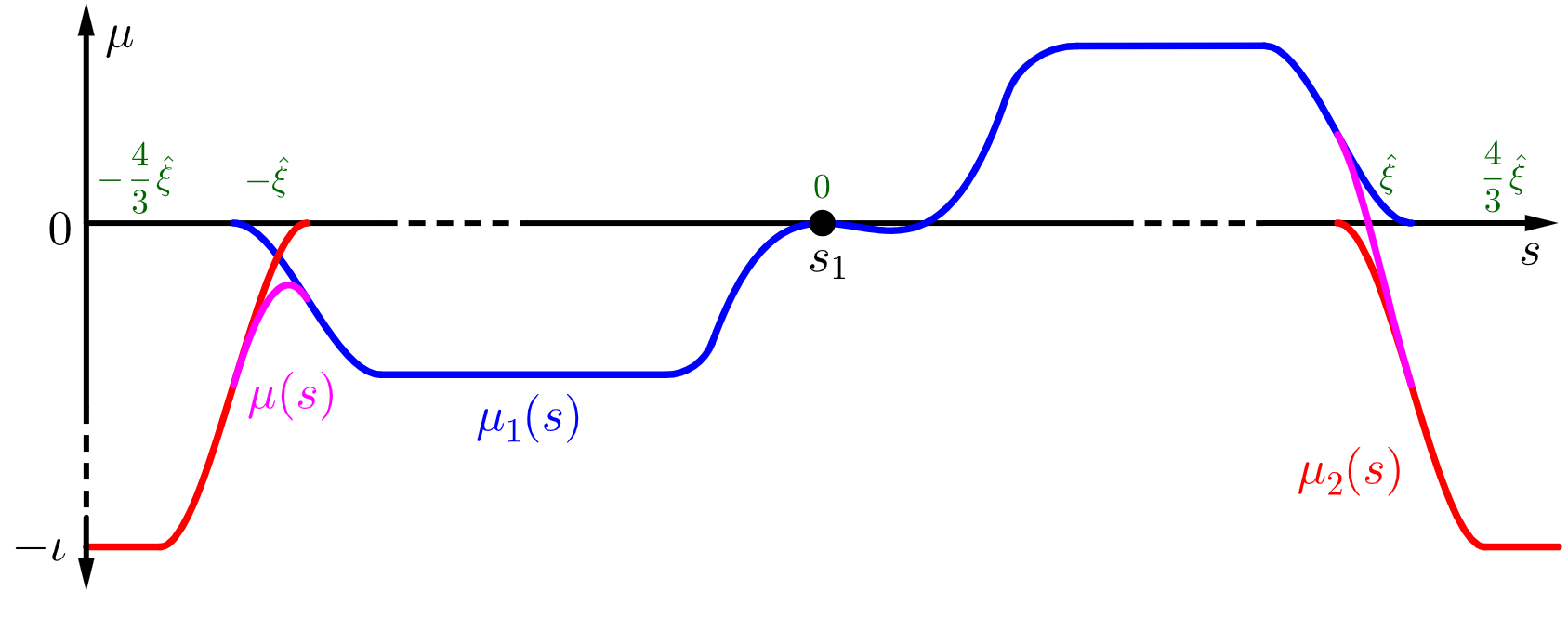}
				\caption{A portrait of the deformation near a first reflection point, with $s$ being the arclength coordinate of a point $x(s) \in \Ee$. The local part $\mu_1$ is in blue, the nonlocal part $\mu_2$ in red, and their sum, $\mu$, in pink. Green text denotes the action-angle coordinates of a point.}
				\label{gluing}
			\end{figure}

	\subsection{Proof of Theorem \ref{th3}}\label{LocalDeformation}
	We now turn our focus to the local deformation $\mu_1(s)$ near $\d \Fm^1$. Let $\gamma_u \in \d \Fm$ be one of our selected orbits. We first prescribe the jet of $\mu_1$ at $x_1(\gamma_u)$ (abbreviated by $x_1$ from now on), so that the deformed length functional will satisfy Condition $4$ of Definition \ref{def61}:
    \begin{align*}
        |\det \d^2 \LL(S_{\gamma_u})| = \eps_u, \quad \d^3 \LL( S_{\gamma_u}) = \dt_u.
    \end{align*}
    We then smoothly connect $\mu_1$ at the reflection point $x_1$ to the nonlocal deformation $\mu_2$.
    \\
    \\
    We begin with the following lemma.
	\begin{lemm}
		Let $\Omega_0 = \Ee$ be deformed by $\mu(s)\Vec{n}(s)$ in a neighborhood of $x_1(\gamma_u)$ for some orbit $\gamma_u \in \Fm$, where $\mu(s)$ is a smooth function satisfying $\mu(s_1) = \mu'(s_1) = 0$. We then have
		\begin{align}
            \begin{split}
                \det \partial^2 \mathcal{L}_{\Omega}(\gamma_u) =& C_1 (-1)^{q-1}\mu''(s_1),\\
                \partial_{\deg}^3 \mathcal{L}_{\Omega}(\gamma_u) =& C_2\mu''(s_1) + C_3\mu'''(s_1) + C_4 \mu''(s_1)^2,
            \end{split}
            \label{eq621eq}
		\end{align}
		where $C_1$, $C_2$, $C_3$ and $C_4$ depend only on \brown{$\Omega_0$ and $\gamma_u$, with $C_1, C_3 > 0$.}
  \label{eq621}
	\end{lemm}
	\begin{proof}
		Consider the Hessian $\d^2 \LL_{\Omega}$. First, note that it depends on $\mu$ only through the entry $a_1$ in \eqref{eq37}, owing to the factor $\kappa_1$ (curvature at the first reflection point $x_1$). \brown{Note that the deformation $\mu$ preserves tangency at the reflection points, so the terms $\cos \theta$ in \eqref{eq37} are independent of $\mu$.}
        All other entries depend either on the constant $1$-jet or on curvatures at undeformed reflection points only. We use the following formula for $\kappa$:
		\begin{equation}
			\kappa = \frac{\det(\alpha', \alpha'')}{||\alpha'||^3},
			\label{eq622}
		\end{equation}
		where $\alpha$ is any parametrization of $\d \Omega$. The arclength coordinate on the ellipse is given by
        \begin{align*}
            s = \int_0^\phi \sqrt{\sin^2 \vartheta + (1- e^2) \cos^2 \vartheta} d\vartheta,
        \end{align*}
        an elliptic integral of the second kind, which parametrizes $\Omega$ in \eqref{eq62}. \brown{For an arclength parametrization $\alpha_{\Omega_0}(s)$ of $\Omega_0$, we have}
		\begin{align*}
			\alpha_{\Omega}''(s_1) = \alpha_{\Omega_0}''(s_1) + \mu''(s_1) \Vec{n}(s_1),
		\end{align*}
		\red{{since $\mu(s_1) = \mu'(s_1) = 0$}}. Hence, $\kappa_1$ decreases linearly in $\mu''(s_1)$ and $a_1$ increases linearly in $\mu''(s_1)$. Since the determinant is linear in $a_1$, it also depends linearly on $\mu''$. To find the sign of $C_1$, we compute the sign of the principal minor of $\partial^2 \mathcal{L}$ associated to $a_1$. By the eigenvalue interlacing theorem, all eigenvalues of that submatrix are nonpositive. Moreover, they are distinct from $0$, since $h_1 \ne 0$; \red{we remind the reader that $h$ was an eigenvector in the kernel direction of the Hessian, (see \eqref{happear})}. This contributes a factor of $(-1)^{q-1}$ in the principal minor associated to $a_1$ and also in \red{\eqref{eq621eq}}.
		\\
        \\
		Now consider the third derivatives \brown{of $\LL_\Omega$. Recall that $a_i$ and $b_i$ are the diagonal and off diagonal entries respectively of the Hessian of the length functional $\LL$ (see \eqref{eq37} and Proposition \ref{DetPoincare}). As they are second derivatives, we differentiate them only once and recall that $\theta'(s) = \kappa(s)$.} This produces many derivatives of $\cos \theta_i$ and $|x(s_i)-x(s_{i+1})|$. Let us first consider the terms $b_i$, which depend only on $\theta$ and not $\kappa$ directly. Their derivatives are all linear in terms of \brown{the undifferentiated} curvature and thus only contribute to $C_2$. \brown{The coefficients $C_3$ and $C_4$ are unaffected since the curvature depends linearly on $\mu''$.} This reduces our study to
		\begin{equation}
			\frac{\partial a_1}{\partial s_1} =  2 \frac{\d}{\d s} \left[ \cos \vartheta_i \left(\frac{\cos \vartheta_i}{|x(s_i) - x(s_{i-1})|}+ \frac{\cos \vartheta_i}{ |x(s_i) - x(s_{i+1})|} - 2\kappa_i\right)\right],
			\label{eq624}
		\end{equation}
		and specifically the term $-4\cos \theta_1 \kappa_1$, which has maximal differential degree.\footnote{The differential degree of a monomial consisting of some function $f$ and its derivatives $f^{(i)}$ is defined by $\d\text{-deg} \left(\prod_{i = 0}^n \left(f^{(i)} \right)^{p_i} \right) = \sum_{i = 0}^n i p_i.$} Since
        $$
        \frac{d}{d s_1} \cos \theta_1 = - \sin \theta_1 \kappa_1
        $$
        is linear in the curvature, it will contribute to $C_2$ and $C_4$. We are therefore only interested in $\kappa_1'$. Instead of differentiating with respect to $\Omega$ arclength, we differentiate \eqref{eq622} with respect to $\Omega_0$ coordinates. This yields a positive multiple of the former, since locally, the deformation is given by $x \mapsto x + O(x^2)$, which has derivative $1$ at the origin.
		\\
        \\
		We obtain terms which are quadratic in $\mu''$ and contribute to $C_4$. We also obtain $\alpha'''$, which features terms of the form $\Vec{n}(s)\mu'''(s)$. Along with $\det(\alpha', \alpha''')$, the value $\kappa_1'$ is linearly decreasing in $\mu'''$, \brown{which can be seen either geometrically or from the formula \eqref{eq622}.} It follows that \eqref{eq624} and hence $\partial_{\deg}^3 \mathcal{L}_{\Omega}(\gamma_u)$ both increase linearly in $\mu'''(s_1)$. Together, these computations result in the expansion \eqref{eq621eq}.
	\end{proof}
    
	We now translate Lemma \ref{eq621} into a set of conditions on a controllable deformation:
	\begin{equation}
		\mu''(s_1) = \pm C_1^{-1} \eps_u, \; \; \mu'''(s_1) = C_3^{-1} \left( \dt_u \mp C_2C_1^{-1}\eps_u - C_4C_1^{-2}\eps_u^2 \right).
		\label{eq625}
	\end{equation}
	In particular, the signs of $\mu''$ and $\mu'''$ remain constant for every deformation in the family, since $\dt_u > \red{\sqrt[3]{\eps_u}}$. 
	\\
    \\
	We set
	\begin{equation}
		\mu_1(s) = \frac{\mu''(s_1)}{2} (s-s_1)^2 + \frac{\mu'''(s_1)}{6} (s-s_1)^3  
		\label{eq626}
	\end{equation} 
	in a small neighborhood $U$ of $s_1$ (of size between $\iota$ and \red{$\xi$}). In view of \eqref{eq625}, we see that $\mu'''(s_1)$ dominates $\mu''(s_1)$. Hence, $\mu_1$ shares a common sign with $\mu'''(s_1)(s-s_1)^3$ in $U$. Observe also that in this neighborhood, $\mu(s) = \mu'(s) = 0$ holds only for $s = s_1$, since \eqref{eq626} produces a polynomial of third degree which cannot have degenerate zeros at two different points.
	\\
    \\
	Outside of the neighborhood $U$, we smoothly connect this $\mu$-polynomial to a fixed \red{number} which is equal to $\max(\eps_0, \dt_0)$ in absolute value but shares its sign with $\mu'''(s_1)(s-s_1)^3$ and \red{is such that the connection} avoids crossing \brown{through} zero. Past this junction, $\mu_1$ is independent of $\eps_1$ and $\dt_1$ and remains constant until being, once again, smoothly connected to $0$. This second connection takes place outside of a $\hat{\xi}$-neighborhood of $x_1$ in action-angle coordinates, so that $\mu_1$ and $\mu_2$ never vanish simultaneously. We can choose it in such a way that $\mu_1 + \mu_2$ does not vanish to order $\geq 2$ at any point. This finalizes the construction of controllable families of deformations and hence completes the proof of Theorem \ref{th3}.

	\section{Proof of Theorem \ref{mainth}}\label{Solving the main equation}
	In Section \ref{Destroying other orbits}, we constructed \brown{multiscale families of} controllable deformations, depending on parameters $\bar \eps, \bar \dt$, near any ellipse with an admissible eccentricity according to Theorem \ref{Eccentricities}. \brown{The parameters $\bar \eps$ controlled the size of the Hessian of the length functional $\LL$ at correctly selected orbits and $\bar \dt$ measured the size of its third derivatives.} Now \textit{for any} controllable family, we solve for parameters $\bar{\eps}$ and $\bar{\dt}$ at which the corresponding domain $\Omega_{\bar{\eps}, \bar{\dt}}$ has cancellations of the Balian-Bloch-Zelditch wave invariants $B_{\gamma, j}$ \brown{to order $m$}. To do so, we express \eqref{eq325} in terms of $\bar{\eps}, \bar{\dt}$ and solve the resulting system of equations \eqref{eq15}. 
	\\
    \\
	These equations feature the error terms $\mathcal R_j$, coming from \eqref{eq325}. Since the computation of their exact structure is intractable and they are higher order in $\eps$, we first solve the linearized system and ignore the error terms in Section \ref{Linearized}. This will provide parameters $\bar{\eps}_{\text{pre}} = (\eps_{1, \text{pre}},\cdots, \eps_{4m, \text{pre}})$ and $\bar{\dt}_{\text{pre}} = (\dt_{1, \text{pre}}, \cdots, \dt_{4m, \text{pre}})$ for which the leading order terms in \eqref{eq15} cancel. We will then apply a fixed point theorem in Section \ref{Dealing with the errors}, which plays the role of a generalized inverse function theorem and guarantees the existence of nearby deformation parameters $\bar{\eps}_{\text{fin}}$ and $\bar{\dt}_{\text{fin}}$ for which the error terms are also cancelled. 

    \subsection{Parameter matching in the linearized system}\label{Linearized}
 
	We now find parameters $\bar{\eps}_{\text{pre}}$ and $\bar{\dt}_{\text{pre}}$ to solve the system of equations \eqref{eq15}, temporarily ignoring the remainders $\mathcal{R}_{\gamma, j}$. \brown{Recall from Definition \ref{def61} (part (4)) of Section \ref{Destroying other orbits} that $|\det \partial^2 \mathcal L(\gamma_u)| = \eps_u$ and $\partial_{\deg}^3 \mathcal L = \dt_u$, where $\d_{\text{deg}}$ is given by equation \eqref{eq45}.} \brown{In view of Theorem \ref{partone},} the leading order term in \brown{the $\eps$ expansion of} $\mathcal D_{\gamma_u} B_{\gamma_u, j}$ can be expressed as
	\begin{align}
    \label{eq:nuj}
		\mathcal D_{\gamma_u} B_{\gamma_u, j} = e^{ikL} i^{n_{u, j}} \frac{C_u}{w(j)} \dt_u^{2j} \eps_u^{-3j- 1/2} + \mathcal{D}_{\gamma_u} \mathcal{R}_{\gamma_u, j},
	\end{align}
	where the complex phases $n_{u, j}  \in \mathbb Z$ are given by \eqref{eq:phases mod 4} and $C_u > 0$ are \brown{constants which are} independent of the deformation. Recalling the phases of different orbits in equation \eqref{eq41}, \brown{cancellation of the BBZ wave invariants $B_{\gamma, j}$ in \eqref{eq15}, modulo the remainders $\mathcal{R}_{\gamma, j}$, is equivalent to} the following reduced system of equations:
	\begin{equation}
		\begin{cases}
			\sum_{u = 1}^m C_u \dt_u^{2j} \eps_u^{-3j- 1/2} = \sum_{u = 2m+1}^{3m} C_u \dt_u^{2j} \eps_u^{-3j- 1/2}, & 0 \leq j \leq m,\\
			\sum_{u = m+1}^{2m} C_u \dt_u^{2j} \eps_u^{-3j- 1/2} = \sum_{u = 3m+1}^{4m} C_u \dt_u^{2j} \eps_u^{-3j- 1/2}, & 0 \leq j \leq m.
		\end{cases}
		\label{eq72}
	\end{equation}
    \red{Note that we have $8m$ free variables ($4m$ of both $\e_u$-s and $\dt_u$-s), but we only need to satisfy $2(m+1)$ equations. Thus, the system is not underdetermined and there is hope that it can be solved. Generally, as $m$ grows, one would likely have to consider more periodic orbits to get a cancellation, which is the main difficulty in extending our proof to $C^\infty$ cancellations.}
    
	\begin{rem}
	    Trivially, one can always take $\bar{\dt} = 0$ (i.e. $\dt_u = 0$ for each $u$) as a solution of the system. This indeed satisfies \eqref{eq72}, but is not amenable to solving the perturbed system with error terms $\mathcal{R}_{\gamma, j}$. Since we will eliminate the remainders by a small perturbation of the parameters, we want the linearized system \eqref{eq72} (the Jacobian) to be nondegenerate in $\bar{\eps}$ and $\bar{\dt}$. This clearly does not hold for the trivial solution $\bar{\dt} = 0$.
    \end{rem}
	\\
    \\
	\brown{Instead, we propose the following reduction in order to solve the system \eqref{eq72}. Let $\eps_{\text{com}}$ be a small parameter, which will be common to all $\eps_{u}$.} Now set $\eps_{u, \text{pre}} = C_u^2 \cdot \eps_{\text{com}}$. Equation \eqref{eq72} will then be automatically satisfied for $j=0$. Moreover, since $\mathcal R_0 = 0$ \brown{(cf. Theorem \ref{partone})}, this will even satisfy \eqref{eq15} for $j=0$. The remaining equations in \eqref{eq72} are now reduced to
	\begin{equation}
		\begin{cases}
			\sum_{u = 1}^m C_u^{-6j} \dt_u^{2j}\;\;\; \;= \sum_{u = 2m+1}^{3m} C_u^{-6j} \dt_u^{2j}, \; \; \; j = 1, 2, \ldots, m\\
			\sum_{u = m+1}^{2m} C_u^{-6j} \dt_u^{2j} = \sum_{u = 3m+1}^{4m} C_u^{-6j} \dt_u^{2j}, \; \; \; j = 1, 2, \ldots, m.
		\end{cases}
		\label{eq74}
	\end{equation}
	\brown{For $\wh \dt> 0$ sufficiently small, we now define a map $\wh F_0: (0,\wh \dt)^{4m} \to \R^{2m}$ by
    \begin{align}
        \label{eq: F0 hat}
        \wh F_0(\bar \dt) =
        \begin{pmatrix}
            \sum_{u = 1}^m C_u^{-6\times 1} \dt_u^{2\times 1} - \sum_{u = 2m+1}^{3m} C_u^{-6\times 1} \dt_u^{2 \times 1}
            \\
            \vdots
            \\
            \sum_{u = 1}^m C_u^{-6 \times m} \dt_u^{2 \times m} - \sum_{u = 2m+1}^{3m} C_u^{-6\times m} \dt_u^{2\times m}
            \\
			\sum_{u = m+1}^{2m} C_u^{-6\times 1} \dt_u^{2\times 1} - \sum_{u = 3m+1}^{4m} C_u^{-6\times 1} \dt_u^{2\times 1}
            \\
            \vdots
            \\
            \sum_{u = m+1}^{2m} C_u^{-6\times m} \dt_u^{2\times m} - \sum_{u = 3m+1}^{4m} C_u^{-6\times m} \dt_u^{2\times m}
        \end{pmatrix}.
    \end{align}
    }
    We choose
	\begin{equation}
		\dt_{vm + u, \text{pre}} = u \cdot C_{vm+u}^3 \cdot \dt_{\text{com}}, \quad 0 \leq v \leq 3, \quad  1 \leq u \leq m,
		\label{eq75}
	\end{equation}
	\brown{where $\dt_{\text{com}}$ is} another small  positive common factor. The choice of parameters \eqref{eq75} satisfies the reduced system \eqref{eq74}. Furthermore, note that for the selected values of $\left\{ \dt_u \right\}_{u = 1}^{2m}$, the Jacobian of \eqref{eq74} can be expressed in terms of a Vandermonde matrix and is thus of full rank -- $2m$ (this is the reason for the factor $u$ in \eqref{eq75}). \red{Note that we only consider $2m$ variables: we have $2m$ equations and we want to get a square Jacobian matrix; the other variables will not change.} In fact, the Jacobian matrix takes the following form:
    \begin{align}
       D_{\dt_1, \cdots, \dt_{2m}}\wh F_0 (\bar \delta_{\text{pre}}) =& \begin{pmatrix}
            V_1 & 0 \\
            0 & V_2
        \end{pmatrix},
        \label{eq.Jacobianwh}
    \end{align}
    where $V_1$ and $V_2$ are $m\times m$ matrices defined by
    \begin{align*}
        V_1 = \left\{ C_u^{-3} \times 2j \cdot \delta_{\text{com}}^{2j-1} \times u^{2j-1}\right\}_{1 \le j \le m, \; 1 \le u \le m}
    \end{align*}
    and 
    \begin{align*}
        V_2 = \left\{ C_{m+u}^{-3} \times 2j \cdot \delta_{\text{com}}^{2j-1} \times u^{2j-1}\right\}_{1 \le j \le m, \; 1 \le u \le m}.
    \end{align*}
    Note that the first two factors are \brown{nonzero and} common \brown{across either the} rows or columns of the Jacobian, so they don't affect \brown{invertibility.} Meanwhile, the factor $u^{2j-1}$ provides a Vandermonde structure \brown{for the matrix}, from which we see that it indeed has full rank.
    \\
    \\
    As we have just demonstrated nondegeneracy \brown{(in particular, full rank)} of the Jacobian using only the first $2m$ parameters $\delta_u$, we are left with $2m$ additional free parameters ($\dt_u$, for $2m+1 \leq u \leq 4m$). \brown{We now fix, once and for all, a set of final values $\eps_{u, \text{fin}}$ for $u = 1, \cdots, 4m$ and $\dt_{u, \text{fin}}$ for $u = 2m+1, \cdots, 4m$. They are given by}
	\begin{equation}\label{last2m}
		\bar{\eps}_{\text{fin}} = \bar{\eps}_{\text{pre}}, \; \;\;\;\;\; \dt_{u, \text{fin}} = \dt_{u, \text{pre}}, \; u = 2m+1, \ldots, 4m.
	\end{equation}
    \brown{In view of the full rank dependence of the system \eqref{eq74} on the first $2m$ $\delta_{u}$-s, the choice \eqref{last2m} has no impact on the solvability of \eqref{eq15}.}
 
	\subsection{Solution of the nonlinear system}\label{Dealing with the errors}
	We now consider contributions of the remainder terms $\mathcal{R}_{\gamma, j}$ \brown{to the system \eqref{eq15}} and choose final values $\dt_{1, \text{fin}}, \ldots, \dt_{2m, \text{fin}}$ which produce full cancellations amongst all $B_{\gamma_u, j}$, $1 \leq j \leq m$. \brown{Recall that since $\mathcal{R}_{\gamma, 0} = 0$, the $B_{\gamma_u, 0}$ terms cancel for any choice of $\bar \delta$.} We \brown{now} use a quantitative version of the inverse function theorem to find local solutions of the nonlinear system from invertibility of the linearized one. More specifically, we \brown{will} use Brouwer's fixed point theorem.
    
    \begin{theo}[\textbf{Brouwer's fixed point theorem}]
        Every continuous function from a nonempty, convex, compact subset $K$ of a Euclidean space to $K$ itself has a fixed point.
        \label{thm.brou}
    \end{theo}
    
    \brown{From here on, we will denote by $\wt \dt$ the truncation of $\bar{\dt}$ to its first $2m$ components, as we have already fixed} the remaining $2m$ \brown{values} in \eqref{last2m}. \brown{Recall our formula for the symplectic prefactor $\mathcal{D}_\gamma(\lambda)$ in \eqref{eq: symplectic prefactor}. Dividing by the oscillatory factor $e^{i \lambda L_\gamma}$, we are left with only dynamical data:
    \begin{align*}
        \mathcal{D}_\gamma(\lambda) e^{- i \lambda L_\gamma} =  \frac{ e^{\frac{i\pi \sgn \d^2 \LL(S_\gamma)}{4}}}{\sqrt{|\det \d^2 \LL(S_\gamma)|}}
    \end{align*}
    As all orbits $\gamma_1, \cdots, \gamma_{4m}$ share the same length, we will continue to denote their common length by $L = L_\gamma$.
    }
    Let $\dt_0$ be such that our deformation is well defined for $\dt_{u, \text{fin}} < \dt_0$, $1 \leq u \leq 4m$.
    \\
    \\
    Now define a map $F: (0,\dt_0)^{2m} \to \R^m \times \R^m$ by setting $F(\dt_{1, \text{fin}}, \cdots, \dt_{2m, \text{fin}})$ equal to
	\begin{align*}
		\sum_{\gamma: \text{length}(\gamma) = L} \bigg(&\Re \red{(\mathcal D_\gamma(\lambda) e^{- i \lambda L} B_{\gamma, 1})}, \cdots, \Re \red{(\mathcal D_\gamma(\lambda) e^{- i \lambda L} B_{\gamma,m})},\\
        &\Im \red{(\mathcal D_\gamma(\lambda) e^{- i \lambda L}B_{\gamma,1})}, \cdots, \Im \red{(\mathcal D_\gamma(\lambda) e^{- i \lambda L} B_{\gamma, m})}\bigg)
	\end{align*}
    \red{for $\lambda > 0$, where $B_{\gamma, j}$ are the Balian-Bloch-Zelditch invariants expressed in Theorem \ref{partone}.}
	Clearly, a solution to $F(\dt_{1,\text{fin}}, \cdots, \dt_{2m, \text{fin}}) = 0$ would provide cancellations in \eqref{eq15}. As before, $F$ consists of a main term and \red{an} error term:
	\begin{align}\label{F0 plus Ferror}
		F(\wt \delta) = F_0(\wt \delta) + F_{\text{err}}(\wt \delta),
	\end{align}
    where
    \begin{align*}
        F_{\text{err}}(\wt \dt) = \sum_{\gamma: \text{length}(\gamma) = L} \bigg(&\Re \red{(\mathcal D_\gamma(\lambda) e^{- i \lambda L} \mathcal{R}_{\gamma,1}}), \cdots, \Re \red{(\mathcal D_\gamma(\lambda) e^{- i \lambda L} \mathcal{R}_{\gamma, m})},\\
        &\Im \red{(\mathcal D_\gamma(\lambda) e^{- i \lambda L}\mathcal{R}_{\gamma, 1})}, \cdots, \Im \red{(\mathcal D_\gamma(\lambda) e^{- i \lambda L} \mathcal{R}_{\gamma,m})}\bigg).
    \end{align*}
	We now want to relate the maps $F_0$ and $\wh F_0$. Fixing the last $2m$ variables $\dt_{2m + 1} = \dt_{2m+1, \text{fin}}, \cdots, \dt_{4m} = \dt_{4m, \text{fin}}$, we have that $F_0$ and $\wh F_0$ agree up to
    \begin{itemize}
        \item possibly swapping the entries $(\wh F_0)_j$ and $(\wh F_0)_{m+j}$ and
        \item multiplying the entries $(\wh F_0)_j$ and $(\wh F_0)_{m+j}$ by
        $$
        \pm\frac{ \eps_{\text{com}}^{-3j - 1/2}}{w(j)},
        $$
        to match the structure in \eqref{eq:nuj}.
    \end{itemize}
    For example, if for some $1 \leq j \leq m$, the exponents $n_{u,j}$ in \eqref{eq:nuj} are odd (for hyperbolic orbits $\gamma_1, \cdots, \gamma_m$), so that $i^{n_{u, j}}$ is a multiple of $i$, we swap the entries $(\wh F_0)_j$ and $(\wh F_0)_{m+j}$ to interchange real and imaginary parts. As row swapping and multiplication are invertible, invertibility of $D_{\dt_1, \cdots, \dt_{2m}}\wh F_0$ is equivalent to that of the full Jacobian $D F_0$.
    \\
    \\
    Taylor expanding $F_0$ at $\wt{\dt} = \wt{\dt}_{\text{pre}} = (\dt_{1, \text{pre}}, \cdots, \dt_{2m, \text{pre}})$, we obtain
	\begin{align*}
		F_0(\wt{\dt}) = D F_0 (\wt \delta_{\text{pre}}) \cdot (\wt{\dt} - \wt{\dt}_{\text{pre}}) + F_{0, \text{quad}} (\wt{\dt}).
	\end{align*}
	\red{Recalling the factor of $\e$ in \eqref{eq:nuj}, we have}
    \begin{align}
        \left\{F_{0, \text{quad}} (\wt{\dt})\right\}_{\hat{\jmath}} = O\left(\red{\e_{\text{com}}^{-3j-1/2}}\cdot |\wt{\dt} - \wt{\dt}_{\text{pre}}|^2\right),
    \label{quadestimate}
    \end{align}
    where $j$ is the residue of $\hat{\jmath}$ modulo $m$. \brown{In view of \eqref{F0 plus Ferror},} the equation $F(\wt{\delta}) = 0$ can be rewritten as
	\begin{align*}
		D F_0 (\wt \delta_\text{pre}) \cdot (\wt{\dt} - \wt{\dt}_{\text{pre}}) = -F_{0, \text{quad}}(\wt \dt) - F_{\text{err}}(\wt \dt),
	\end{align*}
	or alternatively, using \brown{the invertibility of $D F_0 (\wt \dt_{\text{pre}})$ (arising from the Vandermonde property of $V_1, V_2$ above), we have}
	\begin{align}\label{Iterate}
		\wt \dt = \wt{\dt}_{\text{pre}} - \left(D F_0 (\wt \dt_{\text{pre}})\right)^{-1} \left(F_{0, \text{quad}}(\wt \dt) + F_{\text{err}}(\wt \dt)\right).
	\end{align}
	This defines a map on $\wt \dt$ which we can iterate. Any fixed point of \eqref{Iterate} corresponds to a desired solution of $F = 0$. To find such a fixed point, we will apply Theorem \ref{thm.brou}. Let $K \subset \mathbb{R}^{2m}$ be a small cube around $\wt{\dt}_{\text{pre}}$, given by $K = \wt \dt_{\text{pre}} + [-\sigma, \sigma]^{2m}$ for some $\sigma > 0$. $K$ is convex, so we just need to check that \eqref{Iterate} maps it into itself.  
    \\
    \\
    For this, we need to estimate the second term in \eqref{Iterate}. \red{The relationship between $F_0$ and $\wh F_0$, together with the exact form \eqref{eq.Jacobianwh} of the Jacobian of $\wh F_0$, allows us to invert $DF_0(\wt{\dt}_{\text{pre}})$:
    \begin{equation*}
        \left\{\left(D F_0 (\wt \dt_{\text{pre}})^T\right)^{-1}\right\}_{\hat{\jmath}, \hat{u}} =  \e_{\text{com}}^{3j+1/2} \cdot \delta_{\text{com}}^{-2j+1} \cdot \wh{C}_{\hat{\jmath}, \hat{u}},
    \end{equation*}
    where $\wh{C}_{\hat{\jmath}, \hat{u}}$ are constants that arise from inversion of the Vandermonde matrix. They depend on the choice of $\gamma_u$-s but are independent of the deformation. The value $F_{0, \text{quad}}(\wt{\dt})$ was estimated in \eqref{quadestimate} and the $\hat{\jmath}$-component of $F_{\text{err}}(\wt{\dt})$ is $O(\e^{-3j+1/2})$, \brown{since the terms} $\mathcal{R}_{\gamma, j}$ are bounded in Theorem \ref{partone}.
    \\
    \\
    Combining all of this together with $\left|\wt{\dt} - \wt{\dt}_{\text{pre}}\right| = O(\sigma)$, we see that the term
    \begin{align*}
        \left(D F_0 (\wt \dt_{\text{pre}})\right)^{-1} \left(F_{0, \text{quad}}(\wt \dt) + F_{\text{err}}(\wt \dt)\right)
    \end{align*}
    in \eqref{Iterate} is bounded by
	\begin{align*}
	    C\dt_{\text{com}}^{-2m+1}(\sigma^2 + \e_{\text{com}}).
	\end{align*}
    As long as this is less than $\sigma$, the cube $K$ is invariant under iteration of the map \eqref{Iterate} and Theorem \ref{thm.brou} applies. To ensure this, we can choose $\sigma = \e_{\text{com}}^{2/3}$ and set $\delta_{\text{com}} = \e_{\text{com}}^{1/(10m)}$ to be much larger than $\e_{\text{com}}$. Note that this choice agrees with the restrictions we put on $\e$ and $\dt$ in Definition \ref{def61}. Thus, there is a fixed point $\wt \dt_{\text{fin}}$ in $K$, which is associated to the values $\bar{\e}_{\text{fin}}$ and $\bar{\dt}_{\text{fin}}$. The corresponding domain (for any controllable family from Section \ref{Destroying other orbits}) given by these values will exhibit the required cancellations.} This finishes the proof of Theorem \ref{mainth}.
	
	\section{Acknowledgments} The first author acknowledges the support of ERC Grant \#885707. The second author is grateful to both the Institute of Science and Technology Austria and the Erwin Schr\"odinger Institute for hosting him in 2022 and 2023. We would especially like to thank Vadim Kaloshin for sharing several key ideas and participating in much of our early work towards producing cancellations. This project began in 2019 and involved many group discussions with Hamid Hezari, Vadim Kaloshin, and Steve Zelditch. It persisted throughout the COVID-19 pandemic, the war in Ukraine, and the untimely passing of our dear friend and mentor, Steve Zelditch, to whom we dedicated the work \cite{KKV}.
	
	\medskip
	
	\bibliographystyle{alpha}
	\bibliography{MB1}

@article{HezariZelditchEigenfunctionAsymptotics,
  title={Eigenfunction asymptotics and spectral rigidity of the ellipse},
  author={Hezari, Hamid and Zelditch, Steve},
  journal={Journal of Spectral Theory},
  volume={12},
  number={1},
  pages={23--52},
  year={2022}
}

@article{birkhoff1927periodic,
  title={On the periodic motions of dynamical systems},
  author={Birkhoff, George D},
  journal={Acta Mathematica},
  volume={50},
  number={1},
  pages={359--379},
  year={1927},
  publisher={Springer}
}

@article{GLP25,
  title={Marked length spectrum rigidity for Anosov surfaces},
  author={Guillarmou, Colin and Lefeuvre, Thibault and Paternain, Gabriel P},
  journal={Duke Mathematical Journal},
  volume={174},
  number={1},
  pages={131--157},
  year={2025},
  publisher={Duke University Press}
}

@article{Poritsky,
 ISSN = {0003486X, 19398980},
 URL = {http://www.jstor.org/stable/1969334},
 author = {Hillel Poritsky},
 journal = {Annals of Mathematics},
 number = {2},
 pages = {446--470},
 publisher = {[Annals of Mathematics, Trustees of Princeton University on Behalf of the Annals of Mathematics, Mathematics Department, Princeton University]},
 title = {The Billard Ball Problem on a Table With a Convex Boundary--An Illustrative Dynamical Problem},
 urldate = {2025-10-30},
 volume = {51},
 year = {1950}
}

@article{HezariZelditchCentrallySymmetric,
  title={Centrally symmetric analytic plane domains are spectrally determined in this class},
  author={Hezari, Hamid and Zelditch, Steve},
  journal={Transactions of the American Mathematical Society},
  volume={376},
  number={11},
  pages={7521--7553},
  year={2023}
}

@article{Watanabe2,
  title={Plane domains which are spectrally determined II},
  author={Watanabe, Kohtaro},
  journal={Journal of Inequalities and Applications},
  volume={2002},
  number={1},
  pages={613103},
  publisher={Springer}
}

@article{Watanabe1,
  title={Plane domains which are spectrally determined},
  author={Watanabe, Kohtaro},
  journal={Annals of Global Analysis and Geometry},
  volume={18},
  pages={447--475},
  year={2000},
  publisher={Springer}
}

@incollection{HormanderSpectralFunctionofanEllipticOperator,
  title={The spectral function of an elliptic operator},
  author={H{\"o}rmander, Lars},
  booktitle={Mathematics Past and Present Fourier Integral Operators},
  pages={217--242},
  year={1968},
  publisher={Springer}
}

@article{GuilleminWTI,
  title={Wave-trace invariants},
  author={Guillemin, Victor},
  journal={Duke Math. J.},
  volume={85},
  number={1},
  pages={287--352},
  year={1996}
}

@article{GuilleminWTITZ,
  title={Wave-trace invariants and a theorem of Zelditch},
  author={Guillemin, Victor},
  journal={International Mathematics Research Notices},
  volume={1993},
  number={12},
  pages={303--308},
  year={1993},
  publisher={Hindawi Publishing Corporation}
}

@article{ZelditchWTI,
  title={Wave invariants at elliptic closed geodesics},
  author={Zelditch, S},
  journal={Geometric \& Functional Analysis GAFA},
  volume={7},
  pages={145--213},
  year={1997},
  publisher={Springer}
}

@article{kov21,
	author = {Koval, Illya},
	date = {2025/12/11},
	doi = {10.1007/s00222-025-01397-y},
	id = {Koval2025},
	isbn = {1432-1297},
	journal = {Inventiones mathematicae},
	title = {Local strong Birkhoff conjecture and local spectral uniqueness of almost every ellipse},
	url = {https://doi.org/10.1007/s00222-025-01397-y},
	year = {2025}}

@article{Carminati,
title = {On the regularity of Mather's $\beta$-function for standard-like twist maps},
journal = {Advances in Mathematics},
volume = {377},
pages = {107460},
year = {2021},
issn = {0001-8708},
doi = {https://doi.org/10.1016/j.aim.2020.107460},
url = {https://www.sciencedirect.com/science/article/pii/S0001870820304886},
author = {Carlo Carminati and Stefano Marmi and David Sauzin and Alfonso Sorrentino}
}

@article{BialyMironov,
  title={The Birkhoff-Poritsky conjecture for centrally-symmetric billiard tables},
  author={Bialy, Misha and Mironov, Andrey E},
  journal={Annals of Mathematics},
  volume={196},
  number={1},
  pages={389--413},
  year={2022},
  publisher={Department of Mathematics, Princeton University Princeton, New Jersey, USA}
}

@article{KKV,
	author = {Vadim Kaloshin and Illya Koval and Amir Vig},
	journal = {to appear in Contemporary Math., Volume in Honor of Steve Zelditch},
	title = {Balian-Bloch-Zelditch Wave Invariants for Nearly Degenerate Orbits},
	year = {2024}}

@article{Callis22,
	adsurl = {https://ui.adsabs.harvard.edu/abs/2022arXiv220911721C},
	archiveprefix = {arXiv},
	author = {{Callis}, Keagan G.},
	doi = {10.48550/arXiv.2209.11721},
	eid = {arXiv:2209.11721},
	eprint = {2209.11721},
	journal = {arXiv e-prints},
	month = sep,
	pages = {arXiv:2209.11721},
	primaryclass = {math.DS},
	title = {{Absolutely Periodic Billiard Orbits of Arbitrarily High Order}},
	year = 2022}

@article{Vig24,
	author = {Vig, Amir},
	journal = {arXiv preprint arXiv:2310.05426},
	title = {Compactness of Marked Length Isospectral Sets of Birkhoff Billiard Tables},
	year = {2023}}

@book{SiburgPrincipleofleastaction,
	author = {Siburg, Karl Friedrich},
	number = {1844},
	publisher = {Springer Science \& Business Media},
	title = {The principle of least action in geometry and dynamics},
	year = {2004}}

@article{Vig23,
	author = {Vig, Amir},
	journal = {Journal of Spectral Theory},
	number = {3},
	pages = {877--938},
	title = {The wave trace and Birkhoff billiards},
	volume = {12},
	year = {2023}}

@inproceedings{Zel0res,
	address = {New York, NY},
	author = {Zelditch, Steve},
	booktitle = {Geometric Methods in Inverse Problems and PDE Control},
	editor = {Croke, Christopher B. and Vogelius, Michael S. and Uhlmann, Gunther and Lasiecka, Irena},
	isbn = {978-1-4684-9375-7},
	pages = {289--321},
	publisher = {Springer New York},
	title = {Inverse Resonance Problem for $\mathbb{Z}_2$-Symmetric Analytic Obstacles in the Plane},
	year = {2004}}

@article{BB1,
title = {Distribution of eigenfrequencies for the wave equation in a finite domain: I. Three-dimensional problem with smooth boundary surface},
journal = {Annals of Physics},
volume = {60},
number = {2},
pages = {401-447},
year = {1970},
issn = {0003-4916},
doi = {https://doi.org/10.1016/0003-4916(70)90497-5},
url = {https://www.sciencedirect.com/science/article/pii/0003491670904975},
author = {R Balian and C Bloch}
}

@article{BB2,
	author = {Balian, R. and Bloch, C.},
	doi = {10.1016/0003-4916(71)90286-7},
	fjournal = {Annals of Physics},
	issn = {0003-4916},
	journal = {Ann. Physics},
	mrclass = {35.80 (78.00)},
	mrnumber = {284729},
	mrreviewer = {Colin W. Clark},
	pages = {271--307},
	title = {Distribution of eigenfrequencies for the wave equation in a finite domain. {II}. {E}lectromagnetic field. {R}iemannian spaces},
	url = {https://doi.org/10.1016/0003-4916(71)90286-7},
	volume = {64},
	year = {1971}}

@article{GuKa,
	author = {Guillemin, V. and Kazhdan, D.},
	doi = {10.1016/0040-9383(80)90015-4},
	fjournal = {Topology. An International Journal of Mathematics},
	issn = {0040-9383},
	journal = {Topology},
	mrclass = {58G25 (53C20)},
	mrnumber = {579579},
	mrreviewer = {H. R. Gluck},
	number = {3},
	pages = {301--312},
	title = {Some inverse spectral results for negatively curved {$2$}-manifolds},
	url = {https://doi.org/10.1016/0040-9383(80)90015-4},
	volume = {19},
	year = {1980}}

@incollection{CdVLSP,
	author = {Colin de Verdi\`ere, Y.},
	booktitle = {South {R}hone seminar on geometry, {III} ({L}yon, 1983)},
	mrclass = {58F19 (58G25)},
	mrnumber = {753865},
	mrreviewer = {Eugene Gutkin},
	pages = {122--139},
	publisher = {Hermann, Paris},
	series = {Travaux en Cours},
	title = {Sur les longueurs des trajectoires p\'{e}riodiques d'un billard},
	year = {1984}}

@article{BB3,
	author = {Balian, R. and Bloch, C.},
	doi = {10.1016/0003-4916(72)90006-1},
	fjournal = {Annals of Physics},
	issn = {0003-4916},
	journal = {Ann. Physics},
	mrclass = {35.76 (81.00)},
	mrnumber = {289962},
	mrreviewer = {Colin W. Clark},
	pages = {76--160},
	title = {Distribution of eigenfrequencies for the wave equation in a finite domain. {III}. {E}igenfrequency density oscillations},
	url = {https://doi.org/10.1016/0003-4916(72)90006-1},
	volume = {69},
	year = {1972}}

@article{BB4,
title = {Solution of the Schrödinger equation in terms of classical paths},
journal = {Annals of Physics},
volume = {85},
number = {2},
pages = {514-545},
year = {1974},
issn = {0003-4916},
doi = {https://doi.org/10.1016/0003-4916(74)90421-7},
url = {https://www.sciencedirect.com/science/article/pii/0003491674904217},
author = {R. Balian and C. Bloch}
}

@article{HeZe12,
	author = {Hezari, Hamid and Zelditch, Steve},
	fjournal = {Analysis \& PDE},
	issn = {2157-5045},
	journal = {Anal. PDE},
	mrclass = {35P05},
	mrnumber = {3022850},
	mrreviewer = {J. B. Kennedy},
	number = {5},
	pages = {1105--1132},
	title = {{$C^\infty$} spectral rigidity of the ellipse},
	url = {https://doi.org/10.2140/apde.2012.5.1105},
	volume = {5},
	year = {2012}}

@article{Sieber,
	author = {Sieber, Martin},
	doi = {10.1088/0305-4470/30/13/011},
	fjournal = {Journal of Physics. A. Mathematical and General},
	issn = {0305-4470},
	journal = {J. Phys. A},
	mrclass = {81Q20 (81Q50)},
	mrnumber = {1479789},
	mrreviewer = {Julio Espinoza Ortiz},
	number = {13},
	pages = {4563--4596},
	title = {Semiclassical transition from an elliptical to an oval billiard},
	url = {https://doi.org/10.1088/0305-4470/30/13/011},
	volume = {30},
	year = {1997}}

@article{DuGu75,
	author = {Duistermaat, Johannes J. and Guillemin, Victor W.},
	fjournal = {Inventiones Mathematicae},
	issn = {0020-9910},
	journal = {Invent. Math.},
	mrclass = {58G99 (35P05)},
	mrnumber = {0405514},
	mrreviewer = {Alan Weinstein},
	number = {1},
	pages = {39--79},
	title = {The spectrum of positive elliptic operators and periodic bicharacteristics},
	url = {https://doi.org/10.1007/BF01405172},
	volume = {29},
	year = {1975}}

@article{CdVsllgp1,
     author = {Colin de Verdi\`ere, Y.},
     title = {Spectre du laplacien et longueurs des g\'eod\'esiques p\'eriodiques. {I}},
     journal = {Compositio Mathematica},
     pages = {83--106},
     publisher = {Noordhoff International Publishing},
     volume = {27},
     number = {1},
     year = {1973},
     mrnumber = {348798},
     zbl = {0272.53034},
     language = {fr},
     url = {https://www.numdam.org/item/CM_1973__27_1_83_0/}
}

@article{CdVsllgp2,
author = {Colin de Verdière, Yves},
journal = {Compositio Mathematica},
language = {fre},
number = {2},
pages = {159-184},
publisher = {Noordhoff International Publishing},
title = {Spectre du laplacien et longueurs des géodésiques périodiques. II},
url = {http://eudml.org/doc/89186},
volume = {27},
year = {1973},
}

@article{Chazarain1974,
	author = {Chazarain, J. },
	date = {1974/03/01},
	doi = {10.1007/BF01418788},
	id = {Chazarain1974},
	isbn = {1432-1297},
	journal = {Inventiones mathematicae},
	number = {1},
	pages = {65--82},
	title = {Formule de Poisson pour les vari{\'e}t{\'e}s riemanniennes},
	url = {https://doi.org/10.1007/BF01418788},
	volume = {24},
	year = {1974}}

@ARTICLE{HKLV25,
       author = {{Henheik}, Joscha and {Kaloshin}, Vadim and {Li}, Yunzhe and {Vig}, Amir},
        title = "{Spectral rigidity of Liouville tori}",
      journal = {arXiv e-prints},
         year = 2025,
        month = nov,
          eid = {arXiv:2511.10398},
        pages = {arXiv:2511.10398},
          doi = {10.48550/arXiv.2511.10398},
archivePrefix = {arXiv},
       eprint = {2511.10398},
 primaryClass = {math.DG},
       adsurl = {https://ui.adsabs.harvard.edu/abs/2025arXiv251110398H}
}

@article{GuMe79b,
	author = {Guillemin, Victor and Melrose, Richard},
	fjournal = {Advances in Mathematics},
	issn = {0001-8708},
	journal = {Adv. in Math.},
	mrclass = {58G25 (10J25 35P20)},
	mrnumber = {539531},
	mrreviewer = {P. G\"unther},
	number = {3},
	pages = {204--232},
	title = {The {P}oisson summation formula for manifolds with boundary},
	url = {https://doi.org/10.1016/0001-8708(79)90042-2},
	volume = {32},
	year = {1979}}

@article{GuMe79a,
	author = {Guillemin, Victor and Melrose, Richard},
	fjournal = {Advances in Mathematics},
	issn = {0001-8708},
	journal = {Adv. in Math.},
	mrclass = {35P99 (58G25)},
	mrnumber = {535619},
	mrreviewer = {V. M. Babich},
	number = {2},
	pages = {128--148},
	title = {An inverse spectral result for elliptical regions in {${\bf R}^{2}$}},
	url = {https://doi.org/10.1016/0001-8708(79)90039-2},
	volume = {32},
	year = {1979}}

@article{MaMe82,
	author = {Marvizi, Shahla and Melrose, Richard},
	fjournal = {Journal of Differential Geometry},
	issn = {0022-040X},
	journal = {J. Differential Geom.},
	mrclass = {58G25 (35P05)},
	mrnumber = {679068},
	mrreviewer = {Hideo Soga},
	number = {3},
	pages = {475--502},
	title = {Spectral invariants of convex planar regions},
	url = {http://projecteuclid.org/euclid.jdg/1214437138},
	volume = {17},
	year = {1982}}

@article{KaAvDS16,
	author = {Avila, Artur and De Simoi, Jacopo and Kaloshin, Vadim},
	fjournal = {Annals of Mathematics. Second Series},
	issn = {0003-486X},
	journal = {Ann. of Math. (2)},
	mrclass = {37D50},
	mrnumber = {3548532},
	mrreviewer = {Nicolas Bedaride},
	number = {2},
	pages = {527--558},
	title = {An integrable deformation of an ellipse of small eccentricity is an ellipse},
	url = {https://doi.org/10.4007/annals.2016.184.2.5},
	volume = {184},
	year = {2016}}

@article{KaHuSo18,
	author = {Huang, Guan and Kaloshin, Vadim and Sorrentino, Alfonso},
	doi = {10.1215/00127094-2017-0038},
	fjournal = {Duke Mathematical Journal},
	journal = {Duke Math. J.},
	month = {01},
	number = {1},
	pages = {175--209},
	publisher = {Duke University Press},
	title = {On the marked length spectrum of generic strictly convex billiard tables},
	url = {https://doi.org/10.1215/00127094-2017-0038},
	volume = {167},
	year = {2018}}

@article{KaSo16,
	author = {Kaloshin, Vadim and Sorrentino, Alfonso},
	doi = {10.4007/annals.2018.188.1.6},
	fjournal = {Annals of Mathematics. Second Series},
	issn = {0003-486X},
	journal = {Ann. of Math. (2)},
	mrclass = {37D50 (33E05 37E40 37J35 70H06)},
	mrnumber = {3815464},
	number = {1},
	pages = {315--380},
	title = {On the local {B}irkhoff conjecture for convex billiards},
	url = {https://doi.org/10.4007/annals.2018.188.1.6},
	volume = {188},
	year = {2018}}

@article{KaDSWe17,
	author = {De Simoi, Jacopo and Kaloshin, Vadim and Wei, Qiaoling},
	fjournal = {Annals of Mathematics. Second Series},
	issn = {0003-486X},
	journal = {Ann. of Math. (2)},
	mrclass = {37D50 (35J05 35P05 35R30 58J53)},
	mrnumber = {3665005},
	note = {Appendix B coauthored with H. Hezari},
	number = {1},
	pages = {277--314},
	title = {Dynamical spectral rigidity among {$\mathbb{Z}_2$}-symmetric strictly convex domains close to a circle (Appendix B coauthored with H. Hezari)},
	url = {https://doi.org/10.4007/annals.2017.186.1.7},
	volume = {186},
	year = {2017}}

@article{Zel09,
	author = {Zelditch, Steve},
	doi = {10.4007/annals.2009.170.205},
	fjournal = {Annals of Mathematics. Second Series},
	issn = {0003-486X},
	journal = {Ann. of Math. (2)},
	mrclass = {58J53 (35P20 35R30 37D50)},
	mrnumber = {2521115},
	mrreviewer = {David Borthwick},
	number = {1},
	pages = {205--269},
	title = {Inverse spectral problem for analytic domains. {II}. {$\mathbb{Z}_2$}-symmetric domains},
	url = {https://doi.org/10.4007/annals.2009.170.205},
	volume = {170},
	year = {2009}}

@article{ZelditchSurvey2014,
	author = {Zelditch, Steve},
	doi = {10.4310/ICCM.2014.v2.n2.a1},
	fjournal = {ICCM Notices. Notices of the International Congress of Chinese Mathematicians},
	issn = {2326-4810},
	journal = {ICCM Not.},
	mrclass = {58J50 (58J53)},
	mrnumber = {3314780},
	mrreviewer = {Ruth Gornet},
	number = {2},
	pages = {1--20},
	title = {Survey on the inverse spectral problem},
	url = {https://doi.org/10.4310/ICCM.2014.v2.n2.a1},
	volume = {2},
	year = {2014}}

@article{POS1,
	author = {Osgood, Brad and Phillips, Ralph and Sarnak, Peter},
	doi = {10.1016/0022-1236(88)90071-7},
	fjournal = {Journal of Functional Analysis},
	issn = {0022-1236},
	journal = {J. Funct. Anal.},
	mrclass = {58G25 (11F72 35P20 47F05 58G11 81E30)},
	mrnumber = {960229},
	mrreviewer = {Andreas Juhl},
	number = {1},
	pages = {212--234},
	title = {Compact isospectral sets of surfaces},
	url = {https://doi.org/10.1016/0022-1236(88)90071-7},
	volume = {80},
	year = {1988}}

@article{POS2,
	author = {Osgood, Brad and Phillips, Ralph and Sarnak, Peter},
	doi = {10.1073/pnas.85.15.5359},
	fjournal = {Proceedings of the National Academy of Sciences of the United States of America},
	issn = {0027-8424},
	journal = {Proc. Nat. Acad. Sci. U.S.A.},
	mrclass = {58G25},
	mrnumber = {952815},
	mrreviewer = {Andreas Juhl},
	number = {15},
	pages = {5359--5361},
	title = {Compact isospectral sets of plane domains},
	url = {https://doi.org/10.1073/pnas.85.15.5359},
	volume = {85},
	year = {1988}}

@article{GuillarmouLefeuvre,
	adsurl = {http://adsabs.harvard.edu/abs/2018arXiv180604218G},
	archiveprefix = {arXiv},
	author = {{Guillarmou}, Colin and {Lefeuvre}, Thibault},
	eprint = {1806.04218},
	journal = {ArXiv e-prints},
	month = jun,
	primaryclass = {math.DG},
	title = {{The marked length spectrum of Anosov manifolds}},
	year = 2018}

@inproceedings{MelroseIsospectralCompactness,
	author = {Richard B. Melrose},
	title = {Isospectral sets of drumheads are compact in $C^\infty$},
	year = {2007}}

@article{Popov1994,
	author = {Popov, Georgi},
	fjournal = {Communications in Mathematical Physics},
	journal = {Comm. Math. Phys.},
	number = {2},
	pages = {335--364},
	publisher = {Springer},
	title = {Invariants of the length spectrum and spectral invariants of planar convex domains},
	url = {https://projecteuclid.org:443/euclid.cmp/1104269906},
	volume = {161},
	year = {1994}}

@article{PopovTopalov,
	author = {Popov, Georgi and Topalov, Petar},
	doi = {10.1007/s00220-011-1223-z},
	fjournal = {Communications in Mathematical Physics},
	issn = {0010-3616},
	journal = {Comm. Math. Phys.},
	mrclass = {37J35 (35P05 37D50 58J50)},
	mrnumber = {2786215},
	mrreviewer = {Yuri B. Suris},
	number = {3},
	pages = {721--759},
	title = {On the integral geometry of {L}iouville billiard tables},
	url = {https://doi.org/10.1007/s00220-011-1223-z},
	volume = {303},
	year = {2011}}

@article{Sel56,
	author = {Selberg, Atle},
	journal = {J. Indian Math. Soc. (N.S.)},
	mrclass = {10.1X},
	mrnumber = {0088511},
	mrreviewer = {F. V. Atkinson},
	pages = {47--87},
	title = {Harmonic analysis and discontinuous groups in weakly symmetric {R}iemannian spaces with applications to {D}irichlet series},
	volume = {20},
	year = {1956}}

@article{BuKa85,
	author = {Burns, Keith and Katok, Anatole},
	doi = {10.1017/S0143385700002935},
	fjournal = {Ergodic Theory and Dynamical Systems},
	issn = {0143-3857},
	journal = {Ergodic Theory Dynam. Systems},
	mrclass = {53C20 (53-02 58F15)},
	mrnumber = {796758},
	mrreviewer = {Wolfgang Ziller},
	number = {2},
	pages = {307--317},
	title = {Manifolds with nonpositive curvature},
	url = {https://doi.org/10.1017/S0143385700002935},
	volume = {5},
	year = {1985}}

@article{Lazutkin,
	author = {Lazutkin, V. F.},
	fjournal = {Izvestiya Akademii Nauk SSSR. Seriya Matematicheskaya},
	issn = {0373-2436},
	journal = {Izv. Akad. Nauk SSSR Ser. Mat.},
	mrclass = {34C35 (58F99)},
	mrnumber = {0328219},
	mrreviewer = {C. Olech},
	pages = {186--216},
	title = {Existence of caustics for the billiard problem in a convex domain},
	volume = {37},
	year = {1973}}

@article{Vig18,
	author = {Vig, Amir},
	doi = {10.1007/s12220-019-00339-4},
	fjournal = {Journal of Geometric Analysis},
	issn = {1050-6926},
	journal = {J. Geom. Anal.},
	mrclass = {58J50 (31A30 35P05 46N20)},
	mrnumber = {4225807},
	mrreviewer = {Tanya J. Christiansen},
	number = {3},
	pages = {2238--2295},
	title = {Robin spectral rigidity of the ellipse},
	url = {https://doi.org/10.1007/s12220-019-00339-4},
	volume = {31},
	year = {2021}}

@article{Sorr15,
	author = {Sorrentino, Alfonso},
	doi = {10.3934/dcds.2015.35.5055},
	fjournal = {Discrete and Continuous Dynamical Systems. Series A},
	issn = {1078-0947},
	journal = {Discrete Contin. Dyn. Syst.},
	mrclass = {37D50 (37E40 37J50)},
	mrnumber = {3392661},
	number = {10},
	pages = {5055--5082},
	title = {Computing {M}ather's {$\beta$}-function for {B}irkhoff billiards},
	url = {https://doi.org/10.3934/dcds.2015.35.5055},
	volume = {35},
	year = {2015}}

@book{PeSt17,
	author = {Petkov, Vesselin M. and Stoyanov, Luchezar N.},
	doi = {10.1002/9781119107682},
	edition = {Second},
	isbn = {978-1-119-10766-8},
	mrclass = {58J50 (35P25 35R30 37D40 53C65 53D25 81T20 81V10)},
	mrnumber = {3617212},
	mrreviewer = {Boris Hasselblatt},
	pages = {xv+410},
	publisher = {John Wiley \& Sons, Ltd., Chichester},
	title = {Geometry of the generalized geodesic flow and inverse spectral problems},
	url = {https://doi.org/10.1002/9781119107682},
	year = {2017}}

@article{Zelditch1,
	author = {Zelditch, S.},
	doi = {10.1007/PL00001633},
	fjournal = {Geometric and Functional Analysis},
	issn = {1016-443X},
	journal = {Geom. Funct. Anal.},
	mrclass = {58J50 (35J05 35P99 35R30 58J37)},
	mrnumber = {1779616},
	mrreviewer = {Georgi Popov},
	number = {3},
	pages = {628--677},
	title = {Spectral determination of analytic bi-axisymmetric plane domains},
	url = {https://doi.org/10.1007/PL00001633},
	volume = {10},
	year = {2000}}

@article{Zelditch2,
	author = {Zelditch, Steve},
	fjournal = {Journal of Differential Geometry},
	issn = {0022-040X},
	journal = {J. Differential Geom.},
	mrclass = {58G25 (53C22 58D27)},
	mrnumber = {1664907},
	mrreviewer = {Edoh Amiran},
	number = {2},
	pages = {207--264},
	title = {The inverse spectral problem for surfaces of revolution},
	url = {http://projecteuclid.org/euclid.jdg/1214461019},
	volume = {49},
	year = {1998}}

@article{Zelditch3,
	author = {Zelditch, Steve},
	doi = {10.1007/s00220-004-1074-y},
	fjournal = {Communications in Mathematical Physics},
	issn = {0010-3616},
	journal = {Comm. Math. Phys.},
	mrclass = {58J53 (35P20 35R30)},
	mrnumber = {2073139},
	mrreviewer = {David Borthwick},
	number = {2},
	pages = {357--407},
	title = {Inverse spectral problem for analytic domains. {I}. {B}alian-{B}loch trace formula},
	url = {https://doi.org/10.1007/s00220-004-1074-y},
	volume = {248},
	year = {2004}}

@article{Zelditch5,
	author = {Hezari, Hamid and Zelditch, Steve},
	doi = {10.1007/s00039-010-0059-6},
	fjournal = {Geometric and Functional Analysis},
	issn = {1016-443X},
	journal = {Geom. Funct. Anal.},
	mrclass = {58J53 (35P05 37D50)},
	mrnumber = {2647138},
	mrreviewer = {Andrew W. Hassell},
	number = {1},
	pages = {160--191},
	title = {Inverse spectral problem for analytic {$(\mathbb{Z}/2\mathbb{Z})^n$}-symmetric domains in {$\mathbb{R}^n$}},
	url = {https://doi.org/10.1007/s00039-010-0059-6},
	volume = {20},
	year = {2010}}

@article{AndersonMelrose,
	author = {Andersson, K. G. and Melrose, R. B.},
	doi = {10.1007/BF01403048},
	fjournal = {Inventiones Mathematicae},
	issn = {0020-9910},
	journal = {Invent. Math.},
	mrclass = {58G15 (35S15)},
	mrnumber = {0494322},
	mrreviewer = {Bent E. Petersen},
	number = {3},
	pages = {197--232},
	title = {The propagation of singularities along gliding rays},
	url = {https://doi.org/10.1007/BF01403048},
	volume = {41},
	year = {1977}}

@book{KozTresh89,
	author = {Kozlov, Valeri\u{\i} V. and Treshch\"{e}v, Dmitri\u{\i} V.},
	isbn = {0-8218-4550-0},
	mrclass = {58F05 (34C35 58F11 58F17 70D99 70F35 70H05)},
	mrnumber = {1118378},
	mrreviewer = {Donal Hurley},
	note = {A genetic introduction to the dynamics of systems with impacts, Translated from the Russian by J. R. Schulenberger},
	pages = {viii+171},
	publisher = {American Mathematical Society, Providence, RI},
	series = {Translations of Mathematical Monographs},
	title = {Billiards},
	volume = {89},
	year = {1991}}

@article{PopTop12,
	author = {Popov, Georgi and Topalov, Peter},
	doi = {10.1080/03605302.2011.641051},
	fjournal = {Communications in Partial Differential Equations},
	issn = {0360-5302},
	journal = {Comm. Partial Differential Equations},
	mrclass = {47A10 (35P05 35R01 37J40 58J53)},
	mrnumber = {2889558},
	mrreviewer = {Yuri A. Kordyukov},
	number = {3},
	pages = {369--446},
	title = {Invariants of isospectral deformations and spectral rigidity},
	url = {https://doi.org/10.1080/03605302.2011.641051},
	volume = {37},
	year = {2012}}

@book{Lazutkin5,
	author = {Lazutkin, Vladimir F.},
	doi = {10.1007/978-3-642-76247-5},
	isbn = {3-540-53389-3},
	mrclass = {58Fxx (35J05 35P05 58F19 58F27 70H05 81Q20)},
	mrnumber = {1239173},
	mrreviewer = {Helmut R\"{u}ssmann},
	note = {With an addendum by A. I. Shnirelman},
	pages = {x+387},
	publisher = {Springer-Verlag, Berlin},
	series = {Ergebnisse der Mathematik und ihrer Grenzgebiete (3) [Results in Mathematics and Related Areas (3)]},
	title = {K{AM} theory and semiclassical approximations to eigenfunctions},
	url = {https://doi.org/10.1007/978-3-642-76247-5},
	volume = {24},
	year = {1993}}

@article{HeZe19,
	author = {Hezari, Hamid and Zelditch, Steve},
	doi = {10.4007/annals.2022.196.3.4},
	fjournal = {Annals of Mathematics. Second Series},
	issn = {0003-486X},
	journal = {Ann. of Math. (2)},
	mrclass = {58C40 (35P05)},
	mrnumber = {4502596},
	number = {3},
	pages = {1083--1134},
	title = {One can hear the shape of ellipses of small eccentricity},
	url = {https://doi.org/10.4007/annals.2022.196.3.4},
	volume = {196},
	year = {2022}}

@article{Gutzwiller4,
	author = {Gutzwiller,Martin C.},
	doi = {10.1063/1.1665596},
	eprint = {https://doi.org/10.1063/1.1665596},
	journal = {Journal of Mathematical Physics},
	number = {3},
	pages = {343-358},
	title = {Periodic Orbits and Classical Quantization Conditions},
	url = {https://doi.org/10.1063/1.1665596},
	volume = {12},
	year = {1971}}
	
\end{document}